\documentclass[11pt]{amsart}
\usepackage{amssymb, amsmath,latexsym,amsfonts,amsbsy, amsthm,mathtools,graphicx,color}
\usepackage{float}
\usepackage{hyperref}

\numberwithin{equation}{section}
\allowdisplaybreaks

\usepackage{mathrsfs,enumerate,extarrows}
\let\al=\alpha

\let\g=\gamma

\let\la=\lambda

\let\La=\Lambda

\let\Om=\Omega
\let\wt=\widetilde

\let\pa=\partial

\def\C{\mathbb C}
\def\R{\mathbb R}

\def\Z{\mathbb Z}

\def\ii{\text{i}}
\def\ee{\text{e}}

\def\LL{\mathscr L}

\def\dZ{\mathrm{d}Z}

\newcommand{\beq}{\begin{equation}}
	\newcommand{\eeq}{\end{equation}}
\newcommand{\ben}{\begin{eqnarray}}
	\newcommand{\een}{\end{eqnarray}}
\newcommand{\beno}{\begin{eqnarray*}}
	\newcommand{\eeno}{\end{eqnarray*}}

\renewcommand{\Re}{\operatorname{Re}}
\renewcommand{\Im}{\operatorname{Im}}
\renewcommand{\tilde}{\widetilde}
\renewcommand{\hat}{\widehat}

\newtheorem{assumption}{Assumption}
\newtheorem{theorem}{Theorem}[section]

\newtheorem{lemma}[theorem]{Lemma}
\newtheorem{proposition}[theorem]{Proposition}
\newtheorem{corollary}[theorem]{Corollary}
\theoremstyle{remark}
\newtheorem{remark}[theorem]{Remark}

\begin{document}

\title[Blow-up for the supercritical defocusing nonlinear wave equation]
{On blow-up for the supercritical defocusing nonlinear wave equation}

\author[F. Shao]{Feng Shao}
\address{School of Mathematical Sciences, Peking University, Beijing 100871,  China}
\email{2101110016@stu.pku.edu.cn}

\author[D. Wei]{Dongyi Wei}
\address{School of Mathematical Sciences, Peking University, Beijing 100871,  China}
\email{jnwdyi@pku.edu.cn}

\author[Z. Zhang]{Zhifei Zhang}
\address{School of Mathematical Sciences, Peking University, Beijing 100871, China}
\email{zfzhang@math.pku.edu.cn}

\date{\today}

\begin{abstract}
In this paper, we consider the defocusing nonlinear wave equation $-\partial_t^2u+\Delta u=|u|^{p-1}u$ in $\R\times \R^d$. Building on our companion work ({\it \small Self-similar imploding solutions of the relativistic Euler equations}), we prove that for $d=4, p\geq 29$ and $d\geq 5, p\geq 17$, there exists a smooth complex-valued solution that blows up in finite time.
\end{abstract}

\maketitle
\tableofcontents

\section{Introduction}

In this paper, we consider the defocusing nonlinear wave equation
\begin{align}
	\label{Eq.wave_eq}\Box u=|u|^{p-1}u,
\end{align}
where  $u:\R^{1+d}\to\C$ is the unknown field, $ \Box=\partial^{\alpha}\partial_{\alpha}=-\partial_t^2+\sum_{i=1}^d\partial_i^2$ is the d'Alembertian operator\footnote{Here we use the Einstein's summation convention.} on Minkowski spacetime $\R^{1+d}$ with the standard Minkowski metric
\[m_{00}=-1,\quad m_{ii}=1 \text{ for all }i\in\Z\cap[1,d],\quad m_{\mu\nu}=0\text{ if }\mu,\nu\in\Z\cap[0, d]\text{ with }\mu\neq \nu,\]
and we assume $p\in 2\Z_{+}+1$ for simplicity. 

Given smooth initial data $(u|_{t=0}, \pa_tu|_{t=0})$, there exists a local smooth solution on the maximal existence of interval $[0, T)$; $T<+\infty$ if and only if $\limsup_{t\uparrow T}\|u(t)\|_{L^\infty}=+\infty$, see \cite{Sogge, Luk}; moreover, there holds  the energy  conservation
\begin{equation}
	E[u(t)]:=\int_{\R^d}\frac12|\pa_tu|^2+\frac12|\nabla_xu|^2+\frac1{p+1}|u|^{p+1}\,\mathrm dx.
\end{equation}

The class of solutions to \eqref{Eq.wave_eq} is invariant under the scaling
\begin{equation}
	u(t,x)\mapsto u_\la(t,x):=\la^{\frac2{p-1}}u(\la t,\la x),\quad\la>0.
\end{equation}
This scaling symmetry preserves the {critical} norm invariant, i.e., 
\[\|u_\la(t, \cdot)\|_{\dot H_x^{s_c}}=\|u(\la t, \cdot)\|_{\dot H_x^{s_c}}\quad\text{where}\quad s_c:=\frac d2-\frac2{p-1}.\]
We can split the range of parameters $(d,p)$ into three cases accordingly:
\begin{itemize}
	\item \emph{Subcritical} case: $s_c<1$ $\Longleftrightarrow$ $d\leq 2$ or $p<1+4/(d-2)$ for $d\geq 3$.
	\item \emph{Critical} case: $s_c=1$ $\Longleftrightarrow$ $p=1+4/(d-2)$ and $d\geq 3$.
	\item \emph{Supercritical} case: $s_c>1$ $\Longleftrightarrow$ $p>1+4/(d-2)$ and $d\geq 3$.
\end{itemize}

For the subcritical case, the global well-posedness and propagation of regularity dated back to J\"{o}rgens \cite{Jorgens} for $d=3$; see also \cite{Ginibre-Velo1,Ginibre-Velo2} for the global well-posedness within the energy class $H^1\times L^2$ for all dimensions; the propagation of regularity holds at least for $d\leq 9$ \cite{Brenner-Wahl}. The critical case is much {more} difficult. The global regularity result was obtained firstly in \cite{Struwe1988} for $d=3$ and spherically symmetric initial data, and  then extended to $d\leq 9$ for general smooth data in \cite{Grillakis1990,Grillakis1992,Shatah-Struwe1993}, and all dimensions in \cite{Shatah-Struwe1994} {(in the energy class $H^1\times L^2$)}. For the long-time behavior of these global solutions, we refer to \cite{Yang} and references therein. 

For the supercritical case, it is known that the Cauchy problem is ill-posed in some low regularity spaces \cite{Christ-Colliander-Tao}, or even in the energy class \cite{Ibrahim-Majdoub-Masmoudi}, despite the global existence of weak solutions \cite{Strauss}, as well as the global well-posedness with scattering for small smooth data \cite{Lindblad-Sogge}. The global well-posedness for general smooth data  is a long-standing open problem \cite{Bourgain2000,Tao2007}. In the breakthrough work \cite{Merle-4-Invent},   Merle, Raph\"el, Rodnianski and Szeftel construct radial and asymptotically self-similar blow-up solutions for the energy supercritical defocusing nonlinear Schr\"odinger equations (NLS).  
The goal of this paper is to extend similar blow-up result for NLS to the defocusing supercritical wave equation.

Before stating our theorem, we recall Tao's blow-up result \cite{Tao2016} for the defocusing nonlinear wave system of the form $\Box u=(\nabla_{\R^m}F)(u)$, where $u:\R^{1+d}\to\R^m$ is vector-valued, and $F:\R^m\to \R$ is a smooth potential which is positive and homogeneous of order $p+1$ outside of the unit ball for some $p>1$ (letting $m=2$ and $F(u)=|u|^{p+1}/(p+1)$ we recover \eqref{Eq.wave_eq}). Tao \cite{Tao2016} proved that for any supercritical $(d,p)$, and sufficiently large positive integer $m$, there exists a defocusing $F: \R^m\to \R$ such that the system $\Box u=(\nabla_{\R^m}F)(u)$ has no global smooth solution for some smooth compactly supported initial data. A similar result for the defocusing Schr\"{o}dinger system was obtained in \cite{Tao2018}.

\subsection{Main results}

Roughly speaking, we prove that the defocusing supercritical nonlinear complex-valued wave equation for $d\ge 4$
admits finite time blow-up solutions arising from smooth initial data. The leading order term of blow-up solution is given by a self-similar blow-up solution of the relativistic compressible Euler equation, which is stated here as Assumption \ref{Assumption} (in Section \ref{Sec.Proof}). In our companion paper \cite{Shao-Wei-Zhang}, we have verified Assumption \ref{Assumption} for some $(d,p)$.

\begin{theorem}\label{Thm.main_thm}
	Let $d\in\Z\cap[4,+\infty)$ and $p\in 2\Z_{+}+1$ be such that\footnote{In particular, we have $k>\ell$, which is equivalent to $p>1+4/(d-2)$. So we are in the \emph{supercritical} case. Nevertheless, we can not cover the whole supercritical range using the method of current paper.} $k>\ell +\sqrt\ell$, where $k:=d-1$ and $\ell:=1+4/(p-1)$. Assume that there exists $\beta\in(1, k/(\ell+\sqrt\ell))$ such that Assumption \ref{Assumption} holds. Then there exist compactly supported smooth functions $u_0, u_1:\R^d\to \R^2(=\C)$ such that there is no global smooth solution $u: [0,+\infty)\times \R^d\to \R^2(=\C)$ to the defocusing nonlinear wave equation \eqref{Eq.wave_eq} with initial data $u(0)=u_0$, $\pa_t u(0)=u_1$.
\end{theorem}

\begin{corollary}\label{Cor.main_cor}
	If $d=4$, $p\in (2\Z+1)\cap[29,+\infty)$ or $d\geq5, p\in (2\Z+1)\cap[17,+\infty)$, then there exist compactly supported smooth functions $u_0, u_1:\R^d\to \R^2(=\C)$ such that there is no global smooth solution $u: [0,+\infty)\times \R^d\to \R^2(=\C)$ to the defocusing nonlinear wave equation \eqref{Eq.wave_eq} with initial data $u(0)=u_0$, $\pa_t u(0)=u_1$.
\end{corollary}

Several remarks are in order.

\begin{enumerate}
	\renewcommand{\labelenumi}{\theenumi.}
	\item For the blow-up solution $u$ we construct in Theorem \ref{Thm.main_thm}, if $u$ blows up at time $T_*\in (0, +\infty)$, then according to our construction, we have the blow-up speed
	\begin{align*}
		\|u(t,\cdot)\|_{L^\infty}\gtrsim(T_*-t)^{-\frac{2\beta}{p-1}},\quad \|(u(t),\pa_tu(t))\|_{\dot{H}_x^{s_c}\times\dot{H}_x^{s_c-1}}\gtrsim (T_*-t)^{(1-\beta)\frac d2}.
	\end{align*}
	As $\beta>1$, our solution is unbounded in the critical space. This is compatible with the results in literatures which state that the solutions for the supercritical defocusing wave equation that are bounded in the critical space $\dot{H}_x^{s_c}\times\dot{H}_x^{s_c-1}$ must be global and scattering (at least for real-valued solutions and some supercritical $(d,p)$, see \cite{Bulut2012,Bulut2015,Duyckaerts-Yang2018,Killp-Visan2011Trans,Killp-Visan2011Pro}).
	
	\item Following the recent breakthrough work by Merle-Rapha\"{e}l-Rodnianski-Szeftel \cite{Merle-4-Invent, Merle-4-Ann1, Merle-4-Ann2}, the heart of proof of Theorem \ref{Thm.main_thm} is to study \eqref{Eq.wave_eq} in its hydrodynamical formulation, i.e., with respect to its phase and modulus variables. After introducing a front re-normalization \eqref{Eq.front}, we can view the second-order differential operator acting on the modulus as a perturbation. In this way, we reveal the underlying relativistic compressible Euler dynamics. The relativistic Euler dynamics provides us with a self-similar blow-up solution, which has been constructed in our companion paper \cite{Shao-Wei-Zhang} and which, in turn, acts as the leading order term of the blow-up solution of the defocusing supercritical wave equation \eqref{Eq.wave_eq}.
	
	\item Unlike \cite{Merle-4-Invent}, where the initial data for blow-up form a finite co-dimensional manifold in the class of radial smooth fast-decay functions, we only construct the blow-up solution  for one initial data $(u_0, u_1)$ in Theorem \ref{Thm.main_thm}. As a result, we do not need to analyze the stability of the linearized operator near the leading order profile constructed in \cite{Shao-Wei-Zhang}. This simplifies our proof greatly. We believe that the blow-up should hold for a large class of initial data, just as in \cite{Merle-4-Invent}. This is left  to the future work.
	
	\item To prove Corollary \ref{Cor.main_cor}, we just need to verify Assumption \ref{Assumption}, which is related to the existence of a smooth global solution to a specific ODE \eqref{Eq.ODE_v(Z)}. If $d=4, p\geq 29$ or $d=5, p\geq 17$, Assumption \ref{Assumption} is verified in our companion paper \cite{Shao-Wei-Zhang}. As a consequence, if one can find some other methods to verify Assumption \ref{Assumption} for smaller $p$, then one can also get the blow-up for that smaller $p$. The case $d>5$ follows from the result for $d=5$ and truncation, see Subsection \ref{Subsec.Proof_main}.
	
	\item We emphasize that if Assumption \ref{Assumption} is valid, then we must have $d>\beta(\ell+\sqrt\ell)+1$, where $\ell:=1+4/(p-1)>1$. Using $\beta>1$, we get $d>3$. As a result, the case of $d=3$ is not amenable to our analysis at present, and the existence of blow-up solutions for $d=3$ remains open. We point out that similar situation happens in \cite{Merle-4-Invent}, where the construction fails for $3$-D and $4$-D defocusing supercritical NLS.
	
\item In this work, we can only construct the blow-up for the complex-valued solution. The blow-up for the scalar defocusing supercritical wave equation remains open at this point. We guess that the same blow-up result should hold for the scalar nonlinear wave equation, at least for $(d,p)$ satisfying the same hypothesis as in Theorem \ref{Thm.main_thm}.
\end{enumerate}

The roadmap of the proof  of  Theorem \ref{Thm.main_thm} and Corollary \ref{Cor.main_cor} can be found in Section \ref{Sec.Proof}. The proof is based on Propositions \ref{Prop.L_surjective}, \ref{Prop.Approximate} and \ref{Prop.Solve_wave}. Our starting point is to introduce a front re-normalization \eqref{Eq.front}, relying on a constant $b>0$; taking the limit $b\to0$, 
{the defocusing wave equation becomes the relativistic compressible Euler equations}. 

We first write the desired solution to \eqref{Eq.wave_eq} in the form of a power series (see \eqref{Eq.rho_phi_expansion}) with respect to the constant $b>0$. The non-degeneracy of the leading order approximation allows us to solve all high-order approximations $(\rho_n, \phi_n)$, which is exactly the purpose of Proposition \ref{Prop.L_surjective}. The proof of Proposition  \ref{Prop.L_surjective} is rather technical and can be found in Section \ref{Sec.Surjective}. One of key ingredients used is the existence of smooth solutions to the second order ODEs having singular points with a parameter $\la$, see Appendix \ref{Appen.ODE}.

Since we do not have enough information on $(\rho_n, \phi_n)$, especially {the estimate uniform in $n$}, 
we may not have the convergence of the formal series \eqref{Eq.rho_phi_expansion}. To overcome this drawback, we truncate $(\rho_n, \phi_n)$ in the form of \eqref{Eq.rho_*_phi_*}, and in Proposition \ref{Prop.Approximate} we prove that the truncated solution is a good approximate solution to the defocusing wave equation. The proof of Proposition \ref{Prop.Approximate} can be found in Section \ref{Sec.Approximate}.

Finally, we  construct a solution to \eqref{Eq.wave_eq} near the truncated approximation solution. This is exactly what Proposition \ref{Prop.Solve_wave} says. The proof of Proposition \ref{Prop.Solve_wave} can be found in Section \ref{Sec.Sol_wave}, 
where we use the energy method to solve the wave equation in a time-backward direction, {and we need to use a technical
truncation} to avoid the singularity at blow-up time. {Such method of solving backward in time has been used in \cite{Krieger-Schlag-Tataru2008,Krieger-Schlag-Tataru2009,Perelman2014}}.
Let's emphasize that this part does not depend at all on our method of constructing the approximate solutions, and it includes the case $d=3$ and does not require Assumption \ref{Assumption} or the spherical symmetry of the approximate solutions either.

\subsection{Blow-up phenomenon for related models}

Let's review some important results on the blow-up for other related equations.

It is more common to observe the blow-up phenomenon for the focusing  nonlinear wave equation, i.e., 
\begin{equation}\label{Eq.focusingNLW}
	\Box u=-|u|^{p-1}u.
\end{equation}
In fact, the spatial independent function $u(t)=C_p(T-t)^{-2/(p-1)}$, where $C_p^{p-1}=2(p+1)/(p-1)^2$, gives a blow-up solution to \eqref{Eq.focusingNLW}. This ODE-type solution can be further truncated to a smooth compactly supported blow-up solution to \eqref{Eq.focusingNLW} by using the finite speed of propagation \cite{Alinhac,John,Levine}. {We will use similar ideas to prove Corollary \ref{Cor.main_cor} for the case $d>5$.} See also \cite{Donninger2017, Duyckaerts-Kenig-Merle2012, Duyckaerts-Kenig-Merle2013, Duyckaerts-Yang2018, Jendrej2017, Kenig2015, Kenig-Merle2008, Krieger-Schlag2014, Krieger-Schlag-Tataru2009, Martel-Merle2018, Merle-Zaag2003} for the construction and classification of blow-up (or global) solutions as well as recent breakthrough \cite{Duyckaerts-Jia-Kenig-Merle2017, Duyckaerts-Kenig-Merle2023,Jendrej-Lawrie2023} on the soliton resolution conjecture.

Other related models such as nonlinear Schr\"{o}dinger equation, see \cite{Kenig-Merle2006, Merle-Raphael2006, Merle-Raphael-Rodnianski2015, Merle-4-Invent,Merle-Raphael-Szeftel-2010, Merle-Raphael-Szeftel-2014, Perelman2001, Perelman2023}; see \cite{Duyckaerts-Jia-Kenig-Merle2018, Krieger-Miao2020, Krieger-Miao-Schlag2020, Krieger-Schlag-Tataru2008, Raphael-Rodnianski2021, Rodnianski-Sterbenz2010} for the wave map;  see \cite{Merle-Raphael-Rodnianski2013,Perelman2014} for the Shr\"{o}dinger maps; see \cite{Collot-Raphael-Szeftel2019, Collot-Raphael-Szeftel2020, Cortazar-3-2020, del-Pino-3-2020, Harada2020, Matano-Merle2004, Merle-3-2020, Merle-Zaag1997}  for semilinear heat equation and  \cite{Davila-del-Wei2020, Jendrej-Lawrie2023CVPDE, Kim-Merle2024, Raphael-Schweyer2013} for the harmonic heat flow.

\subsection{Notations and conventions}\label{Subsec.notation}
Unless stated otherwise, we adopt the following notations, abbreviations and conventions:
\begin{itemize}
	\item Constants: $\mathrm i=\sqrt{-1}$ is the imaginary unit, $\mathrm e$ is the base of the natural logarithm.
	\item For any $a\in\R$, we denote $\Z_{\geq a}:=\Z\cap[a, +\infty)$ and $\Z_{>a}:=\Z\cap(a, +\infty)$. Moreover, we denote $\Z_+:=\Z_{\geq 1}$. Similarly, $\R_{\geq0}:=\R\cap[0, +\infty)$.
	\item Greek indices run from $0$ to $d$, where $d\in\Z_{\geq 2}$ is the spatial dimension, Latin indices run from $1$ to $d$, and we use the Einstein's summation convention: repeated indices appearing once upstairs and once downstairs are summed over their range.
	\item $(t,x)=(t, x_1, \cdots, x_d)$ denotes coordinates in spacetime, $r=|x|=(\sum_{j=1}^dx_j^2)^{1/2}$. We write $\pa_{0}=-\pa^0=\pa_t=\frac{\partial}{\pa t}$, $\pa_{j}=\pa^j=\pa_{x_j}=\frac{\partial}{\pa x_j}$ for $j\in\Z\cap[1,d]$, $\Box=\pa^\al\pa_\al=-\pa_t^2+\sum_{j=1}^d\pa_j^2$ and $\Delta=\sum_{j=1}^d\pa_j^2$, then $\Box=-\pa_t^2+\Delta$.
	\item We denote $\ell:=1+4/(p-1)>1$, $k:=d-1\in\Z_{+}$ and $\g:=4\beta/(p-1)+2=\beta(\ell-1)+2$. 
	\item For a (vector-valued) differentiable function $f=f(t,x)$, we denote 
	\begin{align*}
		Df:=(\pa_tf, \pa_1f, \pa_2f,\cdots, \pa_df) &\quad\text{and}\quad D_xf:=(\pa_1f, \pa_2f,\cdots, \pa_df)=\nabla_x f,
	\end{align*}and $|Df|:=(|\pa_tf|^2+\sum_{j=1}^d|\pa_jf|^2)^{1/2}$, $|D_xf|:=(\sum_{j=1}^d|\pa_jf|^2)^{1/2}$.
	For all $j\in\Z_{+}$ we denote $D^jf:=DD^{j-1}f$, $D_x^jf:=D_xD_x^{j-1}f$, $D^0f=D_x^0f=f$, noting that $D^{j-1}f$ and $D_x^{j-1}f$ are again vector-valued functions; moreover, $D^{\leq 1}f:=(f, Df)$.
	\item For $(t,x)\in [0, T)\times\R^d$, we let $\tau:=-\ln(T-t)$ and  $Z:=|x|/(T-t)\in [0, +\infty)$.
	\item For $N\geq0$, $H_x^N$ denotes the inhomogeneous Sobolev space with the norm $\|\cdot\|_{H_x^N}$ with respect to the spatial variables and $\dot H_x^N$ denotes the homogeneous Sobolev space with the norm $\|\cdot\|_{\dot H_x^N}$. Moreover, we denote $L_x^2:=H_x^0$.
\item A function space is a linear vector space if it is closed under addition and multiplication by a constant. A function space is a ring (algebra) if it contains all the constant functions and is closed under addition and multiplication. Then a ring is also a linear vector space.
\end{itemize}

\section{A roadmap of the proof}\label{Sec.Proof}

We introduce the modulus-phase decomposition $u=w \mathrm{e}^{\mathrm{i}\Phi}$, with $w:\R^{1+d}\to\R_{>0}$ and $\Phi:\R^{1+d}\to\R$. Then
\begin{align*}
	\Box u=(\Box w+2\mathrm{i}\partial^{\alpha} w\partial_{\alpha}\Phi+\mathrm{i} w\Box\Phi- w\partial^{\alpha}\Phi\partial_{\alpha}\Phi)\mathrm{e}^{\mathrm{i}\Phi},
\end{align*}
and \eqref{Eq.wave_eq} becomes
\begin{align}
	\label{Eq.wave_eq_polar}\Box w= w^{p}+ w\partial^{\alpha}\Phi\partial_{\alpha}\Phi,\qquad 2\partial^{\alpha} w\partial_{\alpha}\Phi+ w\Box\Phi=0.
\end{align}
Let $b>0$ be a positive constant. We re-normalize according to
\begin{equation}\label{Eq.front}
	w(t,x)=b^{-\frac{1}{p-1}}\rho(t,x),\qquad \Phi(t,x)=b^{-\frac{1}{2}}\phi(t,x),
\end{equation}
then \eqref{Eq.wave_eq_polar} becomes
\begin{align}
	\label{Eq.wave_eq_renomalization}b\Box\rho=\rho^{p}+\rho\partial^{\alpha}\phi\partial_{\alpha}\phi,\qquad 2\partial^{\alpha}\rho\partial_{\alpha}\phi+\rho\Box\phi=0.
\end{align}

We seek solutions $(\rho, \phi)$ to \eqref{Eq.wave_eq_renomalization} in the form of
\begin{equation}\label{Eq.rho_phi_expansion}
	\rho(t,x)=\sum_{n=0}^{\infty}\rho_n(t,x) b^n,\qquad \phi=\sum_{n=0}^\infty \phi_n(t,x)b^n.
\end{equation}
Plugging \eqref{Eq.rho_phi_expansion} into \eqref{Eq.wave_eq_renomalization}, we obtain the following recurrence relation for 
$n\in\Z_{\geq 0}$:
\begin{equation}\label{Eq.recurrence_relation}
	\begin{aligned}
		\Box\rho_{n-1}&=\sum_{n_1+n_2+\cdots+n_p=n}\rho_{n_1}\rho_{n_2}\cdots\rho_{n_p}+
\sum_{n_1+n_2+n_3=n}\rho_{n_1}\pa^\alpha\phi_{n_2}\pa_\alpha\phi_{n_3},\\
		0&=2\sum_{n_1+n_2=n}\pa^\alpha\rho_{n_1}\pa_\alpha\phi_{n_2}+\sum_{n_1+n_2=n}\rho_{n_1}\Box\phi_{n_2},
	\end{aligned}
\end{equation}
where we have used the convention that $\rho_{-n'}=\phi_{-n'}=0$ for all $n'\in\Z_{+}$. Here \eqref{Eq.rho_phi_expansion} is only a formal expansion and we will use cutoff functions to construct approximate solutions.

\subsection{The leading order term of the blow-up solution}\label{Subsec.n=0}

Letting $n=0$ in \eqref{Eq.recurrence_relation}, we know that $(\rho_0,\phi_0)$ satisfies the system\footnote{System \eqref{Eq.wave_eq_n=0} is exactly the same as (2.5) and (2.6) in \cite{Shao-Wei-Zhang} as long as we let $\ell=1+4/(p-1)$ and $\varrho=\rho^{p+1}$.}
\begin{equation}\label{Eq.wave_eq_n=0}
	\rho_0^{p}+\rho_0\partial^{\alpha}\phi_0\partial_{\alpha}\phi_0=0,\qquad 2\partial^{\alpha}\rho_0\partial_{\alpha}\phi_0+\rho_0\Box\phi_0=0.
\end{equation}
For any $\beta>1$, the system \eqref{Eq.wave_eq_n=0} is invariant under the scaling
\[\phi_{0,\lambda}(t, x)=\la^{\beta-1}\phi_0(\lambda t, \lambda x),\qquad \rho_{0,\lambda}(t, x)=\la^{\frac{2\beta}{p-1}}\rho_0(\lambda t, \lambda x),\qquad\forall\ \lambda>0.\]
We seek radially symmetric self-similar blow-up solutions to \eqref{Eq.wave_eq_n=0} of the form
\begin{equation}\label{Eq.phi_0_selfsimilar}
	\phi_0(t, r)=(T-t)^{1-\beta}\hat \phi_0(Z),\quad \rho_0(t,x)=(T-t)^{-\frac{2\beta}{p-1}}\hat\rho_0(Z),\quad Z=\frac{r}{T-t}, \quad r=|x|,
\end{equation}
where $T>0$ is the blow-up time and $\beta>1$ is a constant.\footnote{\label{footnote.beta}Note that $\beta$ in this paper is not the same as $\beta$ in \cite{Shao-Wei-Zhang}. In fact, $\beta_{\text{in this paper}}=\beta_{\text{in \cite{Shao-Wei-Zhang}}}/(\ell+1)$. Hence, $\beta_{\text{in this paper}}>1$ is equivalent to $\beta_{\text{in \cite{Shao-Wei-Zhang}}}>\ell+1$, see Lemma A.7 in \cite{Shao-Wei-Zhang}.} Let $v=\pa_r\phi_0/\pa_t\phi_0$, then $v=v(Z)$ solves the ODE\footnote{ODE \eqref{Eq.ODE_v(Z)} is exactly the same as (2.17) in \cite{Shao-Wei-Zhang}, as long as we let $m=\beta\ell$.}
\begin{equation}\label{Eq.ODE_v(Z)}
	\begin{aligned}&\Delta_Z(Z,v){\mathrm dv}/{\mathrm dZ}=\Delta_v(Z,v),\quad \Delta_v(Z,v):=(1-v^2)[\beta\ell(1-v^2)Z-kv(1-Zv)], \\ &\Delta_Z(Z,v):=Z\left[(1-Zv)^2-\ell(v-Z)^2\right],\end{aligned}
\end{equation}
where $\ell:=1+4/(p-1)>1$ and $k:=d-1\in\Z_{\geq 1}$. See Subsection \ref{Subsec.v(Z)_derivation} for the derivation of \eqref{Eq.ODE_v(Z)}.

Recall the following fact from \cite{Shao-Wei-Zhang} (recalling footnote \ref{footnote.beta}).

\begin{lemma}[\cite{Shao-Wei-Zhang}, Lemma 2.1]
	If $v(Z):[0,1]\to(-1,1)$ is a $C^1$ solution to \eqref{Eq.ODE_v(Z)} with $v(0)=0$ and $\ell>1, \beta>0, k>0$, then $k>\beta(\ell+\sqrt{\ell})$.
\end{lemma}

As a consequence, it is natural to restrict the parameters $(k,\ell,\beta)$ in the following range:
\begin{equation}\label{Eq.parameter_range}
	\beta>1,\qquad \ell>1, \qquad k\in \Z\cap[3,+\infty),\qquad k>\beta(\ell+\sqrt{\ell}).
\end{equation}

\begin{assumption}\label{Assumption}
	There exists a smooth function $v=v(Z)\in(-1,1)$ defined on $Z\in[0,+\infty)$ solving the ODE \eqref{Eq.ODE_v(Z)} with $v(0)=0$ and $v\in C_{\operatorname{o}}^{\infty}([0,+\infty))$.
\end{assumption}
Here we define (with $\R_{\geq0}:=[0,+\infty) $)
\begin{align}\label{Ce}
	C_\text e^\infty(\R_{\geq0})
	:&=\left\{f\in C^\infty(\R_{\geq0}): \exists\ \tilde f\in C^\infty(\R_{\geq0})\ \text{s.t. }f(Z)=\tilde f(Z^2)\ \forall\ Z\in\R_{\geq0}\right\},\\
	\label{Co}C_\text o^\infty(\R_{\geq0})
	:&=\left\{f\in C^\infty(\R_{\geq0}): \exists\ \tilde f\in C^\infty(\R_{\geq0})\ \text{s.t. }f(Z)=Z\tilde f(Z^2)\ \forall\ Z\in\R_{\geq0}\right\}.
\end{align}
Then $C_\text e^\infty(\R_{\geq0}) $ is a ring and $C_\text o^\infty(\R_{\geq0}) $ is a linear vector space.

\begin{remark}\label{Rmk.v(Z)_properties}
	Under Assumption \ref{Assumption} and \eqref{Eq.parameter_range}, we can show that the solution $v(Z)$ satisfies 
	\begin{itemize}
	
	\item $v(Z)<Z$ and $Zv(Z)<1$ for all $Z\in(0,+\infty)$.
	
	\item $\Delta_Z(Z, v(Z))>0$ for $Z\in(0, Z_1)$ and $\Delta_Z(Z, v(Z))<0$ for $Z\in (Z_1, +\infty)$, where $Z_1=\frac{k}{\sqrt{\ell}(k-\beta(\ell-1))}>0$.
	
	 \item Let $\Delta_0(Z):=\Delta_Z(Z, v(Z))$ for $Z\in[0,+\infty)$, then $\Delta_0'(Z_1)\neq0$. 
	
	\end{itemize}
	
	See Subsection \ref{Subsec.v(Z)_properties} for the proof.
\end{remark}

In view of Assumption \ref{Assumption}, we can define that for $Z\in[0,+\infty)$ 
\begin{equation}\label{Eq.phi_0_rho_0}
	\begin{aligned}
		\hat\phi_0(Z):&=\frac1{\beta-1}\exp\left((\beta-1)\int_0^Z\frac{v(s)}{1-sv(s)}\,\mathrm ds\right),\\ \hat\rho_0(Z):&=\frac{(\beta-1)^{\frac{2}{p-1}}\hat\phi_0(Z)^{\frac{2}{p-1}}(1-v(Z)^2)^{\frac{1}{p-1}}}
		{(1-Zv(Z))^{\frac{2}{p-1}}}.
	\end{aligned}
\end{equation}
Then $\hat \phi_0(Z)>0$, $\hat\rho_0(0)=1$ and $\hat\rho_0(Z)>0$ for all $Z\in[0,+\infty)$.
As a consequence, $(\phi_0, \rho_0)$ defined by \eqref{Eq.phi_0_selfsimilar} solves \eqref{Eq.wave_eq_n=0} (see Lemma \ref{Lem.leading_order_eq}), and $\hat\phi_0,\ \hat\rho_0\in C_\text e^\infty([0,+\infty)) $ (see Lemma \ref{Lem.v_hat_rho}).
This is the leading order term of our blow-up solution $(\rho,\phi)$ to \eqref{Eq.wave_eq_renomalization}.


\subsection{Solving $(\rho_n, \phi_n)$ for $n\in\Z_{\geq 1}$} In Subsection \ref{Subsec.n=0}, under Assumption \ref{Assumption}, we construct the leading order blow-up solution $(\rho_0, \phi_0)$. In view of the expansion \eqref{Eq.rho_phi_expansion}, we construct $(\rho_n, \phi_n)$ for $n\in\Z_{\geq 1}$.
We rewrite the recurrence relation \eqref{Eq.recurrence_relation} for $n\in\Z_{\geq 1}$ as 
\begin{align}
	\notag&(p\rho_0^{p-1}+\pa^\alpha\phi_0\pa_\alpha\phi_0)\rho_n+2\rho_0\pa^\alpha\phi_0\pa_\alpha\phi_n\\
	\label{Fn}=&\Box\rho_{n-1}
	-\sum_{\substack{n_1+\cdots+n_p=n\\ n_1, \cdots, n_p\leq n-1}}\rho_{n_1}\cdots\rho_{n_p}-\sum_{\substack{n_1+n_2+n_3=n\\ n_1, n_2, n_3\leq n-1}}\rho_{n_1}\pa^\alpha\phi_{n_2}\pa_\alpha\phi_{n_3}=:F_n,\\
	\notag&\rho_0\Box\phi_n+2\pa^\alpha\rho_0\pa_\alpha\phi_n+2\pa^\alpha\phi_0\pa_\alpha\rho_n+\Box\phi_0\rho_n\\
	\label{Gn}=&-2\sum_{\substack{n_1+n_2=n\\ n_1, n_2\leq n-1}}\pa^\alpha\rho_{n_1}\pa_\alpha\phi_{n_2}-\sum_{\substack{n_1+n_2=n\\ n_1, n_2\leq n-1}}\rho_{n_1}\Box\phi_{n_2}=:G_n.
\end{align}
Using the equations for $(\rho_0, \phi_0)$ given by \eqref{Eq.wave_eq_n=0}, the above recurrence relation becomes
\begin{align}
	&(p-1)\rho_0^{p-1}\rho_n+2\rho_0\pa^\alpha\phi_0\pa_\alpha\phi_n=F_n,\label{Eq.recurrence_F_n}\\
	&\pa^\al(\rho_0^2\pa_\al\phi_n)+2\pa^\al(\rho_0\pa_\al\phi_0\rho_n)=\rho_0 G_n.\label{Eq.recurrence_G_n}
\end{align}
By \eqref{Eq.recurrence_F_n}, we have
\begin{equation}\label{Eq.rho_n}
	\rho_n=\frac{\rho_0^{1-p}F_n}{p-1}-\frac2{p-1}\rho_0^{2-p}\pa^\al\phi_0\pa_\al\phi_n.
\end{equation}
Substituting the above identity into \eqref{Eq.recurrence_G_n}, we obtain the following linear equation for $\phi_n$:
\begin{equation}\label{Eq.phi_n_eq}
	\pa^\al\left(\rho_0^2\pa_\al\phi_n-\frac4{p-1}\rho_0^{3-p}\pa_\al\phi_0\pa^{\tilde\alpha}\phi_0\pa_{\tilde\alpha}\phi_n\right)=\rho_0G_n-\frac2{p-1}\pa^\al\left(\rho_0^{2-p}\pa_\al\phi_0 F_n\right)=:H_n.
\end{equation}
We introduce the linearized operator
\begin{equation}\label{Eq.L}
	\mathscr{L}(\phi):=\pa^\al\left(\rho_0^2\pa_\al\phi-\frac4{p-1}\rho_0^{3-p}\pa_\al\phi_0\pa^{\tilde\alpha}\phi_0\pa_{\tilde\alpha}\phi\right),\qquad \phi=\phi(t,x)=\phi(t,r).
\end{equation}
Then our aim is to solve inductively $\LL(\phi_n)=H_n$ for each $n\geq 1$.

Indeed, we can show that $\LL$ is surjective in some well-chosen functional spaces and then we solve  $\LL(\phi_n)=H_n$ in these spaces. 
Letting $\tau=\ln\frac1{T-t}$, we define (here $C_\text e^\infty([0,+\infty)) $ is defined in \eqref{Ce})
\begin{align}
	\label{X0}\mathscr X_0:=&\left\{f(t,x)=\sum_{j=0}^n f_j(Z)\tau^j: n\in\Z_{\geq 0}, f_j\in C_\text e^\infty([0,+\infty))\ \ \forall\ j\in\Z\cap[0,n]\right\},\\
	\label{Xl}\mathscr X_\la:=&(T-t)^\la\mathscr X_0=\left\{f(t,x)=(T-t)^\la g(t,x)=\mathrm e^{-\la\tau}g(t,x):g\in\mathscr X_0\right\},\quad \forall\ \la\in\C.
\end{align}
Then $\mathscr X_0$ is a ring (using that $\{f(t,x)=f(Z)\tau^j:f_j\in C_\text e^\infty([0,+\infty)),\ j\in\Z_{\geq 0}\} $ is closed under multiplication) and $\mathscr X_\la$ is a linear vector space.

We have the following properties for the functional spaces $\mathscr X_\la$.

\begin{lemma}\label{Lem.X_la_properties}
	\begin{enumerate}[(i)]
		\item  Let $\la,\mu\in\C$, $f\in\mathscr X_\la, g\in\mathscr X_\mu$. Then $\partial_t f\in\mathscr X_{\la-1}$, $\Delta f\in\mathscr X_{\la-2}$, $\Box f\in\mathscr X_{\la-2}$, $fg\in\mathscr X_{\la+\mu}$, $\pa^\alpha f\pa_\alpha g\in \mathscr X_{\la+\mu-2}$, and  $\pa^\alpha(f\pa_\alpha g)\in\mathscr X_{\la+\mu-2}$.
		\item Let $\la,\mu\in\R$ and $j\in\Z_{\geq 0}$ be such that $\la>j+\mu$. If $f\in\mathscr X_\la$, then $(T-t)^{-\mu}D^jf\in L^\infty(\mathcal C)$, where $\mathcal C$ is the light cone $\mathcal C:=\left\{(t,x)\in[0, T)\times\R^d: |x|<2(T-t)\right\}$.
		\item Let $\la,\mu\in\R$ and $j\in\Z_{\geq 0}$ be such that $\la\geq j+\mu$. If $f(t,x)=(T-t)^\la\hat f(Z)$ for some $\hat f\in C_{\operatorname{e}}^\infty([0,+\infty))$, then $(T-t)^{-\mu}D^jf\in L^\infty(\mathcal C)$.
	\end{enumerate}
\end{lemma}
The proof of Lemma \ref{Lem.X_la_properties} can be found in Subsection \ref{Subsec.Functional_spaces}.

\begin{proposition}\label{Prop.L_surjective}
	The linear operator $\LL:\mathscr X_{\la}\to\mathscr X_{\la-\g}$ is surjective for all $\la\in\C$, where $\g:=4\beta/(p-1)+2=\beta(\ell-1)+2$.
\end{proposition}
See Section \ref{Sec.Surjective} for the proof of Proposition \ref{Prop.L_surjective}.

Let
\begin{equation}\label{Eq.la_n_mu_n}
	\la_n:=(2n-1)(\beta-1),\qquad\mu_n:=2n(\beta-1)-\frac{2\beta}{p-1},\qquad\forall\ n\in\Z_{\geq 0}.
\end{equation}
Recall from \eqref{Eq.phi_0_selfsimilar} that
\[\phi_0(t,r)=(T-t)^{\la_0}\hat\phi_0(Z),\qquad \rho_0(t,r)=(T-t)^{\mu_0}\hat\rho_0(Z).\]
As $\hat\phi_0(Z), \hat\rho_0(Z)\in C_\text{e}^\infty([0,+\infty))$, by \eqref{X0}, \eqref{Xl} we have $\phi_0\in \mathscr X_{\la_0}$ and $\rho_0\in \mathscr X_{\mu_0}$.
Similarly, for $ a\in\R$ we have $\rho_0(t,r)^a=(T-t)^{a\mu_0}\hat\rho_0(Z)^a $ and $\hat\rho_0(Z)^a\in C_\text{e}^\infty([0,+\infty))$, then $\rho_0^a\in \mathscr X_{a\mu_0}$.
Moreover, $\hat\phi_0(Z), \hat\rho_0(Z)$ are real-valued, so are $\phi_0$, $\rho_0$.
\begin{lemma}\label{lem1}
	Assume that $n\in\Z_{\geq 1}$, $\phi_j\in \mathscr X_{\la_j}$, $\rho_j\in\mathscr X_{\mu_j}$ are are real-valued for $j\in\Z\cap[0, n-1]$. Let $F_n$, $G_n$ be defined in \eqref{Fn}, \eqref{Gn}. Then there exist real-valued $\phi_n\in \mathscr X_{\la_n}$ and $\rho_n\in \mathscr X_{\mu_n}$ such that \eqref{Eq.recurrence_F_n} and \eqref{Eq.recurrence_G_n} hold.
\end{lemma}\begin{proof}
	By Lemma \ref{Lem.X_la_properties} (i) and the definition of $F_n$, we have $F_n\in\mathscr X_{\mu_{n-1}-2}$, where we have used the fact that $\mu_{n_1}+\cdots+\mu_{n_p}=\mu_{n-1}-2$ if $n_1+\cdots+n_p=n$ and $\mu_{n_1}+\la_{n_2}+\la_{n_3}-2=\mu_{n-1}-2$ if $n_1+n_2+n_3=n$. Using Lemma \ref{Lem.X_la_properties} (i) and the definition of $G_n$, we have $G_n\in\mathscr X_{(2n-1)(\beta-1)-2\beta/(p-1)-2}$, where we have used the fact that $\mu_{n_1}+\la_{n_2}-2=(2n-1)(\beta-1)-2\beta/(p-1)-2$ if $n_1+n_2=n$. It follows from Lemma \ref{Lem.X_la_properties} (i) that $\rho_0G_n\in\mathscr X_{\mu_0+(2n-1)(\beta-1)-2\beta/(p-1)-2}=\mathscr X_{2n(\beta-1)-\beta\ell-1}$ (recall that $\rho_0\in \mathscr X_{\mu_0}$, $\ell=1+\frac{4}{p-1}$). Since $\rho_0^{2-p}\in\mathscr X_{(2-p)\mu_0}$, $F_n\in\mathscr X_{\mu_{n-1}-2}$, by Lemma \ref{Lem.X_la_properties} (i) we get $\rho_0^{2-p}F_n\in \mathscr X_{\mu_{n-1}-2+(2-p)\mu_0}$, then by $\phi_0\in \mathscr X_{\la_0}$ we have
	\[\pa^\alpha\left(\rho_0^{2-p}\pa_\alpha\phi_0 F_n\right)\in \mathscr X_{\mu_{n-1}-2+(2-p)\mu_0+\la_0-2}=\mathscr X_{2n(\beta-1)-\beta\ell-1}.\]
	Hence by the definition of $H_n$ in \eqref{Eq.phi_n_eq}, we have $H_n\in \mathscr X_{2n(\beta-1)-\beta\ell-1}=\mathscr X_{\la_n-\g}$ (recall that $\g=\beta(\ell-1)+2$). Moreover, $F_n, G_n, H_n$ are real-valued.
	
	By Proposition \ref{Prop.L_surjective}, there exists (real-valued) $\phi_n\in\mathscr X_{\la_n}$ such that $\LL(\phi_n)=H_n$ (otherwise take $\Re\phi_n$), then \eqref{Eq.phi_n_eq} holds. Let $\rho_n$ be defined by \eqref{Eq.rho_n}. Then $\rho_n$ is real-valued.  Moreover, using (i) of Lemma \ref{Lem.X_la_properties}, $\rho_0^{1-p}\in\mathscr X_{(1-p)\mu_0}$, $\rho_0^{2-p}\in\mathscr X_{(2-p)\mu_0}$, $F_n\in\mathscr X_{\mu_{n-1}-2}$,
	$\phi_0\in \mathscr X_{\la_0}$ and $\phi_n\in \mathscr X_{\la_n}$, we have $$\rho_0^{1-p}F_n\in\mathscr X_{\mu_0(1-p)+\mu_{n-1}-2}=\mathscr X_{\mu_n},\quad  \rho_0^{2-p}\pa^\alpha\phi_0\pa_\alpha\phi_n\in\mathscr X_{\mu_0(2-p)+\la_0+\la_n-2}=\mathscr X_{\mu_n},$$ hence $\rho_n\in \mathscr X_{\mu_n}$. Now \eqref{Eq.recurrence_F_n} follows from \eqref{Eq.rho_n}, and \eqref{Eq.recurrence_G_n} follows from \eqref{Eq.rho_n} and \eqref{Eq.phi_n_eq}.
\end{proof}

As $\phi_0\in \mathscr X_{\la_0}$, $\rho_0\in \mathscr X_{\mu_0}$ and  $\phi_0$, $\rho_0$ are real-valued,  by Lemma \ref{lem1} and the induction, we have the following result.
\begin{proposition}\label{Prop.phi_n_rho_n}
	Let $\phi_0$, $\rho_0$ be defined in \eqref{Eq.phi_0_selfsimilar}. For each $n\in\Z_{\geq 1}$, there exist real-valued $\phi_n\in \mathscr X_{\la_n}$ and $\rho_n\in \mathscr X_{\mu_n}$ such that \eqref{Eq.recurrence_F_n} and \eqref{Eq.recurrence_G_n} hold with $F_n$, $G_n$ defined in \eqref{Fn}, \eqref{Gn}.
	Hence, \eqref{Eq.recurrence_relation} holds for $n\in\Z_{\geq0}$.
\end{proposition}

Now we briefly explain the ideas in the proof of Proposition \ref{Prop.L_surjective}. In the proof of Lemma \ref{lem1}, we see that we only need to use the surjectivity of $\LL$ for $\mathscr X_{\la}$ to $\mathscr X_{\la-\g}$ for $\la\in\{\la_n:n\in\Z_{+}\}$. However, this is not easy to solve the equation $\LL f=g$ for $f\in \mathscr X_{\la}$ even in the simplest case $g=(T-t)^{\la-\g}\hat g(Z)\in \mathscr X_{\la-\g}$ for some $\hat g\in C_{\text{e}}^\infty([0, +\infty))$ (without the logarithm correction $\tau^j$ for $j\in\Z_{+}$), in which process we need to check a non-degenerate property 
{(nonzero of Wronski defined in \eqref{W1})} on the coefficients of $\LL_\la$ (defined in \eqref{Eq.LL_la}), and it is difficult to check that all $\la_n$ satisfy the non-degenerate property,
 even for one $\la_{n_0}$. To overcome this drawback, we 
solve the equation for all $\la\in\C$, not merely for those $\la\in\{\la_n:n\in\Z_{+}\}$. It turns out that the non-degenerate property holds all but countably many $\la\in\C$ (these $\la$ are ``bad" in some sense) and the solution depends analytically on $\la$. In this way, we can show that the solution $f=f(\cdot;\la)$ is a meromorphic function on $\la$. For those countably many $\la\in\C$ not satisfying the non-degenerate property, the analytic property of $f$ allows us to introduce a logarithm correction to solve the corresponding equation for ``bad" $\la$. See Section \ref{Sec.Surjective} and Appendix \ref{Appen.ODE} for details.

\subsection{The approximate solution}
Let $(\rho_n,\phi_n)\in\mathscr X_{\mu_n}\times\mathscr X_{\la_n}$ $(n\in\Z_{\geq 0})$ be given by Proposition \ref{Prop.phi_n_rho_n}. We fix a bump function $\eta\in C_c^\infty(\R;[0,1])$ such that $\eta|_{[0,1]}=1$ and $\eta|_{[2,+\infty)}=0$.
\begin{proposition}\label{Prop.Approximate}
	Let $T=b=1$. There exist $N_0\in\Z_+$, $c_0\in (0,T)$ and a sequence $\{T_n\}_{n\geq 0}$ such that $T_n=T$ for $0\leq n<N_0$, $0<T_n\leq T_{n-1}/4$ for all $n\geq N_0$, and for functions\footnote{\label{footnote.finite_sum}For fixed $(t,x)\in [0,T)\times\R^d$, the summations in \eqref{Eq.rho_*_phi_*} are both finite sums. Indeed, we have $T-t>0$, then $\lim_{n\to\infty}(T-t)/T_n=+\infty$, thus $(T-t)/T_n>2$ for all sufficiently large $n$ and hence $\eta\big((T-t)/T_n\big)=0$ for all sufficiently large $n$. As a consequence, we have $\rho_*, \phi_*\in C^\infty([0,T)\times\R^d)$.}
	\begin{align}
		\rho_*(t,x):&=\sum_{n=0}^\infty \eta\left(\frac{T-t}{T_n}\right)\rho_n(t,x)b^n,\quad \phi_*(t,x):=\sum_{n=0}^\infty \eta\left(\frac{T-t}{T_n}\right)\phi_n(t,x)b^n,\label{Eq.rho_*_phi_*}\\
		E_*:&=\rho_*^p+\rho_*\pa^\alpha\phi_*\pa_\alpha\phi_*-b\Box\rho_*,\quad J_*:=2\pa^\al\rho_*\pa_\alpha\phi_*+\rho_*\Box\phi_*,\label{EJ1}
	\end{align}
	defined on $(t,x)\in[0,T)\times\R^d$ we have
	\begin{align}
		&(T-t)^{\beta+j-1}D^j\phi_*\in L^\infty(\mathcal  C),\quad (T-t)^{\frac{2\beta}{p-1}+j}D^j\rho_*\in L^\infty(\mathcal  C),\qquad\forall\ j\in\Z_{\geq 0},\label{Eq.phi_*_rho_*_1}\\
		&(T-t)^{\beta}(\pa_t\phi_*-|D_x\phi_*|)\geq c_0,\quad (T-t)^{\frac{2\beta}{p-1}}\rho_*\geq c_0,\quad \forall\ (t,x)\in\mathcal C,\ T-t<c_0,\label{Eq.phi_*_rho_*_2}\\
		\label{EJ}&(T-t)^{-\la}(D^jE_*, D^jJ_*)\in L^\infty(\mathcal C),\qquad\forall\ \la>0,\ \forall\ j\in\Z_{\geq 0}.
	\end{align}
\end{proposition}

See Section \ref{Sec.Approximate} for the proof of Proposition \ref{Prop.Approximate}.

\subsection{Solving nonlinear wave equation}
\begin{proposition}\label{Prop.Solve_wave}
	Assume that $T=1$, $w_*\in C^\infty(\mathcal C),\Phi_*\in C^\infty(\mathcal C)$ satisfy
	\begin{align}
		&(T-t)^{\beta+j-1}D^j\Phi_*\in L^\infty(\mathcal  C),\quad (T-t)^{\frac{2\beta}{p-1}+j}D^jw_*\in L^\infty(\mathcal  C),\qquad\forall\ j\in\Z_{\geq 0},\label{Eq.Phi_*_w_*_1}\\
		&(T-t)^{\beta}(\pa_t\Phi_*-|D_x\Phi_*|)\geq c_0,\quad (T-t)^{\frac{2\beta}{p-1}}w_*\geq c_0,\quad \forall\ (t,x)\in\mathcal C,\ T-t<c_0,\label{Eq.Phi_*_w_*_2}
	\end{align}
	for some $c_0\in (0,T)=(0,1)$. Suppose that \eqref{EJ} holds for $E_*$, $J_*$ defined as\begin{align}\label{EJ2}
		E_*:&=w_*^p+w_*\pa^\alpha\Phi_*\pa_\alpha\Phi_*-\Box w_*,\quad J_*:=2\pa^\al w_*\pa_\alpha\Phi_*+w_*\Box\Phi_*,
	\end{align} Then there exist $c_1\in (0,c_0)$ and $u\in C^2((T-c_1,T)\times\R^d;\C)$ such that $u(t,\cdot),\partial_tu(t,\cdot)\in C_c^{\infty}(\R^d;\C)$ for $t\in(T-c_1,T),$
	$\Box u=|u|^{p-1}u$ for $t\in(T-c_1,T),$ $|x|\leq T-t$, and
	\begin{equation}\label{Eq.u_bound}
		C^{-1}(T-t)^{-\frac{2\beta}{p-1}}\leq |u(t,x)|\leq C(T-t)^{-\frac{2\beta}{p-1}},\qquad\forall\ t\in(T-c_1,T),\ |x|\leq T-t
	\end{equation}
	for some constant $C>0$.
\end{proposition}

In fact, $u=(1+h) w_*\mathrm e^{\ii\Phi_*}$, $h=O((T-t)^{\lambda})$, $ \forall\ \lambda>0$. To prove Proposition \ref{Prop.Solve_wave}, it suffices to solve the equation for $h$ (see \eqref{Eq.h_eq0}) and prove that $h$ is  small. In view of the singularity of \eqref{Eq.h_eq0} at blow-up time $T$, we take a sequence $\varepsilon_n\downarrow0$ and then we solve \eqref{Eq.h_eq0} (with technical truncation) with zero initial data at $T-\varepsilon_n$ in a backward direction. We denote the solution for each $n\in\Z_{+}$ by $h_n$. Using energy estimates and a bootstrap argument, we can show that $h_n$ lives in an interval with a positive lower bound independent of $n\in\Z_{+}$. Taking the limit $n\to\infty$ we get a desired solution to \eqref{Eq.h_eq0} (in the light cone). See Section \ref{Sec.Sol_wave} for details.

\subsection{Proof of main results}\label{Subsec.Proof_main}

Let's begin with the proof of  Theorem \ref{Thm.main_thm}.

\begin{proof}[Proof of Theorem \ref{Thm.main_thm}]

Let $T=b=1$ and $(\rho_*, \phi_*)\in C^\infty(\mathcal C)$, $c_0\in (0,T)$ be given by Proposition \ref{Prop.Approximate}, and $w_*=\rho_*$, $\Phi_*=\phi_*$.
Then $w_*, \Phi_*\in C^\infty(\mathcal C)$, \eqref{Eq.Phi_*_w_*_1} is equivalent to \eqref{Eq.phi_*_rho_*_1}, \eqref{Eq.Phi_*_w_*_2} is equivalent to \eqref{Eq.phi_*_rho_*_2}. The definitions of $E_*$, $J_*$ in \eqref{EJ1} and \eqref{EJ2} are the same, and \eqref{EJ} also follows from Proposition \ref{Prop.Approximate}. It follows from Proposition \ref{Prop.Solve_wave} that there exist $c_1\in (0,c_0)\subset(0,1)$, $\tilde u\in C^2((T-c_1,T)\times\R^d;\C)$ such that $\tilde u(t,\cdot),\partial_t\tilde u(t,\cdot)\in C_c^{\infty}(\R^d;\C)$ for $t\in(T-c_1,T),$
$\Box\tilde u=|\tilde u|^{p-1}\tilde u$ for $t\in(T-c_1,T),$ $|x|\leq T-t$,  and
\begin{equation}\label{Eq.tilde_u_blow-up}
	C_1^{-1}(T-t)^{-\frac{2\beta}{p-1}}\leq |\tilde u(t,x)|\leq C_1(T-t)^{-\frac{2\beta}{p-1}},\qquad\forall\ t\in(T-c_1,T),\ |x|\leq T-t
\end{equation}
for some constant $ C_1>0$. Choose 
initial data $u_0, u_1: \R^d\to \C$ such that
\[u_0(x)=\tilde u(T-c_1/2, x), \quad u_1(x)=\pa_t\tilde u(T-c_1/2, x),\qquad\forall\ x\in\R^d.\]
Then $u_0,u_1\in C_c^{\infty}(\R^d;\C)$. Moreover, let $u_*(t,x)=\tilde u(t+T-c_1/2,x)$ for $t\in(-c_1/2,c_1/2)$, $x\in\R^d$ then $u_*(0,x)=u_0(x)$, $\partial_t u_*(0,x)=u_1(x)$ and
$\Box u_*=| u_*|^{p-1} u_*$ for $t\in[0,c_1/2),$ $|x|\leq c_1/2-t$.
Suppose for contradiction that Theorem \ref{Thm.main_thm} fails for this initial data $u_0, u_1$, then there exists a smooth function $u: [0,+\infty)\times\R^d\to\C$ such that $\Box u=|u|^{p-1}u$
and $u(0, x)=u_0(x), \pa_tu(0, x)=u_1(x)$ for all $x\in\R^d$. Finite speed of propagation 
shows that $u= u_*$ in the region $\{(t,x)\in[0, c_1/2)\times\R^d: |x|\leq c_1/2-t\}$. Hence by \eqref{Eq.tilde_u_blow-up} we have
\[|u(t,0)|=|u_*(t,0)|=|\tilde u(t+T-c_1/2,0)|\geq C_1^{-1}(c_1/2-t)^{-\frac{2\beta}{p-1}},\qquad\forall\ t\in[0, c_1/2).\]
On the other hand, since $u$ is smooth on $[0, +\infty)\times\R^d$, we have $|u(t,x)|\leq C$ for all $|x|\leq 2T$ and $t\in[0, c_1/2]$, where $C>0$ is a constant. This reaches  a contradiction.
\end{proof}

The following result  was proved in \cite{Shao-Wei-Zhang} Theorem 2.2 and Lemma A.7 ($ \beta>\ell+1$ in \cite{Shao-Wei-Zhang} is equivalent to $\beta>1$ in this paper, recalling footnote \ref{footnote.beta}).
\begin{lemma}\label{Lem.beta_neq_l+1}
	There exist $\ell^*(3)=\frac{76-4\sqrt{154}}{23}\in(\frac{8}{7},\frac{7}{6})$ and $\ell_1(4)\in(5/4,4/3) $ such that if
	\begin{equation}\label{Eq.A.15}
		k=4,\ 1<\ell<\ell_1(4)\ \ \text{ or }\ k=3,\ 1<\ell<\ell^*(3),
	\end{equation}
	Then  there exists $\beta\in(1, k/(\ell+\sqrt\ell))$ \footnote{Note that $1<\ell^*(3)<\ell_1(4)<3/2$, thus if \eqref{Eq.A.15} holds then $\ell +\sqrt\ell<2\ell<3\leq k$.}
	such that Assumption \ref{Assumption} holds for $d=k+1$.
\end{lemma}
\begin{proof}[Proof of Corollary \ref{Cor.main_cor}.]
	Let $T=1$. If $d=4$, $k=3$, $p\geq 29$, $\ell=1+\frac{4}{p-1}$, then $1<\ell\leq 1+\frac{4}{29-1}=\frac{8}{7}<\ell^*(3)$. If $d=5$, $k=4$, $p\geq 17$, $\ell=1+\frac{4}{p-1}$, then $1<\ell\leq 1+\frac{4}{17-1}=\frac{5}{4}<\ell_1(4)$. Thus, if $d=4$, $p\geq 29$ or $d=5$, $p\geq 17$, then \eqref{Eq.A.15} holds for $k=d-1$, $\ell=1+\frac{4}{p-1}$ and the result follows from Theorem \ref{Thm.main_thm}.
	
	The remaining case is $d>5$, $p\geq 17$. Then Assumption \ref{Assumption} holds with $d$ replaced by $d'=5$. By the proof of Theorem \ref{Thm.main_thm}, there exists $c_1\in (0,1)$,
	$\tilde u\in C^2((T-c_1,T)\times\R^5;\C)$ such that $\tilde u(t,\cdot),\partial_t\tilde u(t,\cdot)\in C_c^{\infty}(\R^5;\C)$ for $t\in(T-c_1,T),$
	$\Box\tilde u=|\tilde u|^{p-1}\tilde u$ for $t\in(T-c_1,T),$ $|x|\leq T-t$,  and \eqref{Eq.tilde_u_blow-up} holds for some constant $ C_1>0$.
	Choose  initial data $u_0, u_1: \R^d\to \C$ such that
	\[u_0(x)=\eta(|x|)\tilde u(T-c_1/2, x_1,\cdots,x_5), \quad u_1(x)=\eta(|x|)\pa_t\tilde u(T-c_1/2, x_1,\cdots,x_5),\]
	for all $x=(x_1,\cdots,x_d)\in\R^d$.
	Then $u_0,u_1\in C_c^{\infty}(\R^d;\mathcal C)$. Moreover, let $u_*(t,x)=\eta(|x|)\tilde u(t+T-c_1/2,x_1,\cdots,x_5)$ for $t\in(-c_1/2,c_1/2)$, $x=(x_1,\cdots,x_d)\in\R^d$. Then $u_*(0,x)=u_0(x)$, $\partial_t u_*(0,x)=u_1(x)$ and
	$\Box u_*=| u_*|^{p-1} u_*$ for $t\in[0,c_1/2),$ $|x|\leq c_1/2-t$. Here we used that $\eta(|x|)=1 $ for $|x|\leq 1$ and that if $t\in[0,c_1/2),$ $|x|\leq c_1/2-t$ then $|x|\leq c_1/2<1$.
	Suppose for contradiction that Corollary \ref{Cor.main_cor} fails for this kind of initial data $u_0, u_1$, then there exists a smooth function $u: [0,+\infty)\times\R^d\to\C$ such that $\Box u=|u|^{p-1}u$
	and $u(0, x)=u_0(x), \pa_tu(0, x)=u_1(x)$ for all $x\in\R^d$, and we can get a contradiction as in the proof of Theorem \ref{Thm.main_thm}.
\end{proof}

\section{The approximate solution}\label{Sec.Approximate}

In this section, we prove Proposition \ref{Prop.Approximate}, i.e., the construction of the approximate solution.

\subsection{Construction of the approximate solution}

Let $T=b=1$ and $(\rho_n,\phi_n)\in\mathscr X_{\mu_n}\times\mathscr X_{\la_n}$ $(n\in\Z_{\geq 0})$ be given by Proposition \ref{Prop.phi_n_rho_n}. For $N\in\Z_+$ and $(t,x)\in[0,T)\times\R^d$, let
\begin{align*}
	&\rho_{(N)}(t,x):=\sum_{n=0}^{N}\rho_n(t,x) b^n,\qquad \phi_{(N)}(t,x):=\sum_{n=0}^N \phi_n(t,x)b^n,\\
	&E_N:=\rho_{(N)}^{p}+\rho_{(N)}\partial^{\alpha}\phi_{(N)}\partial_{\alpha}\phi_{(N)}-b\Box\rho_{(N)},\\
	&J_N:= 2\partial^{\alpha}\rho_{(N)}\partial_{\alpha}\phi_{(N)}+\rho_{(N)}\Box\phi_{(N)}.
\end{align*}
It follows from  \eqref{Eq.recurrence_relation}  that 
\begin{align*}
	E_{N}(t,x)=\sum_{n=N+1}^{pN}E_{N,n}(t,x) b^n,\qquad J_{N}(t,x)=\sum_{n=N+1}^{2N} J_{N,n}(t,x)b^n
\end{align*}with\begin{align*}
	E_{N,n}:&=\sum_{\substack{n_1+\cdots+n_p=n\\ n_1, \cdots, n_p\leq N}}\rho_{n_1}\rho_{n_2}\cdots\rho_{n_p}+\sum_{\substack{n_1+n_2+n_3=n\\ n_1, n_2, n_3\leq N}}\rho_{n_1}\pa^\alpha\phi_{n_2}\pa_\alpha\phi_{n_3}-\Box\rho_{N}\mathbf{1}_{n=N+1},\\
	J_{N,n}:&=2\sum_{\substack{n_1+n_2=n\\ n_1, n_2\leq N}}\pa^\alpha\rho_{n_1}\pa_\alpha\phi_{n_2}+\sum_{\substack{n_1+n_2=n\\ n_1, n_2\leq N}}\rho_{n_1}\Box\phi_{n_2}.
\end{align*}
Then $E_{N,n}\in \mathscr{X}_{2n(\beta-1)-2p\beta/(p-1)}$, $J_{N,n}\in \mathscr{X}_{(2n-1)(\beta-1)-2\beta/(p-1)-2}=\mathscr{X}_{(2n+1)(\beta-1)-2p\beta/(p-1)}$. Here the proof  is  similar to Lemma \ref{lem1}.

Take $N_0\in \Z_+$ such that $2N_0(\beta-1)-2p\beta/(p-1)>3 $. We fix such $N_0$ (which is the same as the one in Proposition \ref{Prop.Approximate}) and a non-decreasing sequence $\{k_N\}_{N\in\Z_{\geq N_0}}\subset \Z_+$ such that
\begin{equation}\label{Eq.k_N}
	2N(\beta-1)-2p\beta/(p-1)>3k_N \ \forall\ N\in\Z\cap[N_0,+\infty), \text{ and }\lim_{N\to\infty}k_N=+\infty.
\end{equation}
Then by Lemma \ref{Lem.X_la_properties} (ii), we have $(T-t)^{-2k_N}(D^jE_{N},D^jJ_{N})\in L^{\infty}(\mathcal C) $ for
$0\leq j\leq k_N$, $N\geq N_0$. Or equivalently, for each $N\in\Z\cap[N_0,+\infty)$ there is a constant $A_N>0$ satisfying
\begin{align*}
	|D^jE_{N}(t,x)|+|D^jJ_{N}(t,x)|\leq A_N(T-t)^{2k_N},\quad \forall\ 0\leq j\leq k_N,\ (t,x)\in \mathcal C.
\end{align*}

{In fact for every fixed $n>d/2$, we can use $(\rho_{(N)},\phi_{(N)})$  as an approximate solution for  $N$ large enough (but fixed) 
to construct blow-up solutions of $H^n$ regularity. But to obtain a blow-up solution of $C^{\infty}$ initial data, we need to sum all the 
$(\rho_{n},\phi_{n}) $ with truncation 
as in \eqref{Eq.rho_*_phi_*}.}
Note that for $T-t\in [2T_{N+1},T_N] $, we have $E_{*}(t,x)=E_{N}(t,x) $ and $J_{*}(t,x)=J_{N}(t,x) $. The following result extends the above estimate to the case $T-t\in [T_{N+1},T_N]$ (with a possible different $A_N$).

\begin{lemma}\label{lem3}
	Let $T=b=1$. Then there exists a sequence $\{A_N\}_{N\in\Z_{\geq N_0-1}}$ such that for all $\{T_n\}_{n\in\Z_{\geq 0}}$ satisfying $T_n=T$ for $0\leq n<N_0$, $0<T_n\leq T_{n-1}/4$ for all $n\geq N_0$, if we define 
	$ \rho_*$, $ \phi_*$, $E_*$, $J_*$ by \eqref{Eq.rho_*_phi_*}, \eqref{EJ1}, then for $j\in\Z\cap[0,k_N]$ we have
	\begin{align}\label{T0}
		|D^jE_{*}(t,x)|+|D^jJ_{*}(t,x)|\leq A_N(T-t)^{2k_N},\quad \forall\ T-t\in [T_{N+1},T_N],\ (t,x)\in \mathcal C.
\end{align}\end{lemma}

\begin{lemma}\label{lem4}
	Let $T=b=1$. There exists a sequence $\{\widetilde{T}_n\}_{n\geq 0}$ {satisfying} $$\widetilde{T}_n=T \text{ for }0\leq n<N_0\quad\text{and}\quad  0<\widetilde{T}_n\leq \widetilde{T}_{n-1}/4 \text{ for }n\geq N_0,$$
	{such that }for every sequence $\{{T}_n\}_{n\geq 0}$ {with} ${T}_n=T$ for $0\leq n<N_0$ and $0<{T}_n\leq \min(\widetilde T_n, T_{n-1}/4)$ for  $n\geq N_0$, for $ \rho_*$, $ \phi_*$ defined in \eqref{Eq.rho_*_phi_*}, 
	we have
	\begin{align}\label{T1}
		(T-t)^{j}D^j(\phi_*-\phi_0)\in L^\infty(\mathcal  C),\ (T-t)^{\frac{2\beta}{p-1}+j-\beta+1}D^j(\rho_*-\rho_0)\in L^\infty(\mathcal  C),\ \forall\ j\in\Z_{\geq0}.\end{align}
\end{lemma}
\begin{lemma}\label{lem5}
	Let $T=1$. There exists $\widetilde{c}\in (0,T)$ such that 
	\begin{align}\label{T2}		&(T-t)^{\beta-1+j}D^j\phi_0\in L^\infty(\mathcal  C),\quad (T-t)^{\frac{2\beta}{p-1}+j}D^j\rho_0\in L^\infty(\mathcal  C),\quad \forall\ j\in\Z_{\geq0},\\
		\label{T3}&(T-t)^{\beta}(\pa_t\phi_0-|D_x\phi_0|)(t,x)\geq \widetilde{c},\quad (T-t)^{\frac{2\beta}{p-1}}\rho_0(t,x)\geq \widetilde{c},\quad \forall\ (t,x)\in\mathcal C.
\end{align}\end{lemma}

Let's first prove Proposition \ref{Prop.Approximate} by admitting Lemma \ref{lem3}$\sim$Lemma \ref{lem5} for the moment.

\begin{proof}[Proof of Proposition \ref{Prop.Approximate}]
	{\bf Step 1.} Construction of the sequence $\{{T}_n\}_{n\geq 0}$. Let $A_N>0$ be given by Lemma \ref{lem3} and the sequence $\{\widetilde{T}_n\}_{n\geq 0}$ be given by Lemma \ref{lem4}.
	Let ${T}_n=T$ for $0\leq n<N_0$ and $T_n=\min(\widetilde{T}_n,A_n^{-1/k_n},T_{n-1}/4)$ for $n\geq N_0$. Then
	$0<{T}_n\leq {T}_{n-1}/4$, $T_n\leq \widetilde{T}_n$, $A_n{T}_n^{k_n}\leq 1$, for all $n\geq N_0$, and $\lim\limits_{N\to+\infty}T_N=0$. So, there hold \eqref{T0} for $j\in\Z\cap[0,k_N]$, \eqref{T1}, \eqref{T2}, and \eqref{T3} with $\widetilde{c}\in (0,T)$ given by Lemma \ref{lem5}.\smallskip
	
	{\bf Step 2.} Proof of \eqref{Eq.phi_*_rho_*_1}. As $ \beta>1$, $(T-t)^{\beta-1}\in L^\infty(\mathcal  C)$, we get by \eqref{T1}  that 
	\begin{align*}
		&(T-t)^{\beta-1+j}D^j(\phi_*-\phi_0)=(T-t)^{\beta-1}(T-t)^{j}D^j(\phi_*-\phi_0)\in L^\infty(\mathcal  C),\quad \forall\ j\in\Z_{\geq0},\\
		&(T-t)^{\frac{2\beta}{p-1}+j}D^j(\rho_*-\rho_0)=(T-t)^{\beta-1}(T-t)^{\frac{2\beta}{p-1}+j-\beta+1}D^j(\rho_*-\rho_0)\in L^\infty(\mathcal  C),\quad \forall\ j\in\Z_{\geq0},
	\end{align*}which along with with \eqref{T2} implies \eqref{Eq.phi_*_rho_*_1}.\smallskip
	
	{\bf Step 3.} Proof of \eqref{Eq.phi_*_rho_*_2}. By \eqref{T1}, we have (for some $ C_1>0$)\begin{align*}
		&(T-t)(|\partial_t(\phi_*-\phi_0)|+|D_x(\phi_*-\phi_0)|)+(T-t)^{\frac{2\beta}{p-1}-\beta+1}|\rho_*-\rho_0|\leq C_1\quad \text{in}\,\, \mathcal  C.
	\end{align*}Now we take $c_0\in(0,T)$ such that $c_0+c_0^{\beta-1}C_1\leq \widetilde{c}$, where the existence of such a $c_0$ is ensured by $\beta>1$ {and $ \widetilde{c}>0$.} Then for $(t,x)\in\mathcal C,$ $T-t<c_0$,  we get by \eqref{T3}  that {(as $T-t>0$)}
	\begin{align*}
		(T-t)^{\beta}(\pa_t\phi_*-|D_x\phi_*|)
		\geq& (T-t)^{\beta}(\pa_t\phi_0-|D_x\phi_0|)-(T-t)^{\beta}(|\partial_t(\phi_*-\phi_0)|+|D_x(\phi_*-\phi_0)|)\\
		\geq& \widetilde{c}-(T-t)^{\beta-1}C_1\geq \widetilde{c}-c_0^{\beta-1}C_1\geq c_0,	
		\end{align*}
	and 
	\begin{align*}
		(T-t)^{\frac{2\beta}{p-1}}\rho_*\geq (T-t)^{\frac{2\beta}{p-1}}&\rho_0-(T-t)^{\frac{2\beta}{p-1}}|\rho_*-\rho_0|
		\geq \widetilde{c}-(T-t)^{\beta-1}C_1\geq \widetilde{c}-c_0^{\beta-1}C_1\geq c_0.
	\end{align*}	
	
	{\bf Step 4.} Proof of \eqref{EJ}. {We fix $ \lambda>0$, $j\in\Z_{\geq0} $. As} $\lim\limits_{N\to+\infty}k_N=+\infty$, there exists $N_1\in\Z_{\geq N_0}$ such that $ k_N>\lambda+j$ for
	$N\in\Z_{\geq N_1}$. Then by \eqref{T0} and $A_n{T}_n^{k_n}\leq 1$,  we have
	\begin{align*}
		|D^jE_{*}(t,x)|+|D^jJ_{*}(t,x)|\leq A_N(T-t)^{2k_N}\leq A_N{T}_N^{k_N}(T-t)^{k_N}\leq (T-t)^{k_N}\leq (T-t)^{\lambda},
	\end{align*}for $T-t\in [T_{N+1},T_N]$, $(t,x)\in \mathcal C$, $N\in\Z_{\geq N_1}$. As $\lim_{N\to+\infty}T_N=0$, we have
	\begin{align}\label{EJN1}
		|D^jE_{*}(t,x)|+|D^jJ_{*}(t,x)|\leq  (T-t)^{\lambda},\quad \forall\ T-t\in (0,T_{N_1}],\ (t,x)\in \mathcal C.
	\end{align}As $\rho_*, \phi_*\in C^\infty([0, T)\times\R^d)$ (see footnote \ref{footnote.finite_sum}),  we have $E_*, J_*\in C^\infty([0,T)\times\R^d)$ by \eqref{EJ1}. Thus, there exists a constant $ C(j,T_{N_1})>0$ such that\begin{align*}
		|D^jE_{*}(t,x)|+|D^jJ_{*}(t,x)|\leq  C(j,T_{N_1}),\quad \forall\ T-t\in [T_{N_1},T],\ |x|\leq 2T.
	\end{align*}Then (recall that $\mathcal C=\left\{(t,x)\in[0, T)\times\R^d: |x|<2(T-t)\right\}$)\begin{align*}
		|D^jE_{*}(t,x)|+|D^jJ_{*}(t,x)|\leq  C(j,T_{N_1})T_{N_1}^{-\lambda}(T-t)^{\lambda},\quad \forall\ T-t\in [T_{N_1},T],\ (t,x)\in \mathcal C,
	\end{align*}which along with with \eqref{EJN1} implies \eqref{EJ}.
\end{proof}

\subsection{Proof of main lemmas}\label{Subsec.Proof_Lem3}

We define the following auxiliary spaces
\begin{align}
	&\mathscr Y_0:=\left\{f\in C^\infty([0,+\infty)): f'=0\ \text{in}\ [0,1]\cup[2,+\infty)\right\},\\
	&\mathscr X_{\la}^*:=\left\{f(t,x,s)=\sum_{j=0}^n f_j(t,x)\eta_j\left(\frac{T-t}{s}\right): n\in\Z_{\geq 0}, f_j\in \mathscr X_{\la}, \eta_j\in\mathscr Y_0,\ \forall\ j\right\}.
\end{align}Note that $ \eta\in \mathscr Y_0$, $ \mathscr Y_0$ is a ring, and $ \mathscr X_{\la}^*$ is a linear vector space.  

\begin{lemma}\label{Lem.X_la*}
	\begin{enumerate}[(i)]
		\item  Let $\la,\mu\in\C$, $f\in\mathscr X_\la^*, g\in\mathscr X_\mu^*$. Then $\Box f\in\mathscr X_{\la-2}^*$, $fg\in\mathscr X_{\la+\mu}^*$, $\pa^\alpha f\pa_\alpha g\in \mathscr X_{\la+\mu-2}^*$.
		\item Let $\la,\mu\in\R$ and $j\in\Z_{\geq 0}$ be such that $\la>j+\mu$. If $f\in\mathscr X_\la^*$, then $(T-t)^{-\mu}D^jf\in L^\infty(\mathcal C\times(0,1])$.
	\end{enumerate}
\end{lemma}Here the operators $\Box $, $\pa_\alpha $ and $D$ are only acted on $(t,x)$ and not on $s$.

\begin{lemma}\label{lem7}
	\begin{enumerate}[(i)]
		\item  Let $\la,\mu\in\C$, $f\in\mathscr X_\la^*, g\in\mathscr X_\mu^*$. Then $\Delta f\in\mathscr X_{\la-2}^*$, $\partial_t f\in\mathscr X_{\la-1}^*$, $fg\in\mathscr X_{\la+\mu}^*$.
		\item Let $\la,\mu\in\R$ and $j\in\Z_{\geq 0}$ be such that $\la>j+\mu$. If $f\in\mathscr X_\la^*$, then $(T-t)^{-\mu}D_x^jf\in L^\infty(\mathcal C\times(0,1])$.
	\end{enumerate}
\end{lemma}\begin{proof}
By the definition of  $\mathscr X_{\la}^*$, it suffices to prove the result for $f(t,x,s)=f_1(t,x)\eta_1\big(\frac{T-t}{s}\big)$,
$g(t,x,s)=g_1(t,x)\widetilde{\eta}_1\big(\frac{T-t}{s}\big)$ for some $f_1\in \mathscr X_{\la},$ $g_1\in\mathscr X_\mu$, $\eta_1,\widetilde{\eta}_1\in\mathscr Y_0$.

In this case, $\Delta f(t,x,s)=\Delta f_1(t,x)\eta_1\big(\frac{T-t}{s}\big)$. By Lemma \ref{Lem.X_la_properties} (i), we have $\Delta f_1\in \mathscr X_{\la-2},$ thus $\Delta f\in \mathscr X_{\la-2}^*$.
We also have 
$$\partial_t f(t,x,s)=\partial_t f_1(t,x)\eta_1\big((T-t)/{s}\big)+(T-t)^{-1}f_1(t,x)\eta_2\big({(T-t)/{s}}\big)$$ 
with $ \eta_2(z)=-z\eta_1'(z)\in\mathscr Y_0$ (as $\eta_2=0 $, $\eta_2'=0 $ in $[0,1]\cup[2,+\infty) $). By Lemma \ref{Lem.X_la_properties} (i), we have $\partial_t f_1\in \mathscr X_{\la-1},$ $(T-t)^{-1} \in \mathscr X_{-1},$ $(T-t)^{-1} f_1\in \mathscr X_{\la-1},$
thus $\partial_t f\in\mathscr X_{\la-1}^*$.

In this case, $(fg)(t,x,s)=(f_1g_1)(t,x)(\eta_1\widetilde{\eta}_1)\big(\frac{T-t}{s}\big)$. By Lemma \ref{Lem.X_la_properties} (i), we have  $f_1g_1\in \mathscr X_{\la+\mu},$ as $\mathscr Y_0$ is a ring we have
$\eta_1\widetilde{\eta}_1\in\mathscr Y_0$, thus $fg\in\mathscr X_{\la+\mu}^*$. This completes the proof of (i).

Assume that $\alpha_1,\cdots,\alpha_d\in\Z_{\geq 0}$ are such that $\alpha_1+\cdots+\alpha_d=j$, then we have
		\begin{align*}
			&(T-t)^{-\mu}\pa_{x_1}^{\alpha_1}\cdots\pa_{x_d}^{\alpha_d}f(t,x,s)
			=(T-t)^{-\mu}\pa_{x_1}^{\alpha_1}\cdots\pa_{x_d}^{\alpha_d} f_1(t,x)\eta_1(({T-t})/{s}).
		\end{align*}Then (ii) follows from Lemma \ref{Lem.X_la_properties}  (ii) and $ \eta_{1}\in L^{\infty}([0,+\infty))$.
\end{proof}

\begin{proof}[Proof of Lemma \ref{Lem.X_la*}]
	By Lemma \ref{lem7} (i), we have $fg\in\mathscr X_{\la+\mu}^*$, $\Delta f\in\mathscr X_{\la-2}^*$, $\partial_t f\in\mathscr X_{\la-1}^*$, $\partial_t^2 f\in\mathscr X_{\la-2}^*$, thus $\Box f=-\partial_t^2f+\Delta f\in\mathscr X_{\la-2}^*$. As a consequence, we have $\Box (fg)\in \mathscr{X}_{(\lambda+\mu)-2}^* $, $(\Box f)g\in \mathscr{X}_{(\lambda-2)+\mu}^*$, $f\Box g\in \mathscr{X}_{\lambda+(\mu-2)}^* $, hence
$\pa^\alpha f\pa_\alpha g=(\Box (fg)-(\Box f)g-f\Box g)/2\in \mathscr{X}_{\lambda+\mu-2}^*$.

Assume that $\alpha_0,\alpha_1,\cdots,\alpha_d\in\Z_{\geq 0}$ are such that $\alpha_0+\cdots+\alpha_d=j$. By Lemma \ref{lem7} (i),  we have $\partial_t^{\alpha_0} f\in\mathscr X_{\la-\alpha_0}^*$. Then by Lemma \ref{lem7} (ii) and $\la-\alpha_0>j'+\mu$ (here $j'=\alpha_1+\cdots+\alpha_d=j-\alpha_0$), we have \begin{align*}
			&(T-t)^{-\mu}|\pa_t^{\alpha_0}\pa_{x_1}^{\alpha_1}\cdots\pa_{x_d}^{\alpha_d}f|\leq 
(T-t)^{-\mu}|D_x^{j'}\partial_t^{\alpha_0}f|\in L^\infty(\mathcal C\times(0,1]).
		\end{align*}
		
		This completes the proof.
\end{proof}

Now we are in a position to prove Lemma \ref{lem3}.

\begin{proof}[Proof of Lemma \ref{lem3}]

For $t\in[0, T), x\in\R^d, s\in(0,+\infty)$ and $N\in\Z_{\geq 0}$, let\begin{align*}
	&\rho_{N*}(t,x,s):=\sum_{n=0}^{N}\rho_n(t,x) b^n+\eta\left(\frac{T-t}{s}\right)\rho_{N+1}(t,x) b^{N+1},\\
	& \phi_{N*}(t,x,s):=\sum_{n=0}^N \phi_n(t,x)b^n+\eta\left(\frac{T-t}{s}\right)\phi_{N+1}(t,x) b^{N+1}.
\end{align*}
Then by \eqref{Eq.rho_*_phi_*}, {for all $N\in\Z_{\geq N_0-1}$,} we have\begin{align*}
	&\rho_{*}(t,x)=\rho_{N*}(t,x,T_{N+1}),\quad
	\phi_{*}(t,x)=\phi_{N*}(t,x,T_{N+1}),\quad \forall\ T-t\in [T_{N+1},T_N],\ (t,x)\in \mathcal C.
\end{align*}

Let
\begin{align}\label{EJN*}
	&E_{N*}=\rho_{N*}^{p}+\rho_{N*}\partial^{\alpha}\phi_{N*}\partial_{\alpha}\phi_{N*}-b\Box\rho_{N*},\quad
	J_{N*}= 2\partial^{\alpha}\rho_{N*}\partial_{\alpha}\phi_{N*}+\rho_{N*}\Box\phi_{N*}.
\end{align}Then by \eqref{EJ1}{, for all $N\in\Z_{\geq N_0-1}$,} we have\begin{align*}
	&E_{*}(t,x)=E_{N*}(t,x,T_{N+1}),\quad
	J_{*}(t,x)=J_{N*}(t,x,T_{N+1}),\quad \forall\ T-t\in [T_{N+1},T_N],\ (t,x)\in \mathcal C.
\end{align*}
Now \eqref{T0} is reduced to the proof of 
\begin{align}\label{T01}
	(T-t)^{-2k_N}(D^jE_{N*},D^jJ_{N*})\in L^{\infty}(\mathcal C\times(0,1]),\quad \forall\ j\in\Z\cap[0,k_N],\ N\in \Z_{\geq N_0-1}.
\end{align}

For $t\in[0, T), x\in\R^d, s\in(0,+\infty)$, let
\begin{align*}
		&\rho_{N,n}(t,x,s):=\left\{\begin{array}{ll}
		\rho_{n}(t,x) & n\in\Z\cap[0,N] \\
		\rho_{n}^*(t,x,s) &n=N+1 \\
		0 & n\in\Z_{\geq N+2}
	\end{array}
	\right.,\\ &\phi_{N,n}(t,x,s):=\left\{\begin{array}{ll}
		\phi_{n}(t,x) & n\in\Z\cap[0,N] \\
		\phi_{n}^*(t,x,s) &n=N+1 \\
		0 & n\in\Z_{\geq N+2}
	\end{array}
	\right..
\end{align*}
where 
\begin{align}\label{n*}
	&\rho_{n}^*(t,x,s):=\eta\left((T-t)/{s}\right)\rho_{n}(t,x) ,\  \phi_{n}^*(t,x,s):=\eta\left((T-t)/{s}\right)\phi_{n}(t,x),\quad \forall\ n\in\Z_{\geq 0}.
\end{align}
As $\rho_n\in\mathscr X_{\mu_n}$, $\phi_n\in\mathscr X_{\la_n}$ for all $n\in\Z_{\geq 0}$, we have
$\rho_n^*,\rho_{N,n}\in\mathscr X_{\mu_n}^*$, $\phi_n^*,\phi_{N,n}\in\mathscr X_{\la_n}^*$ for all $n,N\in\Z_{\geq 0}$. For $t\in[0, T), x\in\R^d, s\in(0,+\infty)$ and $N\in\Z_{\geq 0}$, 
we have
\begin{align*}
	\rho_{N*}(t,x,s):=\sum_{n=0}^{N+1}\rho_{N,n}(t,x,s) b^n,\quad
	\phi_{N*}(t,x,s):=\sum_{n=0}^{N+1}\phi_{N,n}(t,x,s)b^n.
\end{align*}
Then by \eqref{EJN*}, \eqref{Eq.recurrence_relation} and $\rho_{N,n}(t,x,s)=\rho_{n}(t,x) $ for $n\in\Z\cap[0,N] $, we have
\begin{align}\label{EJN2}
	E_{N*}(t,x,s)=\sum_{n=N+1}^{p(N+1)}E_{N,n}^*(t,x,s) b^n,\quad J_{N*}(t,x,s)=\sum_{n=N+1}^{2(N+1)} J_{N,n}^*(t,x,s)b^n,
\end{align}
with (note that $\rho_{N,n}(t,x,s)=0$ for $n\in\Z_{\geq N+2} $)
\begin{align*}
	E_{N,n}^*&=\sum_{n_1+\cdots+n_p=n}\rho_{N,n_1}\rho_{N,n_2}\cdots\rho_{N,n_p}+\sum_{n_1+n_2+n_3=n}\rho_{N,n_1}\pa^\alpha\phi_{N,n_2}\pa_\alpha\phi_{N,n_3}-\Box\rho_{N,n-1},\\
	J_{N,n}^*&=2\sum_{n_1+n_2=n}\pa^\alpha\rho_{N,n_1}\pa_\alpha\phi_{N,n_2}+\sum_{n_1+n_2=n}\rho_{N,n_1}\Box\phi_{N,n_2}.
\end{align*}
By Lemma \ref{Lem.X_la*} (i), we have
$E_{N,n}^*\in \mathscr{X}_{2n(\beta-1)-2p\beta/(p-1)}^*$, ${ J_{N,n}^*}\in \mathscr{X}_{(2n+1)(\beta-1)-2p\beta/(p-1)}^*$, where we have used the facts that $\mu_{n_1}+\cdots+\mu_{n_p}=\mu_{n-1}-2$ if $n_1+\cdots+n_p=n$, $\mu_{n_1}+\la_{n_2}+\la_{n_3}-2=\mu_{n-1}-2=2n(\beta-1)-2p\beta/(p-1)$ if $n_1+n_2+n_3=n$ and $\mu_{n_1}+\la_{n_2}-2=(2n-1)(\beta-1)-2\beta/(p-1)-2=(2n+1)(\beta-1)-2p\beta/(p-1)$ if $n_1+n_2=n$.

If $n\geq N+1$ and $j\in\Z\cap[0,k_N]$, then we get by \eqref{Eq.k_N} that
\begin{align*}
	&(2n+1)(\beta-1)-2p\beta/(p-1)-j>2n(\beta-1)-2p\beta/(p-1)-j\\
	&>2N(\beta-1)-2p\beta/(p-1)-j>3k_N-j\geq 2k_N.
\end{align*}
Thus, by  Lemma \ref{Lem.X_la*} (ii), we have $(T-t)^{-2k_N}D^jE_{N,n}^*\in L^\infty(\mathcal C\times(0,1])$ for $n\in\Z\cap[N+1,p(N+1)]$, and 
$(T-t)^{-2k_N}D^jJ_{N,n}^*\in L^\infty(\mathcal C\times(0,1])$ for $n\in\Z\cap[N+1,2(N+1)]$, which along with \eqref{EJN2}
implies \eqref{T01}. \end{proof}

Next we prove Lemma \ref{lem4}.

\begin{proof}[Proof of Lemma \ref{lem4}]
For $(t,x)\in [0,T)\times\R^d$, and any fixed $\{T_n\}_{n\geq 0}$ satisfying $T_n=T$ for $0\leq n<N_0$ and $0<T_n\le T_{n-1}/4$ for $n\ge N_0$, we define $\rho_{*}, \phi_*$ by \eqref{Eq.rho_*_phi_*}, then {($ (\rho_n^*,\phi_n^*)$ is defined in \eqref{n*})}
\begin{align*}
	&\rho_{*}(t,x)=\sum_{n=0}^{\infty}\rho_n^*(t,x,T_n) b^n, \quad
	\phi_{*}(t,x)=\sum_{n=0}^{\infty} \phi_n^*(t,x,T_n)b^n,\\
	&(\rho_{*}-\rho_0)(t,x)=\sum_{n=1}^{\infty}\rho_n^*(t,x,T_n) b^n,\quad
	(\phi_{*}-\phi_{0})(t,x)=\sum_{n=1}^{\infty} \phi_n^*(t,x,T_n)b^n.
\end{align*}
Recall that $\rho_n^*\in \mathscr X_{\mu_n}^*$ and $\phi_n^*\in \mathscr X_{\la_n}^*$ for all $n\in\Z_{\geq 0}$. By \eqref{Eq.la_n_mu_n}, we have
\begin{align*}
	\la_n-j&=(2n-1)(\beta-1)-j>(n-1)(\beta-1)-j,\quad \forall\ n\in\Z_{+},\\
	\mu_n-j&=2n(\beta-1)-\frac{2\beta}{p-1}-j>n(\beta-1)-\frac{2\beta}{p-1}-j,\quad \forall\ n\in\Z_{+}.
\end{align*}
Thus, by Lemma \ref{Lem.X_la*} (ii), for any $j,n\in\Z_{\geq 0}$, there exists a constant $B_{n,j}>0$, which is independent of the sequence $\{T_n\}_{n\geq 0}$, such that for all $(t,x)\in \mathcal C$, we have
\begin{align*}
	(T-t)^{j-(n-1)(\beta-1)}|D^j\phi_n^*(t,x,T_n)|+(T-t)^{\frac{2\beta}{p-1}+j-n(\beta-1)}|D^j\rho_n^*(t,x,T_n)|\leq B_{n,j},
\end{align*}
which gives (recalling that $\eta((T-t)/T_n))\neq 0$ implies $T-t\leq 2T_n$)
\begin{align*}
	(T-t)^{j}|D^j\phi_n^*(t,x,T_n)|+(T-t)^{\frac{2\beta}{p-1}+j-\beta+1}|D^j\rho_n^*(t,x,T_n)|\leq B_{n,j}(2T_n)^{(n-1)(\beta-1)}.
\end{align*}

Let $\widetilde{T}_n:=T$ for $0\leq n<N_0$ and for $n\geq N_0$ we let
\[\widetilde{B}_{n}:=2^n\max\limits_{0\leq j\leq n-N_0}B_{n,j}, \qquad\widetilde{T}_n:=\min\left(\widetilde{B}_{n}^{-1/[(n-1)(\beta-1)]}/2,\widetilde{T}_{n-1}/4\right).\]
Now we prove that $\{\widetilde T_n\}_{n\ge 0}$ is a desired sequence for Lemma \ref{lem4}.

Let $\{T_n\}_{n\geq0}$ be such that $T_n=T$ for $0\le n<N_0$ and $0<T_n\leq\min(\widetilde T_n, T_{n-1}/4)$ for $n\geq N_0$. Then
$\widetilde{B}_{n}(2\widetilde{T}_n)^{(n-1)(\beta-1)}\leq 1 $ for $n\geq N_0$. Fix $j\in\Z_{\geq 0}$. For any $(t,x)\in\mathcal C$, we have 
\begin{align*}
	&(T-t)^{j}|D^j(\phi_*-\phi_0)(t,x)|+(T-t)^{\frac{2\beta}{p-1}+j-\beta+1}|D^j(\rho_*-\rho_0)(t,x)|\\
	&\leq \sum_{n=1}^{\infty}\left((T-t)^{j}|D^j\phi_n^*(t,x,T_n)|+(T-t)^{\frac{2\beta}{p-1}+j-\beta+1}|D^j\rho_n^*(t,x,T_n)|\right)\\
	&\leq \sum_{n=1}^{\infty}B_{n,j}(2T_n)^{(n-1)(\beta-1)}\leq \sum_{n=1}^{\infty}B_{n,j}(2\widetilde{T}_n)^{(n-1)(\beta-1)}\\
	&\leq \sum_{n=1}^{N_0+j-1}B_{n,j}(2\widetilde{T}_n)^{(n-1)(\beta-1)}+\sum_{n=N_0+j}^{\infty}2^{-n}\widetilde{B}_{n}(2\widetilde{T}_n)^{(n-1)(\beta-1)}\\
	&\leq \sum_{n=1}^{N_0+j-1}B_{n,j}(2\widetilde{T}_n)^{(n-1)(\beta-1)}+\sum_{n=N_0+j}^{\infty}2^{-n}
	\leq \sum_{n=1}^{N_0+j-1}B_{n,j}(2\widetilde{T}_n)^{(n-1)(\beta-1)}+1,
\end{align*}which implies \eqref{T1}, as  the right hand side is a finite constant independent of $(t,x)\in \mathcal C$.
\end{proof}

Finally, we prove  Lemma \ref{lem5}.

\begin{proof}[Proof of Lemma \ref{lem5}]
By \eqref{Eq.phi_0_selfsimilar}, Lemma \ref{Lem.v_hat_rho} and Lemma \ref{Lem.X_la_properties} (iii), we obtain \eqref{T2}. It suffices to prove \eqref{T3}. By \eqref{Eq.phi_0_selfsimilar}, we have
\begin{equation*}
	\pa_t\phi_0=(T-t)^{-\beta}\left((\beta-1)\hat\phi_0(Z)+Z\hat\phi_0'(Z)\right),\quad \pa_r\phi_0=(T-t)^{-\beta}\hat\phi_0'(Z).
\end{equation*}
It follows from {\eqref{Eq.phi_0_rho_0}} that
\begin{equation}\label{Eq.hat_phi_0'}
\hat\phi_0'(Z)=\frac{(\beta-1)\hat\phi_0(Z)v(Z)}{1-Zv(Z)},\quad	Z\hat\phi_0'(Z)+(\beta-1)\hat\phi_0(Z)=\frac{(\beta-1)\hat\phi_0(Z)}{1-Zv(Z)}.
\end{equation}
Hence,
\begin{align}\label{Eq.3.12}
	(T-t)^\beta\left(\pa_t\phi_0-|D_x\phi_0|\right)=(\beta-1)\frac{\hat\phi_0(Z)(1-{|v(Z)|})}{1-Zv(Z)}.
\end{align}
Since $\beta>1$, $\hat\phi_0(Z)>0, v(Z)\in(-1,1), Zv(Z)<1$ for all $Z\in[0,+\infty)$ and $\hat\phi_0, v\in C^\infty([0,+\infty))$, we know that the right hand side of \eqref{Eq.3.12} is strictly positive and {continuous}. Thus, there exists $\tilde c_1\in(0, T)$ such that
\begin{equation*}
	\inf_{Z\in[0,2]}(\beta-1)\frac{\hat\phi_0(Z)(1-{|v(Z)|})}{1-Zv(Z)}>\tilde c_1.
\end{equation*}

On the other hand, by \eqref{Eq.phi_0_selfsimilar}, we have $(T-t)^{2\beta/(p-1)}\rho_0=\hat\rho_0(Z)$. As $\hat\rho(Z)>0$ and $\hat\rho\in C([0,+\infty))$, there exists $\tilde c_2\in(0, T)$ such that $\inf_{Z\in[0,2]}\hat\rho_0(Z)>\tilde c_2$. As a consequence, letting $\tilde c:=\min(\tilde c_1, \tilde c_2)\in(0, T)$, we have \eqref{T3}. 
\end{proof}

\section{The blow-up solution of nonlinear wave equation}\label{Sec.Sol_wave}

Fix $T=1$. Recall that $\mathcal C=\{(t,x)\in [0, T)\times\R^d: |x|<2(T-t)\}$. Let $w_*\in C^\infty(\mathcal C;\R), \Psi_*\in C^\infty(\mathcal C;\R)$ be such that both \eqref{Eq.Phi_*_w_*_1} and \eqref{Eq.Phi_*_w_*_2}  hold; moreover, \eqref{EJ} also holds for $E_*, J_*$ defined by \eqref{EJ2}.

\subsection{Derivation of the error equation} 

We construct a blow-up solution $u$ to $\Box u=|u|^{p-1}u$ of the form $u=(1+h)w_*\ee^{\ii\Phi_*}$, where $h$ is complex-valued. First of all, we deduce the equation for the error $h$.

\begin{lemma}\label{Lem.h_eq}
	Assume that $u=(1+h)w_*\operatorname{e}^{\operatorname{i}\Phi_*}$ solves $\Box u=|u|^{p-1}u$. Then $h$ satisfies 
	\begin{equation}\label{Eq.h_eq0}
		\Box h+2\operatorname{i} \pa^\alpha\Phi_*\pa_\alpha h+2\frac{\pa^\alpha w_*}{w_*}\pa_\alpha h-(p-1)w_*^{p-1}h_{\operatorname{r}}=w_*^{p-1}\varphi_1(h)+\frac{E_*-\operatorname{i}J_*}{w_*}(1+h),
	\end{equation}
	where $h_{\operatorname{r}}=\mathrm{Re}\,h=(h+\bar h)/2$ and 
	\beno
	\varphi_1(h)
=(|1+h|^{p-1}-1-(p-1)h_{\operatorname{r}})(1+h)+(p-1)h_{\operatorname{r}}h=O(|h|^2).
	\eeno
The converse is also true.
\end{lemma}

\begin{proof}
This is a brute force computation. If $u=(1+h)w_*\ee^{\ii\Phi_*}$, then for any $\alpha\in\Z\cap[0, d]$, we have
\begin{align*}
	\pa_\alpha u=\pa_\alpha hw_*\ee^{\ii\Phi_*}+(1+h)\pa_\alpha w_*\ee^{\ii\Phi_*}+\ii(1+h)w_*\ee^{\ii\Phi_*}\pa_\alpha\Phi_*.
\end{align*}
Hence,
\begin{align*}
	(\Box u)\ee^{-\ii\Phi_*}=&(\Box h+2\ii\pa_\alpha h\pa^\alpha\Phi_*)w_*+2\pa_\alpha h\pa^\alpha w_*\\
	&+(1+h)(\Box w_*+2\ii \pa_\alpha w_*\pa^\alpha\Phi_*-w_*\pa^\alpha\Phi_*\pa_\alpha\Phi_*+\ii\Box\Phi_*w_*)\\
	\xlongequal{\eqref{EJ2}}&(\Box h+2\ii\pa_\alpha h\pa^\alpha\Phi_*)w_*+2\pa_\alpha h\pa^\alpha w_*+(1+h)(w_*^p+\ii J_*-E_*),\\
	|u|^{p-1}u=&|1+h|^{p-1}(1+h)w_*^{p}\ee^{\ii\Phi_*}.
\end{align*}
By the definition of $\varphi_1$, we have $1+(p-1)h_\text{r}+h+\varphi_1(h)=(1+h)|1+h|^{p-1}$. Thus,
\begin{align*}
	&\Box u=|u|^{p-1}u\Longleftrightarrow\Box h+2\ii\pa^\alpha\Phi_*\pa_\alpha h+2\frac{\pa^\alpha w_*}{w_*}\pa_\alpha h+(1+h)\left(w_*^{p-1}+\frac{\ii J_*-E_*}{w_*}\right)\\
	&\qquad\qquad\qquad\qquad\qquad\qquad =\big(1+(p-1)h_\text{r}+h+\varphi_1(h)\big)w_*^{p-1}\\
	&\Longleftrightarrow\Box h+2\ii\pa^\alpha\Phi_*\pa_\alpha h+2\frac{\pa^\alpha w_*}{w_*}\pa_\alpha h-(p-1)w_*^{p-1}h_\text{r}=w_*^{p-1}\varphi_1(h)+\frac{E_*-\ii J_*}{w_*}(1+h).
\end{align*}

This completes the proof of Lemma \ref{Lem.h_eq}.
\end{proof}

We fix a bump function $\xi\in C_c^\infty(\R;[0,1])$ such that $\operatorname{supp}\xi\subset[-1,1]$ and $\xi_{[0,4/5]}=1$. We define the vector fields $X, Y$ by\footnote{\label{8}Here we explain the notations to avoid ambiguities. For a smooth function $f(t,x)$, we denote the action of the vector field $X$ on $f$ by $Xf$, i.e., $Xf=X_\alpha\pa^\alpha f=X^\alpha\pa_\alpha f$, where $X_0=\pa_t\Phi_*\xi, X^0=-X_0$ and $X_j=X^j=\pa_{j}\Phi_*\xi$ for $j\in\Z\cap[1,d]$. {The same clarification holds also for $Y$.} Moreover, in \eqref{Eq.X_Y_vector_fields}, although $D\Phi_*(t,x)$ is only defined for $(t,x)\in\mathcal C$, we just simply let $X(t,x)=0$ for $(t,x)\in ([0, T)\times\R^d)\setminus\mathcal C$, noting that $\xi(3|x|/(5(T-t)))=0$ near the boundary of $\mathcal C$. The same clarification holds also for $Y$, and $V_*, N_*$ in \eqref{Eq.V_*_N_*}.}
\begin{equation}\label{Eq.X_Y_vector_fields}
	X(t,x):=D\Phi_*(t,x)\xi\left(\frac{3|x|}{5(T-t)}\right),\quad Y(t,x):=\frac{Dw_*(t,x)}{w_*(t,x)}\xi\left(\frac{3|x|}{5(T-t)}\right)
\end{equation}
for $(t, x)\in [0, T)\times\R^d$. We also define the functions on $[0, T)\times\R^d$ by
\begin{equation}\label{Eq.V_*_N_*}
	V_*(t,x):=(p-1)w_*^{p-1}(t,x)\xi\left(\frac{3|x|}{5(T-t)}\right),\quad N_*(t,x):=\frac{E_*-\operatorname{i}J_*}{w_*}(t,x)\xi\left(\frac{3|x|}{4(T-t)}\right).
\end{equation}
Then $X, Y\in C^\infty ([0, T)\times\R^d; \R^{d+1})$ and $V_*\in C^\infty ([0, T)\times\R^d; \R), N_*\in C^\infty ([0, T)\times\R^d; \C)$. Moreover, we have
\begin{equation}\label{Eq.N_*_support}
	\operatorname{supp}_xN_*(t,\cdot)\subset \{x\in\R^d: |x|\leq 4(T-t)/3\}, \quad \forall\ t\in[0, T).
\end{equation}

Let $c_0\in(0, T)$ satisfy \eqref{Eq.Phi_*_w_*_2}. Let $\mathcal C_1:=\{(t,x)\in[T-c_0, T)\times\R^d: |x|\leq 4(T-t)/3\}\subset\mathcal C$. Using \eqref{Eq.Phi_*_w_*_1}, \eqref{Eq.Phi_*_w_*_2} and \eqref{EJ}, we have $X_0(t,x)>0, V_*(t,x)>0$ for all $(t,x)\in\mathcal C_1$. The following lemma gives more useful properties.
\begin{lemma}\label{Lem.X_Y_V_N_bound}
	There exists a constant $M>0$ such that
	\begin{align}
		&(T-t)\left(\frac{|DX|}{X_0}+|Y|+\frac{|DV_*|}{V_*}\right)\leq M,\label{Eq.X_Y_V_bound}\\
		\frac1M(T-t)^{-\beta}\leq& X_0\leq M(T-t)^{-\beta},\quad \frac1M(T-t)^{-2\beta}\leq V_*\leq M(T-t)^{-2\beta}\label{Eq.X_0_V_bound}
	\end{align}
	on $\mathcal C_1$. Moreover, for any $j\in\Z_{\geq 0}$, 
	\begin{equation}\label{Eq.X_Y_V_*_N_*_bound}
		(T-t)^{\beta+j}|D^j X|+(T-t)^{1+j}|D^j Y|+(T-t)^{2\beta+j}|D^j V_*|+(T-t)^{1+j}|D^j N_*|\in L^{\infty}(\mathcal C_1).
	\end{equation}
For any $j\in\Z_{\geq 0}$ and $\la>0$, there exists a constant $M_{j,\la}>0$ such that
	\begin{equation}\label{Eq.N_*_bound}
		|D^jN_*|\leq M_{j,\la}(T-t)^\la\quad\text{on}\quad\mathcal C_1.
	\end{equation}
\end{lemma}

\begin{proof}
On $\mathcal C_1$, we have
\begin{align}\label{V*}X=D\Phi_*,\quad Y=Dw_*/w_*,\quad V_*=(p-1)w_*^{p-1}.\end{align}
By \eqref{Eq.Phi_*_w_*_1}, we have $(T-t)^{\beta}X_0=(T-t)^{\beta}\pa_t\Phi_*\in L^\infty(\mathcal C_1)$ and $(T-t)^{2\beta}V_*\in L^\infty(\mathcal C_1)$. By \eqref{Eq.Phi_*_w_*_2}, we have $(T-t)^\beta X_0=(T-t)^{\beta}\pa_t\Phi_*\geq c_0, (T-t)^{2\beta}V_*=(p-1)(T-t)^{2\beta}w_*^{p-1}\geq (p-1)c_0^{p-1}$ on $\mathcal C_1$. 
This proves \eqref{Eq.X_0_V_bound}. 

It follows from \eqref{Eq.Phi_*_w_*_1} that $(T-t)^{\beta+1}|DX|=(T-t)^{\beta+1}|D^2\Phi_*|\in L^\infty(\mathcal C_1)$, hence by $(T-t)^\beta X_0\geq c_0$ on $\mathcal C_1$, we have $(T-t)|DX|/X_0\in L^\infty(\mathcal C_1)$. Similarly, by using \eqref{Eq.Phi_*_w_*_1} and \eqref{Eq.X_0_V_bound}, we get $(T-t)(|Y|+|DV_*|/V_*)\in L^\infty(\mathcal C_1)$. This proves \eqref{Eq.X_Y_V_bound}.

Next we prove \eqref{Eq.X_Y_V_*_N_*_bound} and \eqref{Eq.N_*_bound}. Recall the product rule: for smooth functions $f,g$ and $(\alpha_0, \alpha_1,\cdots, \alpha_d)\in\Z_{\geq 0}^{d+1}$, we have (see \cite{Hardy})
\begin{align*}
	\frac{\pa^{\alpha_0+\alpha_1+\cdots+\alpha_d}}{\pa_t^{\alpha_0}\pa_{x_1}^{\alpha_1}\cdots\pa_{x_d}^{\alpha_d}}(fg)=
&\sum_{j=0}^{\alpha_0}\sum_{j_1=0}^{\alpha_1}\cdots\sum_{j_d=0}^{\alpha_d}\binom{\alpha_0}{j_0}\binom{\alpha_1}{j_1}\cdots\binom{\alpha_d}{j_d}\\
	&\qquad\qquad\times\frac{\pa^{j_0+j_1+\cdots+j_d}f}{\pa_t^{j_0}\pa_{x_1}^{j_1}\cdots\pa_{x_d}^{j_d}}\cdot
\frac{\pa^{\alpha_0-j_0+\alpha_1-j_1+\cdots+\alpha_d-j_d}g}{\pa_t^{\alpha_0-j_0}\pa_{x_1}^{\alpha_1-j_1}\cdots\pa_{x_d}^{\alpha_d-j_d}}.
\end{align*}
Hence,
\begin{align}
	|D^n(fg)|&\lesssim_n\sum_{j=0}^n|D^jf||D^{n-j}g|,\quad |D_x^n(fg)|\lesssim_n\sum_{j=0}^n|D_x^jf||D_x^{n-j}g|,\quad \forall\ n\in\Z_{\geq 0},\label{Eq.Leibniz_1}\\
	|gD^nf|&\lesssim_n |D^n(fg)|+\sum_{j=0}^{n-1}|D^jf||D^{n-j}g|,\quad\forall\ n\in\Z_{+}.\label{Eq.Leibniz_2}
\end{align}

As $X=D\Phi_*$ on $\mathcal C_1$, we get by \eqref{Eq.Phi_*_w_*_1} that
\begin{equation}\label{Eq.X_bound}
	(T-t)^{\beta+j}|D^jX|\in L^\infty(\mathcal C_1),\quad\forall\ j\in\Z_{\geq 0}.
\end{equation} 
Now we use the induction argument to prove that
\begin{equation}\label{Eq.Y_claim}
	(T-t)^{1+j}|D^jY|\in L^\infty(\mathcal C_1),\quad\forall\ j\in\Z_{\geq 0}.
\end{equation}
By \eqref{Eq.X_Y_V_bound}, we know that \eqref{Eq.Y_claim} holds for $j=0$. Assume that \eqref{Eq.Y_claim} holds for all $j\in\Z\cap[0, n-1]$ for some $n\in\Z_{+}$. Note that $Dw_*=w_*Y$ on $\mathcal C_1$, hence by \eqref{Eq.Leibniz_2} we have
\[|w_*D^nY|\lesssim_n|D^n(Dw_*)|+\sum_{j=1}^{n}|D^jw_*||D^{n-j}Y|\quad\text{on}\quad \mathcal C_1.\]
Using \eqref{Eq.Phi_*_w_*_2}, \eqref{Eq.Phi_*_w_*_1} and the induction assumption, we obtain
\begin{align*}
	&(T-t)^{1+n}|D^nY|\leq \frac1{c_0}(T-t)^{1+n+\frac{2\beta}{p-1}}|w_*D^nY|\\
	&\lesssim_n (T-t)^{1+n+\frac{2\beta}{p-1}}|D^{1+n}w_*|+\sum_{j=1}^{n}(T-t)^{\frac{2\beta}{p-1}+j}|D^jw_*|(T-t)^{1+n-j}|D^{n-j}Y|\in L^\infty(\mathcal C_1).
\end{align*}
This proves \eqref{Eq.Y_claim}.

Now we prove that
\begin{equation}\label{Eq.V_claim}
	(T-t)^{\frac{2\beta}{p-1}m+j}\left|D^j(w_*^m)\right|\in L^\infty(\mathcal C_1),\quad \forall\ m\in\Z_{+},\ \forall\ j\in\Z_{\geq 0}.
\end{equation}
By \eqref{Eq.Phi_*_w_*_1}, we know that \eqref{Eq.V_claim} holds for $m=1$. We assume that \eqref{Eq.V_claim} holds for $m-1$, where $m\in\Z\cap[2,+\infty)$. By \eqref{Eq.Leibniz_1}, for $j\in\Z_{\geq 0}$ we have
\begin{align*}
	\left|D^j(w_*^m)\right|=\left|D^j(w_*^{m-1}w_*)\right|\lesssim_j \sum_{i=0}^j\left|D^i(w_*^{m-1})\right|\left|D^{j-i}w_*\right|,
\end{align*}
which gives
\begin{align*}
	(T-t)^{\frac{2\beta}{p-1}m+j}\left|D^j(w_*^m)\right|&\lesssim_j\sum_{i=0}^j(T-t)^{\frac{2\beta}{p-1}(m-1)+i}\left|D^i(w_*^{m-1})\right|(T-t)^{\frac{2\beta}{p-1}+j-i}\left|D^{j-i}w_*\right|\\
	&\in L^\infty(\mathcal C_1).
\end{align*}
By the induction argument, we have \eqref{Eq.V_claim}. Letting $m=p-1$ in \eqref{Eq.V_claim}, we get (using \eqref{V*}) 
\begin{equation}\label{Eq.V*_bound}
	(T-t)^{2\beta+j}|D^jV_*|\in L^\infty(\mathcal C_1),\quad\forall\ j\in\Z_{\geq 0}.
\end{equation}

Finally, we estimate $N_*$. Let $\tilde\xi(t,x):=\xi(3|x|/(4(T-t)))$. Then by Lemma \ref{Lem.X_la_properties} (iii), we have $(T-t)^{j}D^j\tilde\xi\in L^\infty(\mathcal C)$. Let $\tilde N_*:=(E_*-\ii J_*)\tilde\xi$, then $N_*=\tilde N_*/w_*$. By \eqref{Eq.Leibniz_1}, we have
\begin{align*}
	\left|D^j\tilde N_*\right|\lesssim_j \sum_{i=0}^j\left|D^i(E_*-\ii J_*)\right|\left|D^{j-i}\tilde \xi\right|,
\end{align*}
hence by \eqref{EJ}, for all $\la>0$ and $j\in\Z_{\geq 0}$ we have
\begin{align}\label{Eq.tilde_N_bound}
	(T-t)^{-\la}\left|D^j\tilde N_*\right|\lesssim_j\sum_{i=0}^j(T-t)^{-(\la+j-i)}\left|D^i(E_*-\ii J_*)\right|(T-t)^{j-i}\left|D^{j-i}\tilde \xi\right|\in L^\infty(\mathcal C).
\end{align}
Now we use the induction argument to prove that
\begin{equation}\label{Eq.N_*_claim}
	(T-t)^{-\la}|D^jN_*|\in L^\infty(\mathcal C_1),\quad\forall\ j\in\Z_{\geq 0},\ \forall\ \la>0.
\end{equation}
For $j=0$, \eqref{Eq.N_*_claim} follows from \eqref{Eq.tilde_N_bound} and \eqref{Eq.Phi_*_w_*_2}. Assume that \eqref{Eq.N_*_claim} holds for all $j\in\Z\cap[0, n-1]$ for some $n\in\Z_{+}$. As $\tilde N_*=w_*N_*$, we get by \eqref{Eq.Leibniz_2} that
\[\left|w_*D^nN_*\right|\lesssim_n \left|D^n\tilde N_*\right|+\sum_{j=1}^{n}|D^jw_*||D^{n-j}N_*|\quad\text{on}\quad \mathcal C_1.\]
Using \eqref{EJ}, \eqref{Eq.Phi_*_w_*_1}, \eqref{Eq.Phi_*_w_*_2} and the induction assumption, for any $\la>0$ we obtain
\begin{align*}
	&(T-t)^{-\la}|D^nN_*|\leq \frac1{c_0}(T-t)^{-\la+\frac{2\beta}{p-1}}\left|w_*D^nN_*\right|\\
	\lesssim_n& (T-t)^{\frac{2\beta}{p-1}}(T-t)^{-\la}\left|D^n\tilde N_*\right|+\sum_{j=1}^{n}(T-t)^{\frac{2\beta}{p-1}+j}\left|D^jw_*\right|(T-t)^{-(\la+j)}\left|D^{n-j}N_*\right|
\in L^\infty(\mathcal C_1).
\end{align*}
{This proves \eqref{Eq.N_*_claim} for $j=n$. By induction, we have \eqref{Eq.N_*_claim}, which is equivalent to \eqref{Eq.N_*_bound}.}

Taking $\la=1$ in \eqref{Eq.N_*_bound}, we get
\begin{equation}\label{Eq.N_*_bound1}
	(T-t)^{1+j}|D^jN_*|\leq T^{j+2}(T-t)^{-1}|D^jN_*|\in L^\infty(\mathcal C_1),\quad\forall\ j\in\Z_{\geq 0}.
\end{equation}
Therefore, \eqref{Eq.X_Y_V_*_N_*_bound} follows from \eqref{Eq.X_bound}, \eqref{Eq.Y_claim}, \eqref{Eq.V*_bound} and \eqref{Eq.N_*_bound1}. 
\end{proof}

\subsection{Energy estimates for the linearized wave equation}

\begin{lemma}\label{Lem.Energy_estimate}
	Let $T_*\in(0, c_0)$ and $h\in C_c^\infty([T-T_*, T)\times\R^d;\C)$ be such that $\operatorname{supp}_xh(t,\cdot)\subset\{x\in\R^d: |x|\leq 4(T-t)/3\}$ for all $t\in[T-T_*, T)$. We define the linear operator
	\begin{equation}\label{Eq.Lh} 
		\mathcal L h:= \Box h+2\operatorname{i} Xh+2Yh-V_*h_{\operatorname{r}},
	\end{equation}
	where $h_{\operatorname{r}}=(h+\bar h)/2$ and energy functionals
	\begin{align}\label{E0}
		E_0[h](t):&=\frac12\int_{\R^d}\left(|Dh(t,x)|^2+V_*(t,x)|h_{\operatorname{r}}(t,x)|^2\right)\,\mathrm dx,\quad\forall\ t\in[T-T_*, T),\\
		E_j[h]:&=E_0[D_x^jh],\quad\forall\ j\in\Z_{+}.\label{Eq.E_j}
	\end{align}
	Then there exist positive constants $M_1>1$ and $\{C_j\}_{j\in\Z_{\geq 0}}$ such that
	\begin{equation}\label{Eq.Energy_estimate}
		\sqrt{E_j[h](t)}\leq C_j\int_t^T\left(\frac{T-t}{T-s}\right)^{M_1}\sum_{i=0}^j\frac{\left\|D_x^i\mathcal Lh(s)\right\|_{L^2}}{(T-s)^{(j-i)\beta}}\mathrm ds,\quad\forall\ t\in[T-T_*, T),\ \forall\ j\in\Z_{\geq 0}.
	\end{equation}
\end{lemma}

\begin{proof}
	Let $T_*\in(0, c_0)$ and $h\in C_c^\infty([T-T_*, T)\times\R^d;\C)$ be such that $$\operatorname{supp}_xh(t,\cdot)\subset\{x\in\R^d: |x|\leq 4(T-t)/3\},\quad \forall\ t\in[T-T_*, T).$$
	We define the energy momentum tensor $T[h]$ by
	\begin{equation}\label{Eq.T[h]}
		T[h]_{\mu\nu}:=\Re\left(\pa_\mu h\overline{\pa_\nu h}\right)-\frac12m_{\mu\nu}\left(\pa^\alpha h\overline{\pa_\alpha h}+V_*h_\text{r}^2\right), \quad\forall\ \mu,\nu\in\Z\cap[0, d],
	\end{equation}
	where we have used the Einstein's convention in $\pa^\alpha h\overline{\pa_\alpha h}$. Then we have
	\begin{equation}\label{Eq.E_0}
		E_0[h](t)=\int_{\R^d}T[h]_{00}(t,x)\,\mathrm dx,\quad \forall\ t\in[T-T_*, T).
	\end{equation}
	We define
	\begin{equation}
		P_\mu^X[h]:=T[h]_{\mu\nu}X^\nu,\quad \forall\ \mu\in\Z\cap[0,d].
	\end{equation}
	
	Let's first claim that there exists a constant $\wt c_0>0$ such that
	\begin{align}\label{Eq.P_0^X}
		P_0^X[h]\leq \wt c_0T[h]_{00}X^0\leq0\ \text{on}\ \mathcal C_*:=\{(t,x)\in[T-T_*, T)\times\R^d: |x|\leq 4(T-t)/3\};
	\end{align}
	and there exists a constant $C_*>0$ such that
	\begin{align}\label{Eq.P_mu^X_div}
		\left|\pa^\mu P_\mu^X[h]\right|\leq C_*(T-t)^{-\beta}\left((T-t)^{-1}T[h]_{00}+\sqrt{T[h]_{00}}|\mathcal Lh|\right)\quad\text{on}\quad\mathcal C_*,
	\end{align}
	and moreover, for all $j\in\Z_{\geq 0}$,
	\begin{align}
		\left\|D_x^j(\mathcal LD_xh-D_x\mathcal Lh)\right\|_{L_x^2}\lesssim_j \sum_{i=0}^j(T-t)^{-\beta-1+i-j}\sqrt{E_i[h](t)}\quad\forall\ t\in[T-T_*, T),\label{Eq.L_D_commute}
	\end{align}
	where the implicit constants only depend on $X, Y, V_*, N_*$ (and they are independent of $h$).

Now we prove \eqref{Eq.Energy_estimate} by the  induction argument. We first consider $j=0$. For all $t\in[T-T_*, T)$, by \eqref{Eq.E_0}, $X^0=-X_0$, $\operatorname{supp}_xh(t,\cdot)\subset\{x\in\R^d: |x|\leq 4(T-t)/3\}$, 
\eqref{Eq.P_0^X} and \eqref{Eq.X_0_V_bound}, we have
	\begin{align}\label{Eq.E_0_tilde_E_0}
		E_0[h](t)\leq \int_{\R^d}\frac{-P_0^X[h](t,x)}{\wt c_0X_0(t,x)}\,\mathrm dx\leq \frac{M}{\wt c_0}(T-t)^{\beta}\int_{\R^d}-P_0^X[h](t,x)\,\mathrm dx.
	\end{align}	
	Let
	\[\wt E_0[h](t)=\int_{\R^d}-P_0^X[h](t,x)\,\mathrm dx\geq0, \quad\forall\ t\in[T-T_*, T).\]
	By the divergence theorem (recall that $\pa^0=-\pa_0=-\pa_t$), we get
	\begin{align*}
		\frac{\mathrm d}{\mathrm dt}\wt E_0[h](t)=\int_{\R^d}\pa^0P_0^X[h](t,x)\,\mathrm dx=\int_{\R^d}\pa^\mu P_\mu^X[h](t,x)\,\mathrm dx, \quad\forall\ t\in[T-T_*, T).
	\end{align*}
	Using \eqref{Eq.P_mu^X_div}, \eqref{Eq.E_0} and Cauchy's inequality, we obtain
	\begin{align*}
		&\left|\frac{\mathrm d}{\mathrm dt}\wt E_0[h](t)\right|\leq C_*(T-t)^{-\beta}\left((T-t)^{-1}E_0[h](t)+\sqrt{E_0[h](t)}\|\mathcal Lh(t)\|_{L_x^2}\right),\quad\forall\ t\in[T-T_*, T).
	\end{align*}
	Hence, by \eqref{Eq.E_0_tilde_E_0} and $h\in C_c^\infty([T-T_*, T)\times\R^d;\C)$, for all $t\in[T-T_*, T)$ we have
	\begin{align*}
		E_0[h](t)&\leq \frac M{\wt c_0}(T-t)^\beta\wt E_0[h](t)\leq \frac M{\wt c_0}(T-t)^\beta\int_t^T\left|\frac{\mathrm d}{\mathrm dt}\wt E_0[h](s)\right|\,\mathrm ds\\
		&\leq \frac{MC_*}{\wt c_0}(T-t)^\beta\int_t^T(T-s)^{-\beta}\left((T-s)^{-1}E_0[h](s)+\sqrt{E_0[h](s)}\|\mathcal Lh(s)\|_{L_x^2}\right)\,\mathrm ds.
	\end{align*}
	By Gr\"{o}nwall's lemma, we have
	\begin{align*}
		\sqrt{E_0[h](t)}\leq\frac{MC_*}{2\wt c_0}\int_t^T\left(\frac{T-t}{T-s}\right)^{\frac{MC_0/\wt c_0+\beta}{2}}\|\mathcal Lh(s)\|_{L_x^2}\,\mathrm ds,\quad\forall\ t\in[T-T_*, T).
	\end{align*}
	Letting $M_1:=\frac{MC_*/\wt c_0+\beta}{2}>0$, we know that \eqref{Eq.Energy_estimate} holds for $j=0$.

Let $n\in\Z_{+}$. We assume that \eqref{Eq.Energy_estimate} holds for all $j\in\Z\cap[0,n-1]$. Then by \eqref{Eq.Energy_estimate} for $j=n-1$ and \eqref{Eq.L_D_commute}, for $t\in[T-T_*, T)$ we have (also using \eqref{E0} and \eqref{Eq.E_j})
	\begin{align*}
		&\sqrt{E_n[h](t)}=\sqrt{E_{n-1}[D_xh](t)}\lesssim_n\int_t^T\left(\frac{T-t}{T-s}\right)^{M_1}\sum_{j=0}^{n-1}\frac{\|D_x^j\mathcal LD_xh(s)\|_{L_x^2}}{(T-s)^{(n-1-j)\beta}}\,\mathrm ds\\
		&\lesssim_n \int_t^T\left(\frac{T-t}{T-s}\right)^{M_1}\sum_{j=0}^{n-1}\frac{\|D_x^{j+1}\mathcal Lh(s)\|_{L_x^2}+\sum_{i=0}^j(T-s)^{-\beta-1+i-j}\sqrt{E_i[h](s)}}{(T-s)^{(n-1-j)\beta}}\,\mathrm ds\\
		&\lesssim_n \int_t^T\left(\frac{T-t}{T-s}\right)^{M_1}\sum_{j=1}^{n}\frac{\|D_x^{j}\mathcal Lh(s)\|_{L_x^2}}{(T-s)^{(n-j)\beta}}\,\mathrm ds+I_n(t),
		\end{align*}
		where
		\beno
		I_n(t):=\sum_{j=0}^{n-1}\sum_{i=0}^j\int_t^T\left(\frac{T-t}{T-s}\right)^{M_1}(T-s)^{-1+i-j-(n-j)\beta}\sqrt{E_i[h](s)}\,\mathrm ds.		
		\eeno

		For $T-T_*\leq t<s<T$, $j\geq i\geq0$ we have $0<T-s<T_*<c_0<T=1$ and $ (T-s)^{-1+i-j-(n-j)\beta}=(T-s)^{-1-(n-i)\beta+(j-i)(\beta-1)}\leq (T-s)^{-1-(n-i)\beta}$ (as $ \beta>1$). Then
		\begin{align*}
		&I_n(t)\leq n\sum_{i=0}^{n-1}\int_t^T\left(\frac{T-t}{T-s}\right)^{M_1}(T-s)^{-1-(n-i)\beta}\sqrt{E_i[h](s)}\,\mathrm ds.
	\end{align*}
	Using the induction assumption and Fubini's theorem, we have\begin{align*}
		I_n(t)&\lesssim_n\sum_{i=0}^{n-1}\sum_{j=0}^i
		\int_t^T\frac{\left((T-t)/(T-s)\right)^{M_1}}{(T-s)^{1+(n-i)\beta}}
\int_s^T\left(\frac{T-s}{T-\tau}\right)^{M_1}\frac{\|D_x^j\mathcal Lh(\tau)\|_{L_x^2}}{(T-\tau)^{(i-j)\beta}}\,\mathrm d\tau\,\mathrm ds\\
		&=\sum_{i=0}^{n-1}\sum_{j=0}^i\int_t^T\left(\frac{T-t}{T-\tau}\right)^{M_1}\frac{\|D_x^j\mathcal Lh(\tau)\|_{L_x^2}}{(T-\tau)^{(i-j)\beta}}\int_t^\tau\frac{\mathrm ds}{(T-s)^{1+(n-i)\beta}}\,\mathrm d\tau\\
		&\leq\sum_{i=0}^{n-1}\sum_{j=0}^i\int_t^T\left(\frac{T-t}{T-\tau}\right)^{M_1}\frac{\|D_x^j\mathcal Lh(\tau)\|_{L_x^2}}{(T-\tau)^{(i-j)\beta}}\frac{1}{(T-\tau)^{(n-i)\beta}}\,\mathrm d\tau\\
		&\leq n\sum_{j=0}^{n-1}\int_t^T\left(\frac{T-t}{T-\tau}\right)^{M_1}\frac{\|D_x^j\mathcal Lh(\tau)\|_{L_x^2}}{(T-\tau)^{(n-j)\beta}}\,\mathrm d\tau.
	\end{align*}
	Therefore, we obtain \eqref{Eq.Energy_estimate} for $j=n$. This proves \eqref{Eq.Energy_estimate} for all $j\in\Z_{\geq 0}$.

Thus, it remains to prove \eqref{Eq.P_0^X}, \eqref{Eq.P_mu^X_div} and \eqref{Eq.L_D_commute}. We start with
	\[P_0^X[h]=T[h]_{0\nu}X^\nu=T[h]_{00}X^0+\sum_{i=1}^dT[h]_{0i}X^i.\]
	On $\mathcal C_*\subset\mathcal C_1$, by \eqref{Eq.X_Y_vector_fields}, we have $-X^0=X_0=\pa_t\Phi_*$ and $X^i=X_i=\pa_i\Phi_*$ for $i\in\Z\cap[1,d]$, hence by \eqref{Eq.T[h]} and Cauchy's inequality,
	\begin{align*}
		\left|\sum_{i=1}^dT[h]_{0i}X^i\right|&\leq\sum_{i=1}^d|\pa_th||\pa_ih||\pa_i\Phi_*|\leq |\pa_th||D_xh||D_x\Phi_*|\leq \frac{|\pa_th|^2+|D_xh|^2}{2}|D_x\Phi_*|\\
		&\leq T[h]_{00}|D_x\Phi_*|.
	\end{align*}On the other hand, by \eqref{Eq.Phi_*_w_*_1} and \eqref{Eq.Phi_*_w_*_2}, there exists a constant $\wt c_0\in(0,1)$ such that
	\[\pa_t\Phi_*-|D_x\Phi_*|\geq c_0(T-t)^{-\beta}\geq \wt c_0\pa_t\Phi_*>0\quad\text{on}\quad \mathcal C_1.\]
	Thus, we have $|X|\leq|\pa_t\Phi_*|+|D_x\Phi_*|\leq 2\pa_t\Phi_*=2X_0 $ and
	\[\left|\sum_{i=1}^dT[h]_{0i}X^i\right|\leq T[h]_{00}|D_x\Phi_*|\leq T[h]_{00}(1-\wt c_0)\pa_t\Phi_*=T[h]_{00}(1-\wt c_0)X_0,\]
	hence
	\begin{align*}
		P_0^X[h]\leq T[h]_{00}X^0+T[h]_{00}(1-\wt c_0)X_0=\wt c_0T[h]_{00}X^0\leq0\quad\text{on}\quad \mathcal C_*.
	\end{align*}
	This proves \eqref{Eq.P_0^X}.	

As for \eqref{Eq.P_mu^X_div}, we compute
	\begin{align*}
		\pa^\mu T[h]_{\mu\nu}&=\Re\left(\Box h\overline{\pa_\nu h}\right)+\Re\left(\pa_\mu h\pa^\mu\overline{\pa_\nu h}\right)-\frac12\pa_\nu\left(\pa^\alpha h\overline{\pa_\alpha h}+V_*h_\text{r}^2\right)\\
		&=\Re\left(\Box h\overline{\pa_\nu h}\right)+\frac12\Re\pa_\nu\left(\pa_\mu h\overline{\pa^\mu h}\right)-\frac12\pa_\nu\left(\pa^\alpha h\overline{\pa_\alpha h}\right)-h_\text{r}\pa_\nu h_\text{r}V_*-\frac12 h_\text{r}^2\pa_\nu V_*\\
		&=\Re\left(\Box h\overline{\pa_\nu h}\right)-h_\text{r}\pa_\nu h_\text{r}V_*-\frac12 h_\text{r}^2\pa_\nu V_*
	\end{align*}
	for $\nu\in\Z\cap[0,d]$. Hence,
	\begin{align*}
		\pa^\mu P_\mu^X[h]&=T[h]_{\mu\nu}\pa^\mu X^\nu+(\pa^\mu T[h]_{\mu\nu})X^\nu\\
		&=T[h]_{\mu\nu}(\pi^X)^{\mu\nu}+\Re\left(\Box hX^\nu\overline{\pa_\nu h}\right)-h_\text{r}X^\nu\pa_\nu h_{\text r}V_*-\frac12h_\text{r}^2X^\nu\pa_\nu V_*\\
		&=T[h]_{\mu\nu}(\pi^X)^{\mu\nu}+\Re\left(\Box h\overline{Xh}\right)-V_*h_{\text r}Xh_{\text r}-\frac12h_{\text r}^2XV_*,
	\end{align*}
	where we have used the fact that $X^\nu$ is real-valued for $\nu\in\Z\cap[0, d]$, $Xh=X^\nu\pa_\nu h$ and we define
	\begin{align}
		(\pi^X)^{\mu\nu}:=\frac{\pa^\mu X^\nu+\pa^\nu X^\mu}2,\quad\forall\ \mu,\nu\in\Z\cap[0,d].
	\end{align}
	Hence, it follows from \eqref{Eq.Lh} that
	\begin{align}
		\pa^\mu P_\mu^X[h]&=T[h]_{\mu\nu}(\pi^X)^{\mu\nu}+\Re\left(\Box h\overline{Xh}\right)-V_*h_{\text r}Xh_{\text r}-\frac12h_{\text r}^2XV_*\nonumber\\
		&=T[h]_{\mu\nu}(\pi^X)^{\mu\nu}-\frac12h_{\text r}^2XV_*+\Re\left(\mathcal Lh\overline{Xh}\right)-2\Re\left(Yh\overline{Xh}\right).\label{Eq.P_mu^X_div1}
	\end{align}By \eqref{Eq.T[h]}, we have $|T[h]_{\mu\nu}|\leq T[h]_{00}$ for all $\mu,\nu\in\Z\cap[0,d]$ and $|Dh|^2\leq 2T[h]_{00}$. Thus, by {\eqref{Eq.P_mu^X_div1}}, $|X|\leq2 X_0$, \eqref{Eq.X_Y_V_bound} and \eqref{Eq.X_0_V_bound}, on $\mathcal C_*$ we have 
(note that $|XV_*|\leq |X||DV_*|$, $|Xh|\leq |X||Dh|$, $|Yh|\leq |Y||Dh|$, see footnote \ref{8})
	\begin{align*}
		\left|\pa^\mu P_\mu^X[h]\right|&\lesssim T[h]_{00}|DX|+T[h]_{00}X_0\frac{|DV_*|}{V_*}+|\mathcal Lh||X||Dh|+|Y||Dh||X||Dh|\\
		&\lesssim T[h]_{00}(T-t)^{-1}X_0+|\mathcal Lh|X_0\sqrt{T[h]_{00}}\\
		&\lesssim (T-t)^{-\beta}\left((T-t)^{-1}T[h]_{00}+\sqrt{T[h]_{00}}|\mathcal Lh|\right),
	\end{align*}
	which gives \eqref{Eq.P_mu^X_div}. 

Finally, we prove \eqref{Eq.L_D_commute}. By \eqref{Eq.Lh}, we have
	\begin{align*}
		\mathcal LD_xh-D_x\mathcal Lh=-2\ii D_xX^{\alpha}\partial_{\alpha}h-2D_xY^{\alpha}\partial_{\alpha}h+D_xV_*\cdot h_{\text r}.
	\end{align*}
	Let $j\in\Z_{\geq 0}$, by \eqref{Eq.Leibniz_1} and \eqref{Eq.X_Y_V_*_N_*_bound}, for any $t\in[T-T_*, T)$ we have
	\begin{align*}
		\|D_x^j(D_xX^{\alpha}\partial_{\alpha}h)(t)\|_{L_x^2}&\lesssim_j\sum_{i=0}^j\|D_x^{j-i}D_xX(t)\|_{L_x^\infty}\|DD_x^{i}h(t)\|_{L_x^2}\\
		&\lesssim_j\sum_{i=0}^j(T-t)^{-\beta-1+i-j}\sqrt{E_i[h](t)}.
	\end{align*}
	Similarly, we have (recalling $\beta>1$)
	\begin{align*}
		\|D_x^j(D_xY^{\alpha}\partial_{\alpha}h)(t)\|_{L_x^2}\lesssim_j \sum_{i=0}^j(T-t)^{-2+i-j}\sqrt{E_i[h](t)}\lesssim_j\sum_{i=0}^j(T-t)^{-\beta-1+i-j}\sqrt{E_i[h](t)}.
	\end{align*}
	By \eqref{Eq.X_0_V_bound} and \eqref{Eq.X_Y_V_*_N_*_bound}, we have
	\begin{align*}
		\|D_x^j(D_xV_*\cdot h_r)\|_{L_x^2}&\lesssim_j\sum_{i=0}^j\|D_x^{j-i+1}V_*/\sqrt{V_*}\|_{L_x^\infty}\|\sqrt{V_*}D_x^ih_\text{r}\|_{L_x^2}\\
		&\lesssim_j\sum_{i=0}^j(T-t)^{-\beta-1+i-j}\sqrt{E_i[h](t)}.
	\end{align*}
	Hence, we get \eqref{Eq.L_D_commute}.
\end{proof}

\subsection{Solving the error equation}

\begin{lemma}\label{Lem.h_solution_property}
	There exists a constant  $c_2\in (0, c_0)$ that depends only on $X, Y, V_*, N_*$ such that for any $f\in C^\infty_c([T-c_0, T)\times\R^d;\C)$ satisfying $|D_x^jf|\leq |D_x^j N_*|$ for all $j\in\Z_{\geq 0}$, there is a solution $h\in C_c^\infty([T-c_2, T)\times \R^d;\C)$ to the error equation
	\begin{equation}\label{Eq.h_eq_f}
		\Box h+2\operatorname{i}Xh+2Yh-V_*h_{\operatorname{r}}-\frac1{p-1}V_*\varphi_1(h)-N_*h=f.
	\end{equation}
	Moreover, $\operatorname{supp}_xh(t,\cdot)\subset\{x\in\R^d: |x|\leq4(T-t)/3\}$ for all $t\in [T-c_2, T)$, and there exists a constant $C_{\Box}>0$ that depends only on $X, Y, V_*, N_*$ ($C_{\Box}$ does not depend on $f$) such that
	\begin{equation}
		|\Box h(t,x)|\leq C_\Box,\quad \forall\ t\in [T-c_2, T),\ \forall\  x\in\R^d,
	\end{equation}
	and for any $j\in \Z_{\geq 0}$, $\la>0$, there exists a constant $C_{j,\la}>0$ that depends only on $X, Y, V_*, N_*$ ($C_{j,\la}$ does not depend on $f$) such that
	\begin{equation}
		|D_x^jh(t,x)|+|\pa_tD_x^jh(t,x)|\leq C_{j,\la}(T-t)^\la,\quad \forall\ t\in [T-c_2, T),\ \forall\ x\in\R^d.
	\end{equation}
\end{lemma}

The proof is based on the following lemma.

\begin{lemma}\label{Lem.h_bootstrap}
	Let $f\in C_c^\infty([T-c_0, T)\times\R^d;\C)$ be such that $|D_x^jf|\leq |D_x^jN_*|$ for all $j\in\Z_{\geq 0}$. Let $T_*\in(0, c_0)$. Assume that $h\in C_c^\infty([T-T_*, T)\times\R^d;\C)$ solves \eqref{Eq.h_eq_f} on $[T-T_*, T)\times\R^d$; moreover, $\operatorname{supp}_xh(t,\cdot)\subset \{x\in\R^d: |x|\leq 4(T-t)/3\}$ for all $t\in[T-T_*, T)$ and
	\begin{equation}\label{Eq.h_bound_bootstrap}
		\|h(t,\cdot)\|_{L^\infty(\R^d)}\leq (T-t)^{2\beta-1},\quad\forall\ t\in[T-T_*, T).
	\end{equation}
	Then there exists a constant $C_{\Box}>0$ that depends only on $X, Y, V_*, N_*$ ($C_{\Box}$ does not depend on $f, T_*$) such that
	\begin{equation}\label{h1}
		|\Box h(t,x)|\leq C_\Box,\quad \forall\ t\in [T-T_*, T),\ \forall\  x\in\R^d,
	\end{equation}
	and for any $j\in \Z_{\geq 0}$, $\la>0$, there exists a constant $C_{j,\la}>0$ that depends only on $X, Y, V_*, N_*$ ($C_{j,\la}$ does not depend on $f, T_*$) such that
	\begin{equation}\label{Eq.h_small1}
		|D_x^jh(t,x)|+|\pa_tD_x^jh(t,x)|\leq C_{j,\la}(T-t)^\la,\quad \forall\ t\in [T-T_*, T),\ \forall\ x\in\R^d.
	\end{equation}
\end{lemma}

Now we present the proof  of Lemma \ref{Lem.h_solution_property}
\begin{proof}[Proof of Lemma \ref{Lem.h_solution_property}]
	Let $f\in C_c^\infty([T-c_0, T)\times\R^d;\C)$ be such that $|D_x^jf|\leq |D_x^jN_*|$ for all $j\in\Z_{\geq 0}$. We assume that $\varepsilon\in(0, c_0)$ satisfies $f(t,x)=0$ for all $(t,x)\in(T-\varepsilon, T)\times\R^d$. By the standard local well-posedness theory (Theorem 6.4.11 in \cite{Hormander}), there is a unique local solution $h\in C^\infty((T-T_+, T)\times \R^d;\C)$ to \eqref{Eq.h_eq_f} with $(h, \pa_th)|_{t=T-\varepsilon/2}=(0, 0)$, where $\varepsilon<T_+\leq c_0$ corresponds to the left life span of $h$; moreover, if $T_+<c_0$, then
	\begin{equation}\label{Eq.h_blow_up_crit}
		\limsup_{t\downarrow T-T_+}\|h(t,\cdot)\|_{L^\infty(\R^d)}=+\infty.
	\end{equation}
	By the uniqueness and $f(t,x)=0$ for all $(t,x)\in(T-\varepsilon, T)\times\R^d$, we have $h(t,x)=0$ for all $(t,x)\in(T-\varepsilon, T)\times\R^d$. Moreover, by $|f|\leq |N_*|$, \eqref{Eq.N_*_support} and the finite speed of propagation, we have
	\[\operatorname{supp}_xh(t,\cdot)\subset \{x\in\R^d: |x|\leq 4(T-t)/3\},\quad\forall\ t\in[T-T_+, T).\]
	
	Let $c_2\in(0, c_0)$ be such that $C_{0,2\beta}\cdot c_2<1/2$, where $C_{0,2\beta}>0$ is given by \eqref{Eq.h_small1}. Note that $c_2$ is independent of $f$ and $T_*$. We claim that $T_+\geq c_2$. We assume in contrary that $\varepsilon<T_+<c_2$. Let
	\begin{equation}
		\mathscr{E}:=\left\{T_0\in(0, T_+): \|h(t,\cdot)\|_{L^\infty(\R^d)}\leq (T-t)^{2\beta-1} \text{ for all }t\in[T-T_0, T)\right\}.
	\end{equation}
	Then $(0,\varepsilon)\subset\mathscr{E}$. Let $T_s:=\sup\mathscr{E}\in[\varepsilon, T_+]$. By \eqref{Eq.h_blow_up_crit}, we have $T_s<T_+$, hence $T_s\in\mathscr{E}$ and $T_s<T_+<c_2$. By \eqref{Eq.h_small1}, we have
	\begin{align*}
		|h(t,x)|\leq C_{0,2\beta}(T-t)^{2\beta}=C_{0,2\beta}(T-t)(T-t)^{2\beta-1}\leq C_{0,2\beta}c_2(T-t)^{2\beta-1}<\frac12(T-t)^{2\beta-1}
	\end{align*}
	for all $t\in [T-T_s, T)\subset[T-c_2, T)$. Thus, by the continuity we have $T_s+\delta\in\mathscr{E}$ for some $\delta>0$. This contradicts with $T_s=\sup\mathscr{E}$. Therefore, $T_+\geq c_2$ and $\|h(t,\cdot)\|_{L^\infty(\R^d)}\leq (T-t)^{2\beta-1}$ for all $t\in[T-c_2, T)$. Now Lemma \ref{Lem.h_bootstrap} (letting $T_*=c_2$) implies Lemma \ref{Lem.h_solution_property}.
\end{proof}

Let's complete the proof of Lemma \ref{Lem.h_bootstrap}.

\begin{proof}[Proof of Lemma \ref{Lem.h_bootstrap}]
	Assume that $h$ solves \eqref{Eq.h_eq_f}. Then {($\mathcal L$ is defined in \eqref{Eq.Lh})}
	\[\mathcal Lh=\frac1{p-1}V_*\varphi_1(h)+N_*h+f.\]
	
	We claim that for each $j\in\Z_{\geq 0}$, there exists a constant $\tilde C_j>0$ such that
	\begin{equation}\label{Eq.E_j_claim}
		\sqrt{E_j[h](t)}\leq \wt C_j\int_t^T\left(\frac{T-t}{T-s}\right)^{M_1}\sum_{i=0}^j\frac{\left\|D_x^i f(s)\right\|_{L_x^2}}{(T-s)^{(j-i)\beta}}\mathrm ds,\quad\forall\ t\in[T-T_*, T).
	\end{equation}
	By the definition of $\varphi_1$, we know that $\varphi_1$ is a polynomial on $(h, \overline{h})$ of the form $\varphi_1=\sum_{2\leq i+j\leq p}c_{i,j}h^i\overline{h}^j$, 
with $c_{i,j}\in\R$, thus
	\begin{equation}\label{Eq.varphi_1_bound}
		|\varphi_1(h)|\lesssim |h|^{2}+|h|^p,\quad\forall\ h\in\C.
	\end{equation}
	Hence, $\operatorname{supp}_x\varphi_1(h)(t,\cdot)\subset\{x\in\R^d: |x|\leq 4(T-t)/3\}$ for all $t\in[T-T_*, T)$. For $j\in\Z_{\geq 0}$ and $t\in [T-T_*, T)$, by \eqref{Eq.Leibniz_1}, \eqref{Eq.X_Y_V_*_N_*_bound}
 and Poincaré's inequality, we have 
	\begin{align*}
		\left\|D_x^j(V_*\varphi_1(h))(t)\right\|_{L_x^2}&\lesssim_j \sum_{i=0}^j\|D_x^{j-i}V_*(t)\|_{L_x^\infty}\|D_x^i\varphi_1(h)(t)\|_{L_x^2}\\
		&\lesssim_j \sum_{i=0}^j(T-t)^{-(2\beta+j-i)}(T-t)^{j-i}\|D_x^j\varphi_1(h)(t)\|_{L_x^2}\\&\lesssim_j (T-t)^{-2\beta}\|D_x^j\varphi_1(h)(t)\|_{L_x^2}\end{align*}
Using the classical  product estimate,
\begin{equation}\label{eq:prod}
\|D_x^n(fg)\|_{L_x^2}\lesssim_n \|f\|_{L_x^\infty}\|D_x^ng\|_{L_x^2}+\|g\|_{L_x^\infty}\|D_x^nf\|_{L_x^2},\quad\forall\ n\in\Z_{\geq0},
\end{equation}
and \eqref{Eq.h_bound_bootstrap}, we infer 
\begin{align*}
		&\|D_x^n(h^i\overline{h}^j)\|_{L_x^2}\lesssim_{n,i,j} \|h\|_{L_x^\infty}^{i+j-1}\|D_x^nh\|_{L_x^2},\quad\forall\ n,i,j\in\Z_{\geq0},\ i+j\geq2,\\
&\|D_x^j\varphi_1(h)(t)\|_{L_x^2}\lesssim_j\left(\|h(t)\|_{L_x^\infty}+\|h(t)\|_{L_x^\infty}^{p-1}\right)\|D_x^jh(t)\|_{L_x^2}\lesssim_j (T-t)^{2\beta-1}\|D_x^jh(t)\|_{L_x^2},
	\\
		&\left\|D_x^j(V_*\varphi_1(h))(t)\right\|_{L_x^2}\lesssim_j(T-t)^{-2\beta} (T-t)^{2\beta-1}\|D_x^jh(t)\|_{L_x^2}=(T-t)^{-1}\|D_x^jh(t)\|_{L_x^2};
	\end{align*}
	Similarly, by \eqref{Eq.X_Y_V_*_N_*_bound} and Poincaré's inequality, we have
	\begin{align*}
		\left\|D_x^j(N_*h)(t)\right\|_{L_x^2}\lesssim_j(T-t)^{-1}\|D_x^jh(t)\|_{L_x^2}.
	\end{align*}
	Therefore, for each $j\in\Z_{\geq 0}$, there holds
	\begin{equation}\label{Eq.Lh_estimate}
		\left\|D_x^j\mathcal Lh(t)\right\|_{L_x^2}\lesssim_j (T-t)^{-1}\|D_x^jh(t)\|_{L_x^2}+\|D_x^jf(t)\|_{L_x^2}.
	\end{equation}
	
	By \eqref{Eq.Lh_estimate}, \eqref{Eq.Energy_estimate},  for any $j\in\Z_{\geq 0}$ and $t\in[T-T_*, T)$ we have
	\begin{align}
		\label{Eq.E_j_induction}\sqrt{E_j[h](t)}&\lesssim_j\int_t^T\left(\frac{T-t}{T-s}\right)^{M_1}\sum_{i=0}^j\frac{(T-s)^{-1}\|D_x^ih(s)\|_{L_x^2}+\|D_x^if(s)\|_{L_x^2}}{(T-s)^{(j-i)\beta}}\mathrm ds.
	\end{align}
	 It follows from Poincaré's inequality and $\operatorname{supp}_xh(t,\cdot)\subset\{x\in\R^d: |x|\leq 4(T-t)/3\}$ that
\begin{align}\label{Ej1}(T-t)^{-1}\|D_x^jh(t)\|_{L_x^2}\lesssim\|D_x^{j+1}h(t)\|_{L_x^2}\leq\sqrt{2E_j[h](t)},\quad \forall\ t\in [T-T_*, T),\ j\in\Z_{\geq0}.\end{align}
Here we also used the definitions of $E_0$ and $E_j$ in \eqref{E0} and \eqref{Eq.E_j}. Next we use the induction argument to prove \eqref{Eq.E_j_claim}.

For $j=0$,  by \eqref{Eq.E_j_induction} and \eqref{Ej1}, there exists a constant $C_0'>0$ satisfying
	\begin{align*}
		\sqrt{E_0[h](t)}\leq C_0'\int_t^T\left(\frac{T-t}{T-s}\right)^{M_1}\left(\sqrt{E_0[h](s)}+\|f(s)\|_{L_x^2}\right)\,\mathrm ds,\quad\forall\ t\in[T-T_*, T).
	\end{align*}
	By Gr\"{o}nwall's lemma, we get
	\begin{align*}
		(T-t)^{-M_1}\sqrt{E_0[h](t)}&\leq C_0'\int_t^T(T-s)^{-M_1}\ee^{C_0'(s-t)}\|f(s)\|_{L_x^2}\,\mathrm ds\\
		&\leq C_0'\ee^{C_0'T}\int_t^T(T-s)^{-M_1}\|f(s)\|_{L_x^2}\,\mathrm ds
	\end{align*}
	for all $t\in[T-T_*, T)$. This proves \eqref{Eq.E_j_claim} for $j=0$. Let $n\in \Z_{+}$, assume that \eqref{Eq.E_j_claim} holds for $j\in\Z\cap[0, n-1]$.
By \eqref{E0}, \eqref{Eq.E_j} and $ \beta>1$, we have 
\beno
(T-s)^{-1}\|D_x^nh(s)\|_{L_x^2}\leq (T-s)^{-1}\sqrt{2E_{n-1}[h](s)}\leq (T-s)^{-\beta}\sqrt{2E_{n-1}[h](s)},
\eeno
{for $ s\in[T-T_*, T)$.} Then by \eqref{Eq.E_j_induction} for $j=n$, \eqref{Ej1} for $j=i<n$, 
 and the induction assumption, we have (as $0<T-t\leq T_*<c_0<1$)
	\begin{align*}
		\sqrt{E_n[h](t)}&\lesssim_n\int_t^T\left(\frac{T-t}{T-s}\right)^{M_1}\left(\|D_x^nf(s)\|_{L_x^2}+\sum_{j=0}^{n-1}\frac{\sqrt{E_j[h](s)}+\|D_x^jf(s)\|_{L_x^2}}{(T-s)^{(n-j)\beta}}\right)\,\mathrm ds\\
		&\lesssim_n \int_t^T\left(\frac{T-t}{T-s}\right)^{M_1}\sum_{j=0}^n\frac{\|D_x^jf(s)\|_{L_x^2}}{(T-s)^{(n-j)\beta}}\,\mathrm ds+I_n(t),
		\end{align*}
where
\begin{align*}
		I_n(t):&=\sum_{j=0}^{n-1}\int_t^T\frac{\big((T-t)/(T-s)\big)^{M_1}}{(T-s)^{(n-j)\beta}}\int_s^T\left(\frac{T-s}{T-\tau}\right)^{M_1}\sum_{i=0}^j\frac{\|D_x^if(\tau)\|_{L_x^2}}{(T-\tau)^{(j-i)\beta}}\,\mathrm d\tau\,\mathrm ds\\
&=\sum_{j=0}^{n-1}\sum_{i=0}^j\int_t^T\left(\frac{T-t}{T-\tau}\right)^{M_1}\frac{\|D_x^if(\tau)\|_{L_x^2}}{(T-\tau)^{(j-i)\beta}}\int_t^{\tau}\frac{\mathrm ds}{(T-s)^{(n-j)\beta}}\,\mathrm d\tau\\
&\leq \sum_{j=0}^{n-1}\sum_{i=0}^j\int_t^T\left(\frac{T-t}{T-\tau}\right)^{M_1}\frac{\|D_x^if(\tau)\|_{L_x^2}}{(T-\tau)^{(j-i)\beta}}\frac{1}{(T-\tau)^{(n-j)\beta}}\,\mathrm d\tau\\
&\leq n\sum_{i=0}^{n-1}\int_t^T\left(\frac{T-t}{T-\tau}\right)^{M_1}\frac{\|D_x^if(\tau)\|_{L_x^2}}{(T-\tau)^{(n-i)\beta}}\,\mathrm d\tau.\end{align*}
Thus, \eqref{Eq.E_j_claim} holds for $j=n$. Therefore, by the the induction, \eqref{Eq.E_j_claim} holds for all $j\in\Z_{\geq 0}$.

As $|D_x^jf|\leq |D_x^jN_*|$ for all $j\in\Z_{\geq 0}$, by \eqref{Eq.N_*_support} and \eqref{Eq.N_*_bound}, for all $j\in\Z_{\geq 0}$ and $\la>0$ there exists a constant $\wt M_{j,\la}>0$ which is independent of $f$
and $T_*$ such that  $\|D_x^jf(t)\|_{L^2}\leq \wt M_{j,\la}(T-t)^\la$ for all $t\in[T-T_*, T)$. Using \eqref{Eq.E_j_claim}, \eqref{Eq.E_j} and \eqref{Ej1}, for all $j\in\Z_{\geq 0}$ and $\la>0$ there exists a constant
$M_{j,\la}'>0$ independent of $f$ and $T_*$ {(depending on $\wt M_{j,\la'} $ for some $\la'>\la $)} such that \begin{align*}
		\|h(t)\|_{H_x^j}+\|\pa_th(t)\|_{H_x^j}\leq M_{j,\la}'(T-t)^\la,\quad\forall\ t\in[T-T_*, T).
	\end{align*}By Sobolev's embedding theorem ($H_x^d(\R^d)\hookrightarrow L_x^\infty(\R^d)$), we have \eqref{Eq.h_small1}. It remains to prove \eqref{h1}. By \eqref{Eq.varphi_1_bound}, \eqref{Eq.h_bound_bootstrap}, \eqref{Eq.h_eq_f}, \eqref{Eq.h_small1}, \eqref{Eq.X_Y_V_*_N_*_bound} ($j=0$) and \eqref{Eq.N_*_bound} ($j=0, \la=1$), we know that there there exist constants $ C_2'>0,  C_3'>0$ such that\begin{align*}
		|\Box h|&\leq  C_2'(T-t)^{-2\beta}|D^{\leq 1}h|+|f|\leq  C_2'(T-t)^{-2\beta}(C_{0, 2\beta}+C_{1,2\beta})(T-t)^{2\beta}+|N_*|
		\leq  C_3'
	\end{align*}on $[T-T_*, T)\times\R^d$, which implies \eqref{h1}.
\end{proof}

\subsection{Solving nonlinear wave equation}

\begin{proof}[Proof of Proposition \ref{Prop.Solve_wave}]
	Let $\xi_1=1-\xi$, then $\xi_1|_{[0,4/5]}=0, \xi_1|_{[1,+\infty)}=1$. Let
	\[\varepsilon_n:=c_2/2^n,\quad f_n(t,x):=N_*(t,x)\xi_1((T-t)/\varepsilon_n), \quad \forall\ [T-c_0, T)\times\R^d,\ \forall\ n\in\Z_{+}.\]
	Then for each $n\in\Z_{+}$, we have $f_n\in C^\infty([T-c_0, T)\times\R^d;\C)$ and $\operatorname{supp}f_n\subset \{(t,x)\in [T-c_0, T-4\varepsilon_n/5]\times\R^d: |x|\leq 4(T-t)/3\}$, hence $f_n\in C_c^\infty([T-c_0, T)\times\R^d;\C)$; as $f_n$ equals to $N_*$ multiplied by a function in $t$ that takes values in $[0, 1]$, we have $|D_x^jf_n|\leq |D_x^jN_*|$ for all $j\in\Z_{+}$. By Lemma \ref{Lem.h_solution_property}, for each $n\in\Z_{+}$, there exists $h_n\in C_c^\infty([T-c_2, T)\times \R^d;\C)$ satisfying
	\begin{equation}\label{Eq.h_n_eq}
		\Box h_n+2\ii Xh_n+2Yh_n-V_*(h_n)_\text{r}-\frac1{p-1}V_*\varphi_1(h_n)-N_*h_n=f_n=N_*\xi_1\left(\frac{T-t}{\varepsilon_n}\right)
	\end{equation}
	on $[T-c_2, T)\times \R^d$. Moreover, for $n\in\Z_{+}$, we have
	\begin{align}\label{Eq.h_n_support}
		\operatorname{supp}_xh_n(t,\cdot)\subset\{x\in\R^d: |x|\leq4(T-t)/3\},\quad\forall\  t\in [T-c_2, T),
	\end{align}
	and for $j\in\Z_{\geq 0},\la>0$ we have
	\begin{align}
		|\Box h_n(t,x)|\leq C_\Box, \quad\forall\ t\in [T-c_2, T),\ \forall\  x\in\R^d,\label{Eq.h_n_Box}\\
		|D_x^jh_n(t,x)|+|\pa_tD_x^jh_n(t,x)|\leq C_{j,\la}(T-t)^\la,\quad \forall\ t\in [T-c_2, T),\ \forall\ x\in\R^d,\label{Eq.h_n_small}
	\end{align}
	where $C_\Box$ and $C_{j,\la}$ are given by Lemma \ref{Lem.h_solution_property}. By $h_n\in C_c^\infty([T-c_2, T)\times \R^d;\C)$, \eqref{Eq.h_n_support}, \eqref{Eq.h_n_Box}, \eqref{Eq.h_n_small} and the Arzelà–Ascoli theorem, there exists a subsequence of $\{h_n\}_{n=1}^\infty$, which is still denoted by $\{h_n\}_{n=1}^\infty$\footnote{Then $\{\varepsilon_n\}$ becomes its subsequence satisfying $\varepsilon_n\leq c_2/2^n$.}, such that $h_n\to h$ in $C^1([T-c_0, T)\times \R^d)$ for some $h\in C^1([T-c_0, T)\times \R^d;\C)$ and (here $D^{\leq 1}f:=(f,Df)$)
	\begin{equation}\label{Eq.h_n_Cauchy}
		\left\|D^{\leq 1}(h_n-h_{n+1})\right\|_{L^\infty([T-c_0, T)\times \R^d)}\leq 2^{-n},\quad \forall\ n\in\Z_{+}.
	\end{equation}
	Letting $j=2$, $\la=1$ in \eqref{Eq.h_n_small}, by \eqref{Eq.h_n_Cauchy} and the Gagliardo–Nirenberg inequality, we have
	\begin{equation}\label{Eq.h_n_Cauchy_C_1,2}
		\left\|D_xD(h_n-h_{n+1})\right\|_{L^\infty([T-c_0, T)\times \R^d)}\leq \tilde C_0 2^{-n/2},\quad \forall\ n\in\Z_{+}
	\end{equation}
	for some constant $\tilde C_0>0$ which is independent of $n\in\Z_{+}$.
	Letting $j=0$ and $\la=1$ in \eqref{Eq.h_n_small}, by the definition of $\varphi_1$, there exists a constant $\tilde C_1>0$ such that for all $n\in \Z_+$ and $(t,x)\in [T-c_2, T)\times\R^d$, we have
	\[|\varphi_1(h_n)-\varphi_1(h_{n+1})|\leq \tilde C_1|h_n-h_{n+1}|.\]
	Combining this with \eqref{Eq.h_n_eq}, \eqref{Eq.h_n_small}, \eqref{Eq.X_Y_V_*_N_*_bound} ($j=0$), and \eqref{Eq.N_*_bound} ($j=0, \la=1$), we know that there exist constants $\tilde C_2>0, \tilde C_3>0$ such that for all $n\in \Z_{+}$, we have\footnote{Here, we need to estimate $|f_n-f_{n+1}|$, which is achieved by combining \eqref{Eq.N_*_bound} ($j=0, \la=1$) and
	\begin{align*}
		\left|\xi_1\left(\frac{T-t}{\varepsilon_n}\right)-\xi_1\left(\frac{T-t}{\varepsilon_m}\right)\right|\leq \int_{\varepsilon_m}^{\varepsilon_n}\frac{T-t}{\varepsilon^2}\left|\xi_1'\left(\frac{T-t}{\varepsilon}\right)\right|\,\mathrm d\varepsilon\leq \|z^2\xi_1'(z)\|_{L^\infty}(T-t)^{-1}\varepsilon_n
	\end{align*} for all positive integers $m>n$, where we have used the fact that $\operatorname{supp}\xi_1'\subset[-1, 1]$.}
	\begin{align*}
		|\Box h_{n}-\Box h_{n+1}|&\leq \tilde C_2(T-t)^{-2\beta}\left|D^{\leq 1}(h_n-h_{n+1})\right|+\tilde C_2\varepsilon_n\\
		&\overset{\eqref{Eq.h_n_small}}{\underset{\eqref{Eq.h_n_Cauchy}}{\leq}}\tilde C_2(T-t)^{-2\beta}\min\left(2^{-n}, 2(C_{0, 4\beta}+C_{1,4\beta})(T-t)^{4\beta}\right)+\tilde C_2\varepsilon_n\\
		&\leq \tilde C_3(2^{-n/2}+\varepsilon_n)
	\end{align*}
	on $[T-c_2, T)\times\R^d$. Hence, by \eqref{Eq.h_n_Cauchy} and \eqref{Eq.h_n_Cauchy_C_1,2} we know that $\{h_n\}_{n=1}^\infty$ is Cauchy in $C^2([T-c_2, T)\times\R^d;\C)$, hence $h\in C^2([T-c_2, T)\times\R^d;\C)$ {and $h_n\to h$ in $C^2$}. Moreover, by \eqref{Eq.h_n_small}, \eqref{Eq.h_n_Cauchy} and the Gagliardo–Nirenberg inequality we know that $\{D_x^jD^{\leq 1}h_n\}$ is Cauchy in $L^\infty([T-c_2, T)\times\R^d)$ for all $j\in\Z_{\geq 0}$. Hence (also using \eqref{Eq.h_n_support}) $h(t,\cdot), \pa_th(t,\cdot)\in C_c^\infty(\R^d)$ with $\operatorname{supp}_xh(t,\cdot)\subset \{x\in\R^d: |x|\leq 4(T-t)/3\}$. Moreover, $h$ solves the equation (as $f_n\to N_*$)
	\begin{equation}\label{Eq.h_eq}
		\Box h+2\ii Xh+2Yh-V_*h_\text{r}-\frac1{p-1}V_*\varphi_1(h)-N_*h=N_*\quad\text{on}\quad [T-c_2, T)\times\R^d.
	\end{equation}
	By \eqref{Eq.h_n_small}, we have
	\[|h(t,x)|\leq C_{0,1}(T-t),\quad \forall\ (t,x)\in[T-c_2, T)\times\R^d.\]
	Let $c_1\in(0,c_2)\subset (0, T)$ be such that $C_{0, 1}c_1<1/2$, hence $|h(t,x)|<1/2$ for all $(t,x)\in[T-c_1, T)\times\R^d$. Let $$u(t,x):=(1+h(t,x))w_*(t,x)\xi\left(\frac{3|x|}{5(T-t)}\right)\exp\left(\ii\Phi_*(t,x)\xi\left(\frac{3|x|}{5(T-t)}\right)\right)$$
	for $(t,x)\in [T-c_1, T)\times\R^d$. Then $u\in C^2([T-c_1, T)\times\R^d;\C)$ with $u(t,\cdot), \pa_tu(t,\cdot)\in C_c^\infty(\R^d)$ and $\operatorname{supp}_xu(t,\cdot), \operatorname{supp}_x\pa_tu(t,\cdot)\subset \{x\in\R^d: |x|\leq 5(T-t)/3\}$ for $t\in[T-c_1, T)$. Moreover, on $\mathcal C_0:=\{(t,x)\in (T-c_1, T)\times\R^d: |x|\leq T-t\}$ we have $u=(1+h)w_*\text{e}^{\ii\Phi_*}$, and by \eqref{Eq.X_Y_vector_fields}, \eqref{Eq.V_*_N_*} and \eqref{Eq.h_eq}, we know that $h$ satisfies \eqref{Eq.h_eq0} on $\mathcal C_0$, hence by Lemma \ref{Lem.h_eq} we know that $\Box u=|u|^{p-1}u$ on $\mathcal C_0$. Finally, by $|h|<1/2$, \eqref{Eq.Phi_*_w_*_1} (for $j=0$) and \eqref{Eq.Phi_*_w_*_2} we have \eqref{Eq.u_bound} on $\mathcal C_0$. 
	
	This completes the proof of Proposition \ref{Prop.Solve_wave}.
\end{proof}

\section{The linearized operator $\LL$}\label{Sec.Surjective}

\subsection{Functional spaces}\label{Subsec.Functional_spaces}
In this subsection, we define some functional spaces consisting of smooth functions. Let $I\subset[0,+\infty)$ be an interval. We denote $I^2:=\{x^2: x\in I\}$ and\footnote{In \eqref{Space.C_e^infty}, ``e" stands for ``even"; In \eqref{Space.C_o^infty}, ``o" stands for ``odd". Please don't confuse $C_{\text{o}}^\infty(I)$ with ``$C_0^\infty(I)$".}
\begin{align}
	C_{\text{e}}^\infty(I):&=\left\{f\in C^\infty(I;\C): \exists \tilde f\in C^\infty(I^2) \text{ s.t. }f(Z)=\tilde f(Z^2),\quad \forall\ Z\in I\right\},\label{Space.C_e^infty}\\
	C_{\text{o}}^\infty(I):&=\left\{f\in C^\infty(I;\C): \exists \tilde f\in C^\infty(I^2) \text{ s.t. }f(Z)=Z\tilde f(Z^2),\quad \forall\ Z\in I\right\}.\label{Space.C_o^infty}
\end{align}
Then $ C_{\text{e}}^\infty(I)$ is a ring, and $ C_{\text{o}}^\infty(I)$ is a linear vector space.
Note that when $I=[0,+\infty)$, the definitions in \eqref{Space.C_e^infty}, \eqref{Space.C_o^infty} are the same as in \eqref{Ce}, \eqref{Co}. 
{For example, we have $f(Z)=Z\in C_{\text{o}}^\infty([0,+\infty))\setminus C_{\text{e}}^\infty([0,+\infty))$ and $f(Z)=Z^2\in C_{\text{e}}^\infty([0,+\infty))\setminus C_{\text{o}}^\infty([0,+\infty))$.}

\if0\begin{remark}\label{Rmk.C_e^infty}
	Let $I\subset[0,+\infty)$ be an interval. If $0\notin I$, then $C_{\text{e}}^\infty(I)=C^\infty(I)$; if $0\in I$, then\footnote{It is a classical fact that every smooth even function $f$ can be write in the form $f(x)=g(x^2)$ for some smooth function $g$, see \cite{Whitney}.}
		\[C_{\text{e}}^\infty(I)=\left\{f\in C^\infty(I):f^\text e\in C^\infty(I^\text{e})\right\},\]
		where $I^\text{e}:=\{Z\in\R: |Z|\in I\}$ and $f^\text{e}(Z):=f(|Z|), \forall\ Z\in I^\text{e}$.
		\end{remark}\fi

\begin{lemma}\label{Lem.smooth_R^d}
	Let $f\in C_{\operatorname{e}}^\infty([0,+\infty))$. Define $F(x)=f(|x|)$ for $x\in\R^d$, then $F\in C^\infty(\R^d)$.
\end{lemma}
\begin{proof}
	As $f\in C_{\text{e}}^\infty([0,+\infty))$, there exists a function $\tilde f\in C^\infty([0,+\infty))$ such that $f(Z)=\tilde f(Z^2)$ for all $Z\in[0,+\infty)$, hence $F(x)=\tilde f(|x|^2)$ for all $x\in \R^d$. The smoothness of $F$ follows from the smoothness of $\tilde f$ and $x\mapsto |x|^2$.
\end{proof}
We also have the following fundamental properties. Let $I\subset [0,+\infty)$ be an interval, then
\begin{align}
	f\in C_{\text{e}}^\infty(I)&\Longrightarrow f'\in C_{\text{o}}^\infty(I);\label{Eq.e_derivative_to_o}\\
	f\in C_{\text{o}}^\infty(I)&\Longrightarrow f'\in C_{\text{e}}^\infty(I);\label{Eq.o_derivative_to_e}\\
	f_1\in C_{\text{e}}^\infty(I), f_2\in C_{\text{e}}^\infty(I)&\Longrightarrow f_1f_2\in C_{\text{e}}^\infty(I);\label{Eq.ee_to_e}\\
	f_1\in C_{\text{o}}^\infty(I), f_2\in C_{\text{o}}^\infty(I)&\Longrightarrow f_1f_2\in C_{\text{e}}^\infty(I);\label{Eq.oo_to_e}\\
	f_1\in C_{\text{e}}^\infty(I), f_2\in C_{\text{o}}^\infty(I)&\Longrightarrow f_1f_2\in C_{\text{o}}^\infty(I).\label{Eq.eo_to_o}
\end{align}
Moreover, if $\Omega\subset\C$ is open, $\varphi\in C^\infty(\Omega;\C)$ (not necessary to be holomorphic), and $f\in C_\text{e}^\infty(I)$ with $f(Z)\in\Omega$ for all $Z\in I$, then the composition $\varphi\circ f\in C_\text{e}^\infty(I)$.
In particular,
\begin{align}
	f\in C_\text{e}^\infty(I)&\text{ with }f(Z)\neq 0 \ \forall\ Z\in I \Longrightarrow 1/f\in C_\text{e}^\infty(I),\label{Eq.reciprocal}\\
	&f\in C_\text{e}^\infty(I)\Longrightarrow \exp f\in C_\text{e}^\infty(I),\label{Eq.exponential}\\
    a\in\R, f\in C_\text{e}^\infty(I)&\text{ with }f(Z)>0 \text{ for all }Z\in I\Longrightarrow f^a\in C_\text{e}^\infty(I).\label{Eq.power}
\end{align}

\begin{lemma}\label{lem6}
	If $ f\in \mathscr X_0$, then $\pa_\tau  f, Z\pa_Z f, \pa_Z^2 f, \pa_Z f/Z\in\mathscr X_0$.
\end{lemma}\begin{proof}
By the definition of $\mathscr X_0$ in \eqref{X0}, it suffices to prove the result for $f=\hat f(Z)\tau^n$ for some $\hat f\in C^\infty_\text{e}([0,+\infty))$ and some $n\in\Z_{\geq 0}$.

As $\pa_\tau f=n\hat f(Z)\tau^{n-1}$, we have $\pa_\tau  f\in\mathscr X_0$ ($\pa_\tau f=0$ for $n=0$).

Note that $Z\pa_Z f=Z\hat f'(Z)\tau^n,$ $\pa_Z^2 f=\hat f''(Z)\tau^n,$ $\pa_Z f/Z=\frac{1}{Z}\hat f'(Z)\tau^n$. As $\hat f\in C^\infty_\text{e}([0,+\infty))$, by \eqref{Eq.e_derivative_to_o} we have
$\hat f'\in C^\infty_\text{o}([0,+\infty))$, then by the definitions of $C^\infty_\text{e}([0,+\infty))$ and $C^\infty_\text{o}([0,+\infty))$ we have
$Z\hat f'(Z),\frac{1}{Z}\hat f'(Z)\in C^\infty_\text{e}([0,+\infty))$, by \eqref{Eq.o_derivative_to_e} we have $\hat f''\in C^\infty_\text{e}([0,+\infty))$. Thus, $Z\pa_Z f$, $\pa_Z^2 f$, $\pa_Z f/Z\in\mathscr X_0$.
\end{proof}

Now we prove  Lemma \ref{Lem.X_la_properties}.

\begin{proof}
	\begin{enumerate}[(i)]
		\item Let $f(t,x)=(T-t)^\la\hat f(\tau, Z)$ for some $\hat f\in \mathscr X_0$ and $g(t,x)=(T-t)^\mu\hat g(\tau, Z)$ for some $\hat g\in \mathscr X_0$. Then
		$(fg)(t,x)=(T-t)^{\la+\mu}(\hat f\hat g)(\tau, Z)$. As $\mathscr X_0 $ is a ring, we have $\hat f\hat g\in \mathscr X_0$ and then $fg\in \mathscr X_{\la+\mu}$. It is direct to check that
		\begin{align*}
			\partial_t f=(T-t)^{\la-1}\left(\pa_\tau+Z\pa_Z-\la\right)\hat f,\quad \Delta f=(T-t)^{\la-2}\left(\pa_Z^2+(k/Z)\pa_Z\right)\hat f.
		\end{align*}
		Since $\hat f\in \mathscr X_0$, by Lemma \ref{lem6} we have $\pa_\tau \hat f, Z\pa_Z\hat f, \pa_Z^2\hat f, \pa_Z\hat f/Z\in\mathscr X_0$, which gives
		\[\left(\pa_\tau+Z\pa_Z-\la\right)\hat f\in\mathscr X_0,\quad \left(\pa_Z^2+(k/Z)\pa_Z\right)\hat f\in\mathscr X_0.\]
		Thus, $\partial_t f\in\mathscr X_{\la-1}$, $\Delta f\in\mathscr X_{\la-2}$. As a consequence, we have $\partial_t^2 f\in\mathscr X_{\la-2}$, $\Box f=-\partial_t^2 f+\Delta f\in \mathscr X_{\la-2}$ and (also using $fg\in \mathscr X_{\la+\mu} $ for $ f\in \mathscr X_{\la}$, $ g\in \mathscr X_{\mu}$)
		\[\Box(fg)\in\mathscr X_{(\la+\mu)-2}, \quad (\Box f)g\in\mathscr X_{(\la-2)+\mu}, \quad f\Box g\in\mathscr X_{\la+(\mu-2)},\]
		hence,
		\beno
		&&\pa^\alpha f\pa_\alpha g=[\Box(fg)-(\Box f)g-f\Box g]/2\in\mathscr X_{\la+\mu-2},\\
		&& \pa^\alpha(f\pa_\alpha g)=\pa^\alpha f\pa_\alpha g+f\Box g\in\mathscr X_{\la+\mu-2}.
		\eeno
		\item[(iii)] Let $\la,\mu\in\R$ and $j\in \Z_{\geq 0}$ be such that $\la\geq j+\mu$. Let $f(t,x)=(T-t)^\la\hat f(Z)$ for some $\hat f\in C^\infty_\text{e}([0,+\infty))$. Assume that $\alpha_0,\alpha_1,\cdots,\alpha_d\in\Z_{\geq 0}$ are such that $\alpha_0+\cdots+\alpha_d=j$.  We 
{only need} to prove that
		\begin{equation}\label{Eq.D^jf_bounded}
			(T-t)^{-\mu}\pa_t^{\alpha_0}\pa_{x_1}^{\alpha_1}\cdots\pa_{x_d}^{\alpha_d}f\in L^\infty(\mathcal C).
		\end{equation}
		Let $\tilde f(x):=\hat f(|x|)$ for $x\in\R^d$, then by Lemma \ref{Lem.smooth_R^d} we have $\tilde f\in C^\infty(\R^d)$. Let $j'=\alpha_1+\cdots+\alpha_d\in\Z\cap[0,j]$, and we let
		\[f_\alpha:=(-(\la-j')+x\cdot\nabla_x)(-(\la-j'-1)+x\cdot\nabla_x)\cdots(-(\la-j+1)+x\cdot\nabla_x)\pa_{x_1}^{\alpha_1}\cdots\pa_{x_d}^{\alpha_d}\tilde f.\]
		Then $f_\alpha\in C^\infty(\R^d)$ and one can check by direct computation that
		\[\pa_t^{\alpha_0}\pa_{x_1}^{\alpha_1}\cdots\pa_{x_d}^{\alpha_d}f(t,x)=(T-t)^{\la-j}f_\alpha(x/(T-t)),\quad \forall\ (t,x)\in[0, T)\times\R^d.\]
		As $|x/(T-t)|<2$ for $(t,x)\in\mathcal C$ and $\la-j-\mu\geq 0$, we have \eqref{Eq.D^jf_bounded}.
		
		\item Let $\la,\mu\in\R$ and $j\in \Z_{\geq 0}$ be such that $\la> j+\mu$. By the definitions of $\mathscr X_0$ and $\mathscr X_\la$, it suffices to prove $(T-t)^{-\mu}D^jf\in L^\infty(\mathcal C)$ for $f(t,x)=(T-t)^\la\hat f(Z)\tau^n$ for some $\hat f\in C^\infty_\text{e}([0,+\infty))$ and some $n\in\Z_{\geq 0}$. Let $P(\tau):=\tau^n$ and $\tilde P(t):=P(\tau)=P(-\ln(T-t))$. Then by the induction, for any $i\in\Z_{\geq 0}$, there is a polynomial $P_i(\tau)$ such that $\tilde P^{(i)}(t)=(T-t)^{-i}P_i(\tau)$. Hence,
		\begin{equation}\label{Eq.D^iP_bound}
			(T-t)^{i+\varepsilon}\tilde P^{(i)}(t)\in L^\infty([0,T)),\quad\forall\ i\in\Z_{\geq 0},\ \forall\ \varepsilon>0.
		\end{equation}
		Let $\tilde f(t,x):=(T-t)^\la\hat f(Z)$ for $(t,x)\in [0,T)\times\R^d$, then $f(t,x)=\tilde f(t,x)\tilde P(t)$ for $(t,x)\in [0,T)\times\R^d$, and by (iii) we have
		\begin{equation}\label{Eq.D^i_tilde_f_bound}
			(T-t)^{i-\la}D^i\tilde f\in L^\infty(\mathcal C),\quad \forall\ i\in\Z_{\geq 0}.
		\end{equation}
		Assume that $\alpha_0,\alpha_1,\cdots,\alpha_d\in\Z_{\geq 0}$ are such that $\alpha_0+\cdots+\alpha_d=j$. It suffices to prove
		\begin{equation}\label{Eq.D^jf_bounded1}
			(T-t)^{-\mu}\pa_t^{\alpha_0}\pa_{x_1}^{\alpha_1}\cdots\pa_{x_d}^{\alpha_d}(\tilde f(t,x)\tilde P(t))\in L^\infty(\mathcal C).
		\end{equation}
		By Leibnitz's product rule, we have
		\begin{align*}
			&(T-t)^{-\mu}\pa_t^{\alpha_0}\pa_{x_1}^{\alpha_1}\cdots\pa_{x_d}^{\alpha_d}(\tilde f(t,x)\tilde P(t))\\
			&=(T-t)^{-\mu}\sum_{i=0}^{\alpha_0}\binom{\alpha_0}{i}\tilde P^{(i)}(t)\pa_t^{\alpha_0-i}\pa_{x_1}^{\alpha_1}\cdots\pa_{x_d}^{\alpha_d}\tilde f(t,x)\\
			&=\sum_{i=0}^{\alpha_0}\binom{\alpha_0}{i}(T-t)^{\la-j-\mu+i}\tilde P^{(i)}(t)\cdot(T-t)^{j-i-\la}\pa_t^{\alpha_0-i}\pa_{x_1}^{\alpha_1}\cdots\pa_{x_d}^{\alpha_d}\tilde f(t,x).
		\end{align*}
		Then \eqref{Eq.D^jf_bounded1} follows from $\la-j-\mu>0$, \eqref{Eq.D^iP_bound} and \eqref{Eq.D^i_tilde_f_bound}.
	\end{enumerate}
	
	This completes the proof of Lemma \ref{Lem.X_la_properties}.
\end{proof}

\subsection{$\LL$ acting on $\mathscr X_\la$}

Let's first compute the linear operator $\LL_\la$ induced by $\LL$ acting on $\mathscr X_\la$.

\begin{lemma}\label{Lem.LL_la_formula}
	There exist real-valued $A_0(Z),\tilde B_0(Z),D_1(Z),D_2(Z)\in C_{\operatorname{e}}^\infty([0,+\infty))$ and
$\hat B_0\in C_{\operatorname{o}}^\infty([0,+\infty))$ with\begin{align}
		A_0(Z)=\hat\rho_0(Z)^2\frac{(1-Zv(Z))^2-\ell(v(Z)-Z)^2}{1-v(Z)^2}=\hat\rho_0(Z)^2\frac{\Delta_Z(Z, v(Z))}{Z(1-v(Z)^2)},\label{Eq.A_0}
	\end{align}
	such that if we define 
	\beno
	B_0(Z;\la):=Z^{-1}\tilde B_0(Z)+\la \hat B_0(Z), \quad D_0(Z;\la):=\lambda D_1(Z)+\lambda^2D_2(Z)
	\eeno
	and
	\begin{equation}\label{Lf}
		(\LL_\la f)(Z):=A_0(Z)f''(Z)+B_0(Z;\la)f'(Z)+D_0(Z;\la)f(Z),
	\end{equation}
then there hold {(here $\LL$ is defined in \eqref{Eq.L} and $\gamma:=4\beta/(p-1)+2=\beta(\ell-1)+2$)}
\beno
\LL((T-t)^{\lambda}f(Z))=(T-t)^{\lambda-\gamma}(\LL_\la f)(Z)\quad \text{for}\,\, f\in C_{\operatorname{e}}^\infty([0,+\infty)), \quad \la\in\C,
\eeno
 and
	\begin{align}
		&\qquad A_0(0)=1,\quad A_0(Z_1)=0, \quad A_0'(Z_1)<0, \label{Eq.A_0_non_degenerate}\\
		&\quad A_0(Z)>0 \ \forall\ Z\in[0, Z_1),\quad A_0(Z)<0\  \forall\ Z\in(Z_1,+\infty),\label{Eq.A_0_neq0}\\
		&\qquad\hat B_0(Z)>0 \text{ for all }Z>0, \quad \tilde B_0(0)=k\in \Z_+.\label{Eq.B_0}
	\end{align}
\end{lemma}

\begin{proof}

We first consider the  functions in the form of $(T-t)^{\la} f(\tau, Z)$, where $f$ is a smooth function and
\begin{equation}\label{Eq.ss_variables}
	\tau=\ln\frac1{T-t},\qquad Z=\frac r{T-t},\qquad r=|x|.
\end{equation}
Let $\gamma:=4\beta/(p-1)+2=\beta(\ell-1)+2$, and let $f=f(\tau, Z)$ and $\la\in\C$. Now we compute $\mathscr{L}\big((T-t)^\la f\big)$. 

We will use the following identities: for $\la,\mu\in\C, f=f(\tau, Z), g=g(\tau, Z)$,
\begin{align}
	&\pa^\alpha\left((T-t)^\la f\right)\pa_\alpha\left((T-t)^\mu g\right)\label{Eq.pa^al_pa_al}\\&\qquad=(T-t)^{\la+\mu-2}\Big[-(\pa_\tau f+Z\pa_Z f-\la f)
	(\pa_\tau g+Z\pa_Zg-\mu g)+\pa_Z f\pa_Zg\Big],\nonumber\\
	&\pa^\alpha\left((T-t)^\la f\pa_\alpha\left((T-t)^\mu g\right)\right)=(T-t)^{\la+\mu-2}\Big[-\big(\pa_\tau+Z\pa_Z-(\la+\mu-1)\big)\big(f(\pa_\tau g\label{Eq.divergence}\\
	&\qquad\qquad\qquad\qquad+Z\pa_Zg-\mu g)\big)+(\pa_Z+k/Z)(f\pa_Zg)\Big].\nonumber
\end{align}
Readers can check \eqref{Eq.pa^al_pa_al} and \eqref{Eq.divergence} by using direct computation.

By \eqref{Eq.phi_0_selfsimilar}, \eqref{Eq.divergence} {and $\gamma=4\beta/(p-1)+2$}, we have
\begin{align*}
	\pa^\alpha\left(\rho_0^2\pa_\alpha\left((T-t)^\la f\right)\right)=&(T-t)^{\la-\gamma}\big\{-\big(\pa_\tau+Z\pa_Z-(\la-\gamma+1)\big)\left[\hat\rho_0(Z)^2(\pa_\tau+Z\pa_Z-\la)f\right]\\
	&+\left(\pa_Z+k/Z\right)\left(\hat\rho_0(Z)^2\pa_Zf\right)\big\}.
\end{align*}
By \eqref{Eq.phi_0_selfsimilar} and \eqref{Eq.pa^al_pa_al}, we have
\begin{align*}
	\pa^{\tilde\alpha}\phi_0\pa_{\tilde\alpha}\left((T-t)^\la f\right)=(T-t)^{\la-\beta-1}\left[-(Z\pa_Z\hat\phi_0-(1-\beta)\hat\phi_0)(\pa_\tau+Z\pa_Z-\la)f+\pa_Z\hat\phi_0\pa_Zf\right].
\end{align*}
It follows from \eqref{Eq.hat_phi_0'} that
\begin{align*}
	-(Z\pa_Z\hat\phi_0-(1-\beta)\hat\phi_0)(\pa_\tau+Z\pa_Z-\la)f+\pa_Z\hat\phi_0\pa_Zf&=\frac{(\beta-1)\hat\phi_0}{1-Zv(Z)}\big(-\pa_\tau+(v-Z)\pa_Z+\la\big)f\\
	&=:g(\tau, Z),
\end{align*}
and then $\pa^{\tilde\alpha}\phi_0\pa_{\tilde\alpha}\left((T-t)^\la f\right)=(T-t)^{\la-\beta-1}g(\tau, Z)$. By \eqref{Eq.phi_0_selfsimilar} and \eqref{Eq.divergence}, we have
\begin{align*}
	\pa^\alpha&\left(\rho_0^{3-p}\pa_\alpha\phi_0\pa^{\tilde\alpha}\phi_0\pa_{\tilde\alpha}\left((T-t)^\la f\right)\right)=\pa^\alpha\left((T-t)^{\mu}\hat\rho_0^{3-p}g\pa_\alpha\left((T-t)^{1-\beta}\hat\phi_0\right)\right)\\
	&=(T-t)^{\la-\gamma}\Big\{-\big(\pa_\tau+Z\pa_Z-(\la-\g+1)\big)\left[\hat\rho_0^{3-p}g\big(Z\pa_Z\hat\phi_0-(1-\beta)\hat\phi_0\big)\right]\\
	&\qquad\qquad+(\pa_Z+k/Z)\left(\hat\rho_0^{3-p}g\pa_Z\hat\phi_0\right)\Big\},
\end{align*}
where $\mu=-\frac{2\beta}{p-1}(3-p)+\la-\beta-1$ satisfies $\mu+(1-\beta)-2=\la-\gamma$. By \eqref{Eq.phi_0_rho_0} and \eqref{Eq.hat_phi_0'},
\begin{align*}
	 \hat\rho_0(Z)^{3-p}\frac{(\beta-1)\hat\phi_0(Z)}{1-Zv(Z)}\big(Z\pa_Z\hat\phi_0-(1-\beta)\hat\phi_0\big)&=\hat\rho_0(Z)^2\hat\rho_0(Z)^{1-p}\frac{(\beta-1)^2\hat\phi_0(Z)^2}{(1-Zv(Z))^2}=\frac{\hat\rho_0(Z)^2}{1-v(Z)^2},\\
	\hat\rho_0(Z)^{3-p}\frac{(\beta-1)\hat\phi_0(Z)}{1-Zv(Z)}\pa_Z\hat\phi_0&=\frac{\hat\rho_0(Z)^2v(Z)}{1-v(Z)^2},
\end{align*}
thus
\begin{align*}
	\hat\rho_0^{3-p}g\big(Z\pa_Z\hat\phi_0-(1-\beta)\hat\phi_0\big)&=\frac{\hat\rho_0(Z)^2}{1-v(Z)^2}\big(-\pa_\tau+(v(Z)-Z)\pa_Z+\la\big)f,\\
	\hat\rho_0^{3-p}g\pa_Z\hat\phi_0&=\frac{\hat\rho_0(Z)^2v(Z)}{1-v(Z)^2}\big(-\pa_\tau+(v(Z)-Z)\pa_Z+\la\big)f.
\end{align*}
Therefore,
\begin{align*}
	&\pa^\alpha\left(\rho_0^{3-p}\pa_\alpha\phi_0\pa^{\tilde\alpha}\phi_0\pa_{\tilde\alpha}\left((T-t)^\la f\right)\right)\\
	=&(T-t)^{\la-\g}\Bigg\{-\big(\pa_\tau+Z\pa_Z-(\la-\g+1)\big)\left[\frac{\hat\rho_0(Z)^2}{1-v(Z)^2}\big(-\pa_\tau+(v(Z)-Z)\pa_Z+\la\big)f\right]\\
	&\qquad\qquad\qquad+\left(\pa_Z+\frac kZ\right)\left[\frac{\hat\rho_0(Z)^2v(Z)}{1-v(Z)^2}\big(-\pa_\tau+(v(Z)-Z)\pa_Z+\la\big)f\right]\Bigg\}.
\end{align*}
Finally, recall that $ \mathscr{L}(\phi)=\pa^\al\left(\rho_0^2\pa_\al\phi-\frac4{p-1}\rho_0^{3-p}\pa_\al\phi_0\pa^{\tilde\alpha}\phi_0\pa_{\tilde\alpha}\phi\right)$, $\frac4{p-1}=\ell-1$, we obtain
\begin{align*}
	\mathscr{L}\big((T-t)^\la& f\big)=(T-t)^{\la-\gamma}\Bigg\{-\big(\pa_\tau+Z\pa_Z-(\la-\gamma+1)\big)\left[\hat\rho_0(Z)^2(\pa_\tau+Z\pa_Z-\la)f\right]\\
	+&\left(\pa_Z+k/Z\right)\left(\hat\rho_0(Z)^2\pa_Zf\right)\\
	+&(\ell-1)\big(\pa_\tau+Z\pa_Z-(\la-\gamma+1)\big)\left[\frac{\hat\rho_0(Z)^2}{1-v(Z)^2}\big(-\pa_\tau+(v(Z)-Z)\pa_Z+\la\big)f\right]\\
	-&(\ell-1)\left(\pa_Z+\frac kZ\right)\left[\frac{\hat\rho_0(Z)^2v(Z)}{1-v(Z)^2}\big(-\pa_\tau+(v(Z)-Z)\pa_Z+\la\big)f\right]\Bigg\}.
\end{align*}

For any $\lambda\in\C$, we define a linear operator $\LL_\lambda$ 
by
\begin{align}
	(\LL_\lambda f)(Z):=&-\big(Z\pa_Z-(\lambda-\g+1)\big)\left[\hat\rho_0(Z)^2(Z\pa_Z-\lambda)f\right]+\left(\pa_Z+k/Z\right)\left(\hat\rho_0(Z)^2\pa_Zf\right)\label{Eq.LL_la}\\
	&+(\ell-1)\big(Z\pa_Z-(\la-\gamma+1)\big)\left[\frac{\hat\rho_0(Z)^2}{1-v(Z)^2}\big((v(Z)-Z)\pa_Z+\la\big)f\right]\nonumber\\
	&-(\ell-1)\left(\pa_Z+\frac kZ\right)\left[\frac{\hat\rho_0(Z)^2v(Z)}{1-v(Z)^2}\big((v(Z)-Z)\pa_Z+\la\big)f\right],\nonumber
\end{align}
where $f=f(Z)$ depends only on $Z\in[0,+\infty)$ (not on $ \tau$). Assume that $f=f(Z)=f(Z; \lambda)$ satisfies $(\LL_\lambda f)(Z)=g(Z)=g(Z; \lambda)$, then (here $\LL$ and $\LL_\lambda$ do not act on $ \lambda$)
\begin{equation}\label{Eq.L_L_la_relation}
	\LL\big((T-t)^\la f\big)=(T-t)^{\la-\g}g(Z).
\end{equation}

Now it is enough to prove that $\LL_\lambda$ defined in \eqref{Eq.LL_la} can be written in the form of \eqref{Lf} with $B_0(Z;\la)=Z^{-1}\tilde B_0(Z)+\la \hat B_0(Z)$, $D_0(Z;\la)=\lambda D_1(Z)+\lambda^2D_2(Z)$ and
$A_0(Z),\tilde B_0(Z),D_1(Z),$ $D_2(Z)\in C_{\operatorname{e}}^\infty([0,+\infty))$,
$\hat B_0\in C_{\operatorname{o}}^\infty([0,+\infty))$ satisfying \eqref{Eq.A_0}, \eqref{Eq.A_0_non_degenerate}, \eqref{Eq.A_0_neq0}, \eqref{Eq.B_0}.

Comparing the coefficients of
	   $ \partial_Z^j$ ($j=0,1,2$) in \eqref{Eq.LL_la} and \eqref{Lf}, we find
	\begin{align*}
		A_0(Z)&=\hat\rho_0(Z)^2(1-Z^2)+(\ell-1)\frac{\hat\rho_0(Z)^2}{1-v(Z)^2}Z(v(Z)-Z)-(\ell-1)\frac{\hat\rho_0(Z)^2v(Z)}{1-v(Z)^2}(v(Z)-Z)\\
		&=\hat\rho_0(Z)^2\left((1-Z^2)-(\ell-1)\frac{(v(Z)-Z)^2}{1-v(Z)^2}\right)=\hat\rho_0(Z)^2\frac{(1-Zv(Z))^2-\ell(v(Z)-Z)^2}{1-v(Z)^2},
	\end{align*}\begin{align*}
		&B_0(Z;\la)=-Z^2\pa_Z(\hat\rho_0^2)-Z\hat\rho_0^2+\la Z\hat\rho_0^2+(\la-\g+1)Z\hat\rho_0^2+\pa_Z(\hat\rho_0^2)+(k/Z)\hat\rho_0^2\\
		&+(\ell-1)Z\pa_Z\left(\frac{\hat\rho_0^2}{1-v^2}\right)(v-Z)+\frac{(\ell-1)\hat\rho_0^2}{1-v^2}\big(Zv'-Z-(\la-\g+1)(v-Z)+\la Z\big)\\
		&-(\ell-1)\pa_Z\left(\frac{\hat\rho_0^2v}{1-v^2}\right)(v-Z)-(\ell-1)\frac{\hat\rho_0^2v}{1-v^2}\big(v'-1+(k/Z)(v-Z)+\la\big),
	\end{align*}\begin{align*}
		D_0(Z;\la)=&-\big(Z\pa_Z-(\la-\g+1)\big)\left(-\la\hat\rho_0(Z)^2\right)+(\ell-1)\big(Z\pa_Z-(\la-\g+1)\big)\left(\la\frac{\hat\rho_0(Z)^2}{1-v(Z)^2}\right)\\
		&-(\ell-1)\left(\pa_Z+\frac kZ\right)\left(\la\frac{\hat\rho_0(Z)^2v(Z)}{1-v(Z)^2}\right).
	\end{align*}Then \eqref{Eq.LL_la} and \eqref{Lf} are equivalent
	and $A_0$ satisfies \eqref{Eq.A_0}.

By the expression of $ B_0(Z;\la) $, we have $B_0(Z;\la)=Z^{-1}\tilde B_0(Z)+\la \hat B_0(Z)$ with\begin{align*}
		\hat B_0(Z)&=2Z\hat\rho_0^2+\frac{(\ell-1)\hat\rho_0^2}{1-v^2}(2Z-v)-\frac{(\ell-1)\hat\rho_0^2v}{1-v^2}
		=\frac{2\hat\rho_0^2}{1-v^2}\big(Z(1-v^2)+(\ell-1)(Z-v)\big).
	\end{align*}\begin{align*}
		\tilde{B}_0(Z)=&\ k\hat \rho_0^2+(1-Z^2)Z\pa_Z(\hat\rho_0^2)-\gamma Z^2\hat\rho_0^2+(\ell-1)(v-Z)\bigg[Z^2\pa_Z\left(\frac{\hat\rho_0^2}{1-v^2}\right)\\
		&-Z\pa_Z\left(\frac{\hat\rho_0^2v}{1-v^2}\right)\bigg]+(\ell-1)\frac{\hat\rho_0^2}{1-v^2}\left(Z(Z-v)v'+(k+\g)Zv-\gamma Z^2-kv^2\right).
	\end{align*}
	By the expression of $ D_0(Z;\la)$, we have $D_0(Z;\la)=\lambda D_1(Z)+\lambda^2D_2(Z)$ with
	\begin{align*}
		D_1(Z)=&\big(Z\pa_Z+\g-1\big)\left(\hat\rho_0(Z)^2\right)+(\ell-1)\big(Z\pa_Z+\g-1\big)\left(\frac{\hat\rho_0(Z)^2}{1-v(Z)^2}\right)\\
		&-(\ell-1)\left(\pa_Z+\frac kZ\right)\left(\frac{\hat\rho_0(Z)^2v(Z)}{1-v(Z)^2}\right),\\
D_2(Z)=&-\hat\rho_0(Z)^2-(\ell-1)\frac{\hat\rho_0(Z)^2}{1-v(Z)^2}=-\hat\rho_0(Z)^2\frac{\ell-v(Z)^2}{1-v(Z)^2}.
	\end{align*}It remains to prove that $A_0(Z),\tilde B_0(Z),D_1(Z),$ $D_2(Z)\in C_{\operatorname{e}}^\infty([0,+\infty))$,
$\hat B_0\in C_{\operatorname{o}}^\infty([0,+\infty))$ and \eqref{Eq.A_0_non_degenerate}, \eqref{Eq.A_0_neq0}, \eqref{Eq.B_0}.

By Lemma \ref{Lem.v_hat_rho} we have $\hat\rho_0\in C_\text e^\infty([0,+\infty)), v\in C_\text o^\infty([0,+\infty))$; by \eqref{Eq.ee_to_e}  and \eqref{Eq.oo_to_e} we have $\hat\rho_0^2\in C_\text e^\infty([0,+\infty))$, $(v(Z)-Z)^2\in C_\text e^\infty([0,+\infty))$ and $1-v(Z)^2\in C_\text e^\infty([0,+\infty))$; by \eqref{Eq.oo_to_e} we have $Zv(Z)\in C_\text e^\infty([0,+\infty))$, hence $1-Zv(Z)\in C_\text e^\infty([0,+\infty))$, then using \eqref{Eq.ee_to_e} we get $(1-Zv(Z))^2\in C_\text e^\infty([0,+\infty))$. Therefore, by $v\in(-1, 1)$ (see Assumption \ref{Assumption}), \eqref{Eq.reciprocal} {and \eqref{Eq.A_0}} we have $A_0(Z)\in C_\text e^\infty([0,+\infty))$.

Similarly, by Lemma \ref{Lem.v_hat_rho} and \eqref{Eq.e_derivative_to_o}--\eqref{Eq.reciprocal} we have $\hat B_0(Z)\in C_\text o^\infty([0,+\infty))$ and $\tilde B_0(Z)$, $D_1(Z)$, $D_2(Z)\in C_{\operatorname{e}}^\infty([0,+\infty))$.

It follows from $\hat\rho_0(0)=1$ and $v(0)=0$ that $A_0(0)=1$. By Remark \ref{Rmk.v(Z)_properties} and $\hat\rho_0(Z)>0$ for all $Z\in[0,+\infty)$, we have $A_0(Z_1)=0$ and \eqref{Eq.A_0_neq0}. 
Let $\Delta_0(Z)=\Delta_Z(Z, v(Z))$, then by $\Delta_0(Z_1)=0$ and Remark \ref{Rmk.v(Z)_properties}, $A_0'(Z_1)=\hat\rho_0(Z_1)^2\Delta_0'(Z_1)/(Z_1(1-v(Z_1)^2))\neq 0$. This along with $A_0(Z_1)<A_0(Z)$ for all 
$Z\in[0, Z_1)$ implies $A_0'(Z_1)< 0$. So we have \eqref{Eq.A_0_non_degenerate}.

As $v(Z)\in(0, 1)$, $v(Z)<Z$ for all $Z>0$ (see Remark \ref{Rmk.v(Z)_properties}) and $\hat\phi_0(Z)>0$ for all $Z\in[0,+\infty)$, we have $\hat B_0(Z)>0$ for all $Z>0$; as $\hat\rho_0(0)=1$ and $v(0)=0$
we have $\tilde B_0(0)=k\in\Z_{+} $. This proves \eqref{Eq.B_0}.
\end{proof}

Next we compute the dual operator of $\LL_\la$. For any $\la\in\C$, we define an operator $\LL_\la^*$, called the \emph{dual operator} of $\LL_\la$, by
\begin{equation}\label{Eq.duality_def}
	\int_0^\infty (\LL_\la f)(Z)g(Z)Z^k\,\mathrm dZ=\int_0^\infty f(Z)(\LL_\la^*g)(Z)Z^k\,\mathrm dZ,\quad \forall\ f,g\in C_c^\infty((0,+\infty)).
\end{equation}

\begin{lemma}\label{Lem.duality}
	For any $\la\in\C$, we have $\LL_\la^*=\LL_{-\la+\g-k-2}$.
\end{lemma}\begin{proof}By the definition \eqref{Eq.duality_def}, it is enough to prove that \begin{align}\label{dual}\int_0^\infty(\LL_\la f)(Z)g(Z)Z^k\,\dZ=\int_0^\infty f(Z)(\LL_{-\la+\g-k-2} g)(Z)Z^k\,\dZ\end{align}
	for all $\la\in\C$ and $f,g\in C_c^\infty((0,+\infty))$. We fix $\la\in\C$ and $f,g\in C_c^\infty((0,+\infty))$. Let
	$$\widetilde{f}(t,x):=(T-t)^{\la}f(Z),\quad \widetilde{g}(t,x):=(T-t)^{-\la+\g-k-2}g(Z), \quad \forall\ (t,x)\in [0, T)\times \R^d,$$ recalling $Z=|x|/(T-t)$. Then by Lemma \ref{Lem.LL_la_formula}, we have $\LL\widetilde{f}(t,x)=(T-t)^{\la-\g}(\LL_\la f)(Z) $ and $\LL\widetilde{g}(t,x)=(T-t)^{-\la-k-2}(\LL_{-\la+\g-k-2} g)(Z) $, thus $(\LL\widetilde{f}\cdot\widetilde{g})(t,x)=(T-t)^{-k-2}(\LL_\la f)(Z)g(Z) $, and $(\widetilde{f}\cdot\LL\widetilde{g})(t,x)=(T-t)^{-k-2}f(Z)(\LL_{-\la+\g-k-2} g)(Z)$. Recall that $d=k+1$, $ Z=|x|/(T-t)$, then we have
(here $|S^k| $ is the area of the unit sphere $S^k$ in $\R^d=\R^{k+1}$)
\begin{align*}
		&\int_{\R^d}(\LL\widetilde{f}\cdot\widetilde{g})(t,x)dx=(T-t)^{-1}|S^k|\int_0^\infty(\LL_\la f)(Z)g(Z)Z^k\,\dZ,\\
&\int_{\R^d}(\widetilde{f}\cdot\LL\widetilde{g})(t,x)dx=(T-t)^{-1}|S^k|\int_0^\infty f(Z)(\LL_{-\la+\g-k-2} g)(Z)Z^k\,\dZ,
	\end{align*} for all $t\in[0, T)$. 
Thus, it is enough to prove that \begin{align}\label{dual1}\int_{\R^d}(\LL\widetilde{f}\cdot\widetilde{g})(t,x)\,\mathrm dx=\int_{\R^d}(\widetilde{f}\cdot\LL\widetilde{g})(t,x)\,\mathrm dx, 
\quad\ \forall\ t\in[0,T).\end{align}
Let $ \mathcal{J}:=\LL\widetilde{f}\cdot\widetilde{g}-\widetilde{f}\cdot\LL\widetilde{g}$. Then \eqref{dual1} is further reduced to \begin{align}\label{dual2}\int_{\R^d}\mathcal{J}(t,x)\,\mathrm dx=0, 
\quad\ \forall\ t\in[0,T).
\end{align}

By the definition of $\LL $ in \eqref{Eq.L}, we can write $ \mathcal{J}$ in  the divergence form $ \mathcal{J}=\pa^\al P_{\al}$ with
\begin{align*}
		&P_{\al}:=\rho_0^2(\pa_\al\widetilde{f}\widetilde{g}-\widetilde{f}\pa_\al\widetilde{g})-
\frac4{p-1}\rho_0^{3-p}\pa_\al\phi_0(\pa^{\tilde\alpha}\phi_0\pa_{\tilde\alpha}\widetilde{f}\widetilde{g}-
\widetilde{f}\pa^{\tilde\alpha}\phi_0\pa_{\tilde\alpha}\widetilde{g}),\quad \forall\ \al\in\Z\cap[0, d].
	\end{align*}
	Let $E(t):=\int_{\R^d}P_0(t,x)\,\mathrm dx$ for $t\in[0, T)$.  By the divergence theorem (recalling $\pa^0=-\pa_0=-\pa_t$ and the fact that $\operatorname{supp}_xP_\al(t,\cdot)$ is compact for each $t$ and $\al$), we have
	\begin{align}\label{J1}
		-\frac{\mathrm d}{\mathrm dt}E(t)=\int_{\R^d}\pa^0P_0(t,x)\,\mathrm dx=\int_{\R^d}\pa^\al P_\al(t,x)\,\mathrm dx=\int_{\R^d}\mathcal{J}(t,x)\,\mathrm dx,\quad\forall\ t\in[0, T).
	\end{align}
Thus, it is enough to prove that $E(t)$ is constant in $t$. We can write $P_0=P_{0,1}-\frac4{p-1}P_{0,2}$ with\begin{align*}
		&P_{0,1}:=\rho_0^2(\pa_t\widetilde{f}\widetilde{g}-\widetilde{f}\pa_t\widetilde{g}),\quad 
P_{0,2}:=\rho_0^{3-p}\pa_t\phi_0(\pa^{\tilde\alpha}\phi_0\pa_{\tilde\alpha}\widetilde{f}\widetilde{g}-
\widetilde{f}\pa^{\tilde\alpha}\phi_0\pa_{\tilde\alpha}\widetilde{g}).
	\end{align*}
	As $\widetilde{f}(t,x)=(T-t)^{\la}f(Z) $, $\widetilde{g}(t,x)=(T-t)^{-\la+\g-k-2}g(Z) $, we have
	\begin{align*}
		&\pa_t\widetilde{f}(t,x)=(T-t)^{\la-1}f_1(Z)\text{ with } f_1(Z):=-\la f(Z)+Zf'(Z),\\& \pa_t\widetilde{g}(t,x)=(T-t)^{-\la+\g-k-3}g_1(Z)\text{ with } g_1(Z):=-(-\la+\g-k-2)g(Z)+Zg'(Z).
	\end{align*}
	Then by \eqref{Eq.phi_0_selfsimilar},  $\g=\frac{4\beta}{p-1}+2 $ and $d=k+1$, we get
	\begin{align*}
		P_{0,1}(t,x)&=(T-t)^{-\frac{4\beta}{p-1}+\g-k-3}\hat\rho_0(Z)^2[f_1(Z)g(Z)-f(Z)g_1(Z)]\\&=(T-t)^{-d}\hat\rho_0(Z)^2[f_1(Z)g(Z)-f(Z)g_1(Z)].
	\end{align*}
	As $\widetilde{f}(t,x)=(T-t)^{\la}f(Z) $, $\widetilde{g}(t,x)=(T-t)^{-\la+\g-k-2}g(Z) $, we get by \eqref{Eq.phi_0_selfsimilar} and \eqref{Eq.pa^al_pa_al} that
\begin{align*}
	&\pa^{\tilde\alpha}\phi_0\pa_{\tilde\alpha}\widetilde{f}(t,x)=
(T-t)^{\la-\beta-1}f_2(Z),\quad \pa^{\tilde\alpha}\phi_0\pa_{\tilde\alpha}\widetilde{g}(t,x)=
(T-t)^{-\la+\g-k-3-\beta}g_2(Z),
\end{align*}
where 
\begin{align*}
	&f_2:=-(Z\pa_Z\hat\phi_0-(1-\beta)\hat\phi_0)(Z\pa_Z-\la)f+\pa_Z\hat\phi_0\pa_Zf,\\
&g_2:=-(Z\pa_Z\hat\phi_0-(1-\beta)\hat\phi_0)(Z\pa_Z+\la-\g+k+2)g+\pa_Z\hat\phi_0\pa_Zg.
\end{align*}
Then by \eqref{Eq.phi_0_selfsimilar}, \eqref{phi0} and $\g=\frac{4\beta}{p-1}+2=\frac{2(3-p)\beta}{p-1}+2\beta+2 $, $d=k+1$, we have
\begin{align*}
		&P_{0,2}(t,x)\\ &=(T-t)^{-\frac{2(3-p)\beta}{p-1}-2\beta+\g-k-3}\hat\rho_0(Z)^{3-p}[(\beta-1)\hat\phi_0(Z)+Z\hat\phi_0'(Z)]
[f_2(Z)g(Z)-f(Z)g_2(Z)]
\\&=(T-t)^{-d}\hat\rho_0(Z)^{3-p}[(\beta-1)\hat\phi_0(Z)+Z\hat\phi_0'(Z)]
[f_2(Z)g(Z)-f(Z)g_2(Z)].
	\end{align*}
	As $P_0=P_{0,1}-\frac4{p-1}P_{0,2}$, we have $P_{0}(t,x)=(T-t)^{-d}H(Z) $ with
	\begin{align*}
	H(Z):=&\hat\rho_0(Z)^2[f_1(Z)g(Z)-f(Z)g_1(Z)]\\&-\frac4{p-1}\hat\rho_0(Z)^{3-p}[(\beta-1)\hat\phi_0(Z)+Z\hat\phi_0'(Z)]
[f_2(Z)g(Z)-f(Z)g_2(Z)].
\end{align*}
Then by $d=k+1$, $ Z=|x|/(T-t)$, we have $E(t)=\int_{\R^d}P_0(t,x)\,\mathrm dx=|S^k|\int_0^\infty H(Z)Z^k\,\dZ$, which is constant in $t$. By \eqref{J1}, we have \eqref{dual2}, thus \eqref{dual1} and \eqref{dual}. 
\end{proof}

\subsection{Surjection of $\LL$}

This subsection is devoted to the proof of Proposition \ref{Prop.L_surjective}, i.e., $\LL:\mathscr X_\la\to\mathscr X_{\la-\g}$ is surjective for all $\la\in\C$.
For this, it suffices to show that
\begin{lemma}\label{Lem.L_surjective_R}
	If $R\in(k,+\infty)$, then the linear operator $\LL:\mathscr X_\la\to\mathscr X_{\la-\g}$ is surjective for all $\la\in B_R:=\{\la\in\C:|\la|<R\}$, where $\g:=4\beta/(p-1)+2=\beta(\ell-1)+2$.
\end{lemma}
From here until the end of this section, we fix an  $R\in(k,+\infty)$.

We consider functions depending analytically on a complex number $\la$. Let $I\subset[0,+\infty)$ be an interval and let $\Omega\subset \C$ be an open set. We define
\begin{align*}
	\operatorname{Hol}(\Omega):&=\{\text{all holomorphic function on $\Omega$}\},\\
\mathcal{H}_I(\Omega):&=\left\{f=f(Z;\la)\in C^\infty(I\times\Omega;\C): f(Z;\cdot)\in\operatorname{Hol}(\Omega)\text{ for all }Z\in I\right\},\\
	\mathcal H_I^{\text{e}}(\Omega):&=\{f\in C^\infty(I\times\Omega):\exists\ \widetilde{f}\in\mathcal H_{I^2}(\Omega),\text{ s.t. }f(Z;\la)=\tilde f(Z^2;\la)  \quad \forall\ Z\in I,\ \lambda\in \Omega\}.
\end{align*}
Then $ \operatorname{Hol}(\Omega)$, $ \mathcal H_I(\Omega)$, $ \mathcal H_I^{\text{e}}(\Omega)$ are rings. Moreover, we have
\[\mathcal H_{[0,+\infty)}^{\text{e}}(\Omega)=\mathcal H_{[0,a_2)}^{\text{e}}(\Omega)\cap\mathcal H_{(a_1,+\infty)}(\Omega), \quad \forall\ 0<a_1<a_2<+\infty.\]

The proof of Lemma \ref{Lem.L_surjective_R} is based on the following result, which will be proved in next subsection. 

\begin{lemma}\label{Lem.L_laf=g}
There exists $ \varphi\in \operatorname{Hol}(B_R)\setminus\{0\}$	such that if $g\in C_{\operatorname{e}}^\infty([0,+\infty))$, then there exists
$f=f(Z;\la)\in\mathcal H_{[0,+\infty)}^{\operatorname{e}}(B_R)$ such that $\LL_\la f(\cdot;\la)=\varphi(\lambda)\cdot g$ on $(0,+\infty)$ (for all $ \lambda\in B_R$).
\end{lemma}

\begin{proof}[Proof of Lemma \ref{Lem.L_surjective_R}]
We first prove that $\LL$ maps $\mathscr X_\la$ to $\mathscr X_{\la-\g}$.
 
 Recall that $\mathscr{L}(\phi)=\pa^\al\left(\rho_0^2\pa_\al\phi-\frac4{p-1}\rho_0^{3-p}\pa_\al\phi_0\pa^{\tilde\alpha}\phi_0\pa_{\tilde\alpha}\phi\right)$, $\rho_0^2\in \mathscr X_{2\mu_0}$, $\rho_0^{3-p}\in \mathscr X_{(3-p)\mu_0}$, $\phi_0\in \mathscr X_{\la_0}$, $\la_0=1-\beta$, $\mu_0=-\frac{2\beta}{p-1}$ (see \eqref{Eq.la_n_mu_n}) and $\g=4\beta/(p-1)+2$. If $ \phi\in\mathscr X_\la $, by Lemma \ref{Lem.X_la_properties} (i), we have $\pa^\al(\rho_0^2\pa_\al\phi)\in\mathscr X_{\la-\g}$, $\pa^\al(\rho_0^{3-p}\pa_\al\phi_0\pa^{\tilde\alpha}\phi_0\pa_{\tilde\alpha}\phi)\in\mathscr X_{\la-\g}$, where we have used that $\la+2\mu_{0}-2=\la-\g=\la+\la_0-2+(3-p)\mu_0+\la_0-2$, thus $ \mathscr{L}(\phi)\in\mathscr X_{\la-\g} $. 
 
 Now we prove that $\LL$ is surjective.
By the definitions of $\mathscr X_0$ and $\mathscr X_\la$, it suffices to prove that for every $g\in C^\infty_\text{e}([0,+\infty))$,
$n\in\Z_{\geq 0}$ and $\la_*\in B_R$, there exists $F_n\in \mathscr X_{\la_*}$ such that $\LL F_n(t,x)=(T-t)^{\la_*-\gamma} g(Z)\tau^n/n!$. Now we  fix  $g\in C_\text{e}^\infty([0,+\infty))$ and $\la_*\in B_R$.

By Lemma \ref{Lem.L_laf=g}, there exist $ \varphi\in \operatorname{Hol}(B_R)\setminus\{0\}$ and a function $f=f(Z;\la)\in\mathcal H_{[0,+\infty)}^{\operatorname{e}}(B_R)$ such that
$\LL_\la f(Z;\la)=\varphi(\lambda)g(Z)$ for $Z\in(0,+\infty)$, $ \lambda\in B_R$. As $f\in\mathcal H_{[0,+\infty)}^{\operatorname{e}}(B_R)$ there exists
$\widetilde{f}\in\mathcal H_{[0,+\infty)}(B_R)$ such that $f(Z;\la)=\widetilde{f}(Z^2;\la)$ for $Z\in[0,+\infty)$, $ \lambda\in B_R$. As $\la_*\in B_R$, there exist $\delta_*>0$ and $m_*\in\Z_{\geq 0}$ such that $B_{2\delta_*}(\la_*)\subset B_R$ and
	\[\varphi(\la)=(\la-\la_*)^{m_*}\tilde\varphi(\la)\text{ with }\tilde \varphi(\la)\neq0,\quad\forall\ \la\in \Omega_*:=B_{2\delta_*}(\la_*),\text{ where }\tilde \varphi\in \operatorname{Hol}(\Omega_*).\]
Here $B_{2\delta_*}(\la_*):=\{\la\in\C:|\la-\la_*|<2\delta_*\}$ and we have used the fact that if $\varphi\in\operatorname{Hol}(\Omega)\setminus\{0\}$, then the zero set $\mathcal Z(\varphi):=\{\la\in\Omega:\varphi(\la)=0\}$ is discrete.

Let $\tilde F({Z};\la):=\tilde f({Z};\la)/\tilde \varphi(\la)$, $F(Z;\la):=\tilde F(Z^2;\la)$ for $Z\in[0,+\infty)$, $ \lambda\in B_R$. Then $\tilde F\in \mathcal H_{[0,+\infty)}(\Omega_*)$,
$F\in\mathcal H_{[0,+\infty)}^{\operatorname{e}}(\Omega_*)$,
$F(Z;\la)=f(Z;\la)/\tilde \varphi(\la)$, and
$$\LL_\la F(Z;\la)=\varphi(\lambda)g(Z)/\tilde \varphi(\la)=(\la-\la_*)^{m_*}g(Z),\quad\forall\ Z\in (0,+\infty), \la\in \Omega_*.$$ 

By Lemma \ref{Lem.LL_la_formula}, we  have $$ \LL((T-t)^{\lambda}F(Z;\la))=(T-t)^{\lambda-\gamma}\LL_\la F(Z;\la)=(T-t)^{\lambda-\gamma}(\la-\la_*)^{m_*}g(Z)$$ for all $\la\in\Omega_*$ and $Z\in(0, +\infty)$. Let
$$ F_*(t,x;\la):=(T-t)^{\lambda}F(Z;\la)=(T-t)^{\lambda}\tilde F(Z^2;\la),\quad  G(t,x;\la):=(T-t)^{\lambda-\gamma}(\la-\la_*)^{m_*}g(Z).$$ Then $F_*,G\in C^{\infty}([0,T)\times\R^d\times\Omega_*) $
(as $Z^2=|x|^2/(T-t)^2$ is smooth on $[0,T)\times\R^d $) and
$ \LL F_*(t,x;\la)=G(t,x;\la)$ on $[0,T)\times\R^d\times\Omega_* $ (the case $Z=0$ follows by continuity).

Recall that $ \tau=\ln\frac{1}{T-t}$ and then
\begin{align*}
	G(t,x;\la)=(T-t)^{\lambda_*-\gamma}\mathrm{e}^{-(\lambda-\lambda_*)\tau}(\la-\la_*)^{m_*}g(Z)=\sum_{n=0}^{\infty}(T-t)^{\lambda_*-\gamma}\frac{(-\tau)^n}{n!}(\la-\la_*)^{m_*+n}g(Z)
\end{align*} locally uniformly on $[0,T)\times\R^d\times B_{\delta_*}(\la_*)$. By Cauchy's integration formula (Theorem 4.4 in Chapter 2 of \cite{Stein}), we have (for $n\in\Z_{\geq0}$)\begin{align*}
	(T-t)^{\la_*-\gamma} g(Z)\frac{\tau^n}{n!}&=\frac{(-1)^n}{2\pi\ii}\oint_{|\la-\la_*|=\delta_*}\frac{G(t,x;\la)}{(\la-\la_*)^{m_*+n+1}}\,\mathrm d\la\\&=
\frac{\delta_*^{-m_*-n}}{2\pi(-1)^n}\int_0^{2\pi}G(t,x;\la_*+\delta_*\mathrm{e}^{\ii \theta})\mathrm{e}^{-\ii (m_*+n)\theta}\,\mathrm d\theta.
\end{align*}
Now let (for $n\in\Z_{\geq0}$)
\begin{align*}
	F_n(t,x):&=\frac{(-1)^n}{2\pi\ii}\oint_{|\la-\la_*|=\delta_*}\frac{F_*(t,x;\la)}{(\la-\la_*)^{m_*+n+1}}\,\mathrm d\la\\&=
\frac{\delta_*^{-m_*-n}}{2\pi(-1)^n}\int_0^{2\pi}F_*(t,x;\la_*+\delta_*\mathrm{e}^{\ii \theta})\mathrm{e}^{-\ii (m_*+n)\theta}\,\mathrm d\theta.
\end{align*}
Then $F_n\in C^{\infty}([0,T)\times\R^d) $ and $\LL F_n(t,x)=(T-t)^{\la_*-\gamma} g(Z)\tau^n/n!$. 
It remains to prove that $F_n\in \mathscr X_{\la_*}$. 

As $ \tau=\ln\frac{1}{T-t}$, $ F_*(t,x;\la)=(T-t)^{\lambda}\tilde F(Z^2;\la)$ then\begin{align*}
	F_*(t,x;\la)=(T-t)^{\lambda_*}\mathrm{e}^{-(\lambda-\lambda_*)\tau}\tilde F(Z^2;\la)=\sum_{j=0}^{\infty}(T-t)^{\lambda_*}\frac{(-\tau)^j}{j!}(\la-\la_*)^{j}\tilde F(Z^2;\la),
\end{align*}locally uniformly on $[0,T)\times\R^d\times B_{\delta_*}(\la_*)$, so we have\begin{align*}
	F_n(t,x)&=\sum_{j=0}^{\infty}(T-t)^{\lambda_*}\frac{(-\tau)^j}{j!}F_{n,j}(Z^2),\qquad\text{where}\\ F_{n,j}(\widetilde{Z}):&=
\frac{(-1)^n}{2\pi\ii}\oint_{|\la-\la_*|=\delta_*}\frac{(\la-\la_*)^{j}\tilde F(\widetilde{Z};\la)}{(\la-\la_*)^{m_*+n+1}}\,\mathrm d\la\\
&=\frac{\delta_*^{j-m_*-n}}{2\pi(-1)^n}\int_0^{2\pi}\tilde F(\widetilde{Z};\la_*+\delta_*\mathrm{e}^{\ii \theta})\mathrm{e}^{\ii (j-m_*-n)\theta}\,\mathrm d\theta.
\end{align*}As $\tilde F\in \mathcal H_{[0,+\infty)}(\Omega_*)\subset C^{\infty}([0,+\infty)\times\Omega_*)$, we have
$F_{n,j}\in C^\infty([0,+\infty))$, $Z\mapsto F_{n,j}(Z^2)\in C_{\operatorname{e}}^\infty([0,+\infty))$ for every $n,j\in\Z_{\geq0}$; moreover by Cauchy's theorem (Corollary 2.3 in Chapter 2 of \cite{Stein}),
we have $F_{n,j}=0 $ for $j>m_*+n$, $n,j\in\Z_{\geq0}$. Thus, 
$$F_n(t,x)=\sum_{j=0}^{m_*+n}(T-t)^{\lambda_*}\frac{(-\tau)^j}{j!}F_{n,j}(Z^2) \in \mathscr X_{\la_*}.$$ 

This completes the proof of Lemma \ref{Lem.L_surjective_R}.
\end{proof}

\subsection{Solvability of $\LL_\la$}

In this subsection, we prove Lemma \ref{Lem.L_laf=g}.

\begin{lemma}\label{Lem.Fundamental_sol_1}
For $g\in C_{\operatorname{e}}^{\infty}([0, Z_1))$,	there exists $F=F(Z;\la)\in \mathcal H_{[0, Z_1)}^{\operatorname{e}}(B_R)$ satisfying $\LL_\la F=g$ on $(0,Z_1)$ and $F(0;\la)=1$ for all $\la\in B_R$.
\end{lemma}

\begin{proof}
	By Lemma \ref{Lem.LL_la_formula}, we have $\hat B_0\in C_{\operatorname{o}}^\infty([0,+\infty))$ and $A_0$, $\tilde B_0$, $D_1$, $D_2\in C_{\operatorname{e}}^\infty([0,+\infty))$. Thus, there exist $\widetilde{A}_0,B_1,B_2,\widetilde{D}_1,\widetilde{D}_2\in C^\infty([0,+\infty))$ such that
$A_0(Z)=\widetilde{A}_0(Z^2),$ $\tilde B_0(Z)=B_1(Z^2),$ $\hat B_0(Z)=ZB_2(Z^2),$ $D_1(Z)=\widetilde{D}_1(Z^2),$ $D_2(Z)=\widetilde{D}_2(Z^2)$. Then
$B_0(Z;\la)=Z^{-1}\tilde B_0(Z)+\la \hat B_0(Z)=Z^{-1} B_1(Z^2)+\la Z B_2(Z^2)$,
$D_0(Z;\la)=\lambda D_1(Z)+\lambda^2D_2(Z)=\lambda \widetilde{D}_1(Z^2)+\lambda^2\widetilde{D}_2(Z^2)$. Let $\widetilde{Z}:=Z^2$. Then for $ f(Z)=\widetilde{f}(Z^2)=\widetilde{f}(\widetilde{Z})$, we have
$ f'(Z)=2Z\widetilde{f}'(\widetilde{Z})$, $ f''(Z)=2\widetilde{f}'(\widetilde{Z})+4Z^2\widetilde{f}''(\widetilde{Z})$, and by \eqref{Lf},
\begin{align*}
		&(\LL_\la f)(Z)=A_0(Z)f''(Z)+B_0(Z;\la)f'(Z)+D_0(Z;\la)f(Z)\\
&=\widetilde{A}_0(Z^2)[2\widetilde{f}'(\widetilde{Z})+4Z^2\widetilde{f}''(\widetilde{Z})]+[Z^{-1} B_1(Z^2)+\la Z B_2(Z^2)]\cdot2Z\widetilde{f}'(\widetilde{Z})\\&\quad+
[\lambda \widetilde{D}_1(Z^2)+\lambda^2\widetilde{D}_2(Z^2)]\widetilde{f}(\widetilde{Z})\\
&=4\widetilde{Z}\widetilde{A}_0(\widetilde{Z})\widetilde{f}''(\widetilde{Z})+2[\widetilde{A}_0(\widetilde{Z})+B_1(\widetilde{Z})+\la \widetilde{Z} B_2(\widetilde{Z})]\widetilde{f}'(\widetilde{Z})+
[\lambda \widetilde{D}_1(\widetilde{Z})+\lambda^2\widetilde{D}_2(\widetilde{Z})]\widetilde{f}(\widetilde{Z}).
	\end{align*}
	Let
\begin{align*}
		&\widetilde{A}(\widetilde{Z})=4\widetilde{Z}\widetilde{A}_0(\widetilde{Z}),\ \widetilde{B}(\widetilde{Z};\la)=2[\widetilde{A}_0(\widetilde{Z})+B_1(\widetilde{Z})+\la \widetilde{Z} B_2(\widetilde{Z})],\ \widetilde{D}(\widetilde{Z};\la)=\lambda \widetilde{D}_1(\widetilde{Z})+\lambda^2\widetilde{D}_2(\widetilde{Z}).
	\end{align*}
	Then we get
\begin{align}\label{Eq.A_B_D_1}
		&(\LL_\la f)(Z)=\widetilde{A}(\widetilde{Z})\widetilde{f}''(\widetilde{Z})+\widetilde{B}(\widetilde{Z};\la)\widetilde{f}'(\widetilde{Z})+
\widetilde{D}(\widetilde{Z};\la)\widetilde{f}(\widetilde{Z}),\ \text{for}\ f(Z)=\widetilde{f}(\widetilde{Z}),\ \widetilde{Z}=Z^2.
	\end{align}
	Let $I_1=[0,Z_1^2)$. As $\widetilde{A}_0,B_1,B_2,\widetilde{D}_1,\widetilde{D}_2\in C^\infty([0,+\infty))$, we have $\widetilde{A}\in C^\infty(I_1)$, $\widetilde{B}, \widetilde{D}\in \mathcal H_{I_1}(\C)$,
 and $\widetilde{A}'(0)=4\widetilde{A}_0(0)=4{A}_0(0)=4\neq0$ (using \eqref{Eq.A_0_non_degenerate}). By \eqref{Eq.A_0_neq0}, we have $\widetilde{A}_0(Z^2)=A_0(Z)>0$ for $Z\in[0,Z_1)$. Thus,
 $\widetilde{A}(\widetilde{Z})=4\widetilde{Z}\widetilde{A}_0(\widetilde{Z})=0$ has a unique solution $\widetilde{Z}=0 $ in $I_1=[0,Z_1^2)$.

 Moreover, we have $\widetilde{B}(\widetilde{Z};\la)=\tilde B_1(\widetilde{Z})+\la \widetilde{B}_2(\widetilde{Z})$, where $\tilde{B}_1(\widetilde{Z}):=2[\widetilde{A}_0(\widetilde{Z})+B_1(\widetilde{Z})], \widetilde{B}_2(\widetilde{Z}):=2\widetilde{Z}B_2(\widetilde{Z})$, then $\tilde B_1(0)=2[\widetilde{A}_0(0)+B_1(0)]=2[{A}_0(0)+\tilde B_0(0)]
 =2(1+k)>0$ (using \eqref{Eq.A_0_non_degenerate} and \eqref{Eq.B_0}) and $\widetilde{B}_2(0)=0$.
As a consequence, for any $\la\in\C$ and $n\in\Z_{\geq 0}$ we have $n\widetilde{A}'(0)+\widetilde{B}(0;\la)=4n+2(1+k)\neq0$. As $g\in C_{\operatorname{e}}^{\infty}([0, Z_1))$, there exists 
$\widetilde{g}\in C^{\infty}([0,Z_1^2))$ such that $g(Z)=\widetilde{g}(Z^2)$.
 By Proposition \ref{Prop.Appen_1}, there exists $\tilde F=\tilde F(\widetilde{Z};\la)\in \mathcal H_{I_1}(B_R)$ satisfying\begin{align*}
		&\widetilde{A}(\widetilde{Z})\tilde F''(\widetilde{Z};\la)+\widetilde{B}(\widetilde{Z};\la)\tilde F'(\widetilde{Z};\la)+
\widetilde{D}(\widetilde{Z};\la)\tilde F(\widetilde{Z};\la)=\widetilde{g}(\widetilde{Z}),\quad \tilde F(0;\la)=1,
	\end{align*}where the prime $'$ denotes the derivative with respect to $\widetilde{Z}$. Now we define
	\[F(Z;\la):=\tilde F(Z^2;\la),\qquad\forall\ Z\in[0,Z_1), \forall\ \la\in B_R,\]
	then  $F\in\mathcal H_{[0, Z_1)}^{\operatorname{e}}(B_R)$, $F(0;\la)=\tilde F(0;\la)=1$ and $\LL_\la F=g$ in $(0, Z_1)$ by recalling \eqref{Eq.A_B_D_1}.
\end{proof}In view of Lemma \ref{Lem.LL_la_formula} and Proposition \ref{Prop.Appen_1}, we let
\begin{equation}\label{Eq.La_*0}
	\La_*:=\{\la\in\C: nA_0'(Z_1)+B_0(Z_1;\la)=0\text{ for some }n\in\Z_{\geq 0}\}.
\end{equation}
By $B_0(Z_1;\la)=Z_1^{-1}\tilde B_0(Z_1)+\la \hat B_0(Z_1)$ and $\hat B_0(Z_1)>0$, we know that $\La_*\subset \C$ is a non-empty (countable) discrete set.

\begin{lemma}\label{Lem.Fundamental_sol_2}
	 There exists a nonzero polynomial $ \psi_1(\lambda)$ satisfying $\{\lambda\in B_R: \psi_1(\lambda)=0\}=\La_*\cap B_R$ such that for $g\in C^{\infty}((0, +\infty))$,	there exists a function $F=F(Z;\la)\in\mathcal H_{(0,+\infty)}(B_R)$ satisfying $\LL_\la F=\psi_1(\lambda)\cdot g$ on $(0, +\infty)$ and $F(Z_1;\la)=\psi_1(\lambda)$ for all $\la\in B_R$.
\end{lemma}\begin{proof}
	By Lemma \ref{Lem.LL_la_formula}, we have $A_0(Z_1)=0$, $A_0'(Z_1)\neq0$, $\hat B_0(Z_1)>0$ and $Z_1$ is the unique solution of $A_0(Z)=0$ in $(0, +\infty) $. 
Hence Lemma \ref{Lem.Fundamental_sol_2} follows from Proposition \ref{Prop.Appen_1}.
\end{proof}
Taking $g=0$ in Lemma \ref{Lem.Fundamental_sol_1} we know that there exists $\Psi_1=\Psi_1(Z;\la)\in \mathcal H_{[0, Z_1)}^{\operatorname{e}}(B_R)$ satisfying $\LL_\la\Psi_1=0$ on $(0,Z_1)$ and $\Psi_1(0;\la)=1$ for all $\la\in B_R$. Taking $g=0$ in Lemma \ref{Lem.Fundamental_sol_2} we know that there exists $\Psi_2=\Psi_2(Z;\la)\in\mathcal H_{(0,+\infty)}(B_R)$ satisfying $\LL_\la\Psi_2=0$ on $(0, +\infty)$ and 
$\Psi_2(Z_1;\la)=\psi_1(\lambda)$ for all $\la\in B_R$. We define the Wronski
\begin{equation}\label{W1}W(Z;\la):=\Psi_1(Z;\la)\Psi_2'(Z;\la)-\Psi_1'(Z;\la)\Psi_2(Z;\la),\quad \forall\  Z\in(0, Z_1),\ \forall\ \la\in B_R,\end{equation}
where the prime $'$ denotes the derivative with respect to $Z$. Then we have\begin{equation}\label{Eq.Wronksi_eq}
		A_0(Z)W'(Z;\la)+B_0(Z;\la)W(Z;\la)=0,\qquad\forall\  Z\in(0, Z_1),\quad \lambda\in B_R.
	\end{equation}
	
\begin{lemma}\label{Lem.Wronski_1}
	Fix $Z_0\in(0,Z_1)$. Let $\psi_2(\lambda):=W(Z_0;\la)$ for all $\la\in B_R$ and $\la_0^*:=\g-k-2$. Then $\psi_2\in \operatorname{Hol}(B_R)$, $0<-\la_0^*<k<R$ and $\psi_2(\la_0^*)\neq0$.
\end{lemma}

\begin{proof}
As $\Psi_1(Z;\la)\in \mathcal H_{[0, Z_1)}^{\operatorname{e}}(B_R)$, $\Psi_2(Z;\la)\in\mathcal H_{(0,+\infty)}(B_R)$, by \eqref{W1} we have 
$W(Z;\la)\in \mathcal H_{(0, Z_1)}(B_R)$, then by $Z_0\in(0,Z_1)$ we have $\psi_2(\la)=W(Z_0;\la)\in \operatorname{Hol}(B_R)$. As $\beta >0$, $\ell>1$, $\gamma=\beta(\ell-1)+2$, $\la_0^*=\g-k-2$, we get by \eqref{Eq.parameter_range}  that
	\begin{equation}\label{la0*}
		\la_0^*=\g-k-2=\beta(\ell-1)+2-k-2=\beta(\ell-1)-k<\beta(\ell-1)-\beta(\ell+\sqrt\ell)<0,
	\end{equation}and $R>k>k-\beta(\ell-1)=-\la_0^*>0$, then $\la_0^*\in B_R$. It remains to prove that $\psi_2(\la_0^*)\neq0$.

We consider the dual $\LL_0^*$ of $\LL_0$, defined by \eqref{Eq.duality_def}. On one hand, we get by Lemma \ref{Lem.duality} that
	\begin{equation}\label{Eq.L_0*1}
		\LL_0^*=\LL_{\la_0^*}=A_0\pa_Z^2+B_0(\cdot;\la_0^*)\pa_Z+D_0(\cdot;\la_0^*).
	\end{equation}
	On the other hand, by (recalling that $D_0(\cdot;0)=0$) $$\LL_0f=A_0\pa_Z^2+B_0(\cdot;0)\pa_Z=A_0\pa_Z^2+Z^{-1}\tilde B_0\pa_Z$$
	and \eqref{Eq.duality_def}, we compute that
	\begin{align}\label{Eq.L_0*2}
		(\LL_0^*f)(Z)=\frac{1}{Z^k}\left(\pa_Z^2(Z^kA_0f)(Z)-\pa_Z(Z^{k-1}\tilde B_0f)(Z)\right),\quad\forall\ Z\in(0,+\infty).
	\end{align}
Comparing the coefficients of $\pa_Z$ in \eqref{Eq.L_0*1} and \eqref{Eq.L_0*2}, we obtain
	\begin{align*}
		Z^{-1}\tilde B_0(Z)+\la_0^*\hat B_0(Z)&=B_0(Z;\la_0^*)=[2\pa_Z(Z^kA_0)(Z)-Z^{k-1}\tilde B_0(Z)]/{Z^k}\\
&={2k}{Z}^{-1}A_0(Z)+2A_0'(Z)-Z^{-1}\tilde B_0(Z)
	\end{align*}
	for all $Z\in(0,+\infty)$. Letting $Z=Z_1$, we get(as $A_0(Z_1)=0$, see \eqref{Eq.A_0_non_degenerate})
	\begin{equation}\label{Eq.5.47}
		Z_1^{-1}\tilde B_0(Z_1)+{\la_0^*}\hat B_0(Z_1)/{2}=A_0'(Z_1).
	\end{equation}
	For any $n\in\Z_{\geq 0}$, by \eqref{Eq.5.47}, \eqref{Eq.A_0_non_degenerate}, \eqref{Eq.B_0} and $\la_0^*<0$ (i.e. \eqref{la0*}), we have
	\begin{equation}\label{Eq.la_0^*_not_in_La}
		\begin{aligned}
			nA_0'(Z_1)+B_0(Z_1;\la_0^*)&=nA_0'(Z_1)+Z_1^{-1}\tilde B_0(Z_1)+\la_0^*\hat B_0(Z_1)\\
			&=(n+1)A_0'(Z_1)+{\la_0^*}\hat B_0(Z_1)/{2}<0,
		\end{aligned}
	\end{equation}
It follows from \eqref{Eq.La_*0} and \eqref{Eq.la_0^*_not_in_La} that $\la_0^*\notin\La_*$. Then by $\la_0^*\in B_R$ and Lemma \ref{Lem.Fundamental_sol_2}, we have $\psi_1(\la_0^*)\neq0$. Let $f_1=\Psi_1(\cdot;\la_0^*)$ and $f_2=\Psi_2(\cdot;\la_0^*)$, then $f_1\in C_\text{e}^\infty([0,Z_1))$, $f_2\in C^\infty((0,+\infty))$ and $(\LL_{\la_0^*}f_j)(Z)=0$ for $Z\in(0, Z_1)$, $j\in\{1, 2\}$. By 
\eqref{Eq.L_0*1} and \eqref{Eq.L_0*2}, we get
	\[\pa_Z^2(Z^kA_0f_j)(Z)-\pa_Z(Z^{k-1}\tilde B_0f_j)(Z)=0,\quad\forall\ Z\in(0, Z_1), j\in\{1,2\}.\]
	By $f_1,A_0,\tilde B_0\in C_\text{e}^\infty([0,Z_1))$, $k\geq 3$ we have 
$[\pa_Z(Z^kA_0f_1)(Z)-Z^{k-1}\tilde B_0(Z)f_1(Z)]|_{Z=0}=0$, so
	\begin{align}\label{Eq.f_1_constant}
		\pa_Z(Z^kA_0f_1)(Z)-Z^{k-1}\tilde B_0(Z)f_1(Z)=0,\quad\forall\ Z\in(0, Z_1).
	\end{align}
For $f_2$, since $A_0(Z_1)=0$ (see \eqref{Eq.A_0_non_degenerate}), we have
	\begin{align}\label{f2}
		&\pa_Z(Z^kA_0f_2)(Z)-Z^{k-1}\tilde B_0(Z)f_2(Z)=[\pa_Z(Z^kA_0f_2)(Z)-Z^{k-1}\tilde B_0(Z)f_2(Z)]|_{Z=Z_1}\\
\notag&=Z_1^kA_0'(Z_1)f_2(Z_1)-Z_1^{k-1}\tilde B_0(Z_1)f_2(Z_1)=Z_1^k(A_0'(Z_1)-Z_1^{-1}\tilde B_0(Z_1))\psi_1(\la_0^*)=:C'
	\end{align}
	for all $Z\in(0, Z_1)$, where we have used $f_2(Z_1)=\Psi_2(Z_1;\la_0^*)=\psi_1(\la_0^*)$ (recalling Lemma \ref{Lem.Fundamental_sol_2}). Moreover, by \eqref{Eq.5.47}, $\la_0^*<0$ (in \eqref{la0*}) and $\hat B_0(Z_1)>0$ (in \eqref{Eq.B_0}), we have $A_0'(Z_1)-Z_1^{-1}\tilde B_0(Z_1)={\la_0^*}\hat B_0(Z_1)/{2}<0$, then by $\psi_1(\la_0^*)\neq0$ we have $C'\neq0$. We claim that 
	\begin{equation}\label{Eq.f_1neq0}
		f_1(Z)\neq 0,\quad\forall\ Z\in(0, Z_1).
	\end{equation}
	Indeed, if $f_1(Z^*)=0$ for some $Z^*\in(0, Z_1)$, by the uniqueness of solutions to \eqref{Eq.f_1_constant} in $(0, Z_1)$ with $f_1(Z^*)=0$, we have $f_1(Z)=0$ for all $Z\in(0, Z_1)$, which contradicts with $1=f_1(0)=\lim_{Z\to0+}f_1(Z)$. This proves \eqref{Eq.f_1neq0}. As $f_1=\Psi_1(\cdot;\la_0^*)$, $f_2=\Psi_2(\cdot;\la_0^*)$, by \eqref{W1}, \eqref{Eq.f_1_constant}, \eqref{f2}, \eqref{Eq.f_1neq0} and $C'\neq0$, we have 
	\begin{align*}
		&Z^kA_0(Z)W(Z;\la_0^*)=Z^kA_0(Z)[f_1(Z)f_2'(Z)-f_1'(Z)f_2(Z)]\\&=f_1(Z)\pa_Z(Z^kA_0f_2)(Z)-\pa_Z(Z^kA_0f_1)(Z)f_2(Z)\\
&=f_1(Z)[Z^{k-1}\tilde B_0(Z)f_2(Z)+C']-Z^{k-1}\tilde B_0(Z)f_1(Z)f_2(Z)=C'f_1(Z)\neq0,
	\end{align*}for all $Z\in(0, Z_1)$. Thus, $W(Z;\la_0^*)\neq0 $ for all $Z\in(0, Z_1)$, and $\psi_2(\la_0^*)=W(Z_0;\la_0^*)\neq0$.
\end{proof}

Now we fix $Z_0\in(0,Z_1)$, $\psi_2(\lambda)=W(Z_0;\la)$, $\la_0^*:=\g-k-2$. Let $ \psi_1(\lambda)$ be given by Lemma \ref{Lem.Fundamental_sol_2} and $\varphi(\lambda):=\psi_1(\lambda)\psi_2(\lambda)$ for all $\la\in B_R$. Let $g\in C_\text{e}^\infty([0,+\infty))$, we need to prove that there exists a function $f=f(Z;\la)\in\mathcal H_{[0,+\infty)}^{\operatorname{e}}(B_R)$ such that $\LL_\la f(\cdot;\la)=\varphi(\lambda)\cdot g$ on $(0,+\infty)$.

We first consider the case when $g$ is supported near $Z=0$.

\begin{lemma}\label{Lem.g_near_0}
	Assume that $g\in C_{\operatorname{e}}^\infty([0,+\infty))$ satisfies $\operatorname{supp}g\subset [0, Z_1)$, then there exists a function $f=f(Z;\la)\in\mathcal H_{[0,+\infty)}^{\operatorname{e}}(B_R)$
such that $\LL_\la f(\cdot;\la)=\varphi(\lambda)\cdot g$ on $(0,+\infty)$.
\end{lemma}

\begin{proof}
	By Lemma \ref{Lem.Fundamental_sol_1}, there exists $f_0\in \mathcal H_{[0, Z_1)}^{\text e}(B_R)$ such that $\LL_\la f_0=g$ on $(0, Z_1)$ with $f_0(0;\la)=1$ for all $\la\in B_R$. We assume that $\operatorname{supp}g\subset [0,\delta)$ for some $\delta\in (0, Z_1)$, then $(\LL_\la f_0)(Z)=0$ for $Z\in[\delta, Z_1)$. For $\la\in B_R$, let
	\begin{align}
		C_1(\la):&={f_0(\delta;\la)\Psi_2'(\delta;\la)-f_0'(\delta;\la)\Psi_2(\delta;\la)}\in\C,\label{Eq.C_1}\\
		C_2(\la):&={f_0'(\delta;\la)\Psi_1(\delta;\la)-f_0(\delta;\la)\Psi_1'(\delta;\la)}\in\C.\label{Eq.C_2}
	\end{align}
	Then $C_1, C_2$ are holomorphic functions on $B_R$ and for all $\la\in B_R$ there holds
	\begin{align*}&W(\delta;\la)f_0(\delta;\la)=C_1(\la)\Psi_1(\delta;\la)+C_2(\la)\Psi_2(\delta;\la),\\& W(\delta;\la)f_0'(\delta;\la)=C_1(\la)\Psi_1'(\delta;\la)+C_2(\la)\Psi_2'(\delta;\la).\end{align*}
	By the uniqueness of the solution on $[\delta, Z_1)$, we have 
	$$W(\delta;\la)f_0(Z;\la)=C_1(\la)\Psi_1(Z;\la)+C_2(\la)\Psi_2(Z;\la),\qquad\forall\ Z\in[\delta, Z_1),\ \forall\ \la\in B_R.$$
	For $\la\in B_R$, let
	\[f_*(Z;\la):=\begin{cases}
		W(\delta;\la)f_0(Z;\la)-C_1(\la)\Psi_1(Z;\la) & \text{if }Z\in [0, Z_1),\\
		C_2(\la)\Psi_2(Z;\la) & \text{if }Z\in[\delta, +\infty).
	\end{cases}\]
	Then $f_*\in \mathcal H_{[0,+\infty)}^{\text e}(B_R)$ and $\LL_\la f_*=W(\delta;\la)\cdot g$ on $(0,+\infty)$. By \eqref{Eq.Wronksi_eq}, we have 
$W(Z;\la)=W(Z_0;\la)\mathrm{e}^{- A_*(Z;\la)}=\psi_2(\lambda)\mathrm{e}^{- A_*(Z;\la)}$ with $A_*(Z;\la):=\int_{Z_0}^Z\frac{B_0(Z;\la)}{A_0(Z)}\mathrm{d}Z\in\mathcal H_{(0,Z_1)}(B_R)  $ 
(using Lemma \ref{Lem.LL_la_formula}). Recall that $\varphi=\psi_1\psi_2$, $ \psi_1$ is a polynomial, then $\varphi(\lambda)=\psi_1(\lambda)W(\delta;\la)\mathrm{e}^{A_*(\delta;\la)} $, and the result follows by taking $ f(Z;\la):=\psi_1(\lambda)\mathrm{e}^{A_*(\delta;\la)}f_*(Z;\la)$ for $Z\in[0, +\infty),\la\in B_R$.
\end{proof}

Now we consider the case when $g$ is supported away from $Z=0$.

\begin{lemma}\label{Lem.g_away_0}
	Assume that $g\in C_{\operatorname{e}}^\infty([0,+\infty))$ satisfies $\operatorname{supp}g\subset (0, +\infty)$, then there exists a function $f=f(Z;\la)\in\mathcal H_{[0,+\infty)}^{\operatorname{e}}(B_R)$
such that $\LL_\la f(\cdot;\la)=\varphi(\lambda)\cdot g$ on $(0,+\infty)$.
\end{lemma}

\begin{proof}
	By Lemma \ref{Lem.Fundamental_sol_2}, there exists $f_0=f_0(Z;\la)\in\mathcal H_{(0,+\infty)}(B_R)$ such that $\LL_\la f_0=\psi_1(\lambda)\cdot g$ on $(0, +\infty)$. We assume that $\operatorname{supp} g\subset (\delta, +\infty)$ for some $\delta\in(0, Z_1)$, then $(\LL_\la f_0)(Z)=0$ for $Z\in (0,\delta]$. For $\la\in B_R$, let
	$C_1(\la), C_2(\la)$ be defined by \eqref{Eq.C_1} and \eqref{Eq.C_2} respectively. For the same reason as in the proof of Lemma \ref{Lem.g_near_0}, we have $$W(\delta;\la)f_0(Z;\la)=C_1(\la)\Psi_1(Z;\la)+C_2(\la)\Psi_2(Z;\la),\qquad\forall\ Z\in(0, \delta],\ \forall\ \la\in B_R.$$
	For $\la\in B_R$, let
	\[f_*(Z;\la):=\begin{cases}
		W(\delta;\la)f_0(Z;\la)-C_2(\la)\Psi_2(Z;\la) & \text{if }Z\in (0, +\infty),\\
		C_1(\la)\Psi_1(Z;\la) & \text{if }Z\in[0,\delta].
	\end{cases}\]
	Then $f_*\in\mathcal H_{[0,+\infty)}^{\operatorname{e}}(B_R)$ and $\LL_\la f_*=W(\delta;\la)\psi_1(\lambda)\cdot g$ on $(0,+\infty)$. 
As in the proof of Lemma \ref{Lem.g_near_0}, we have $\varphi(\lambda)=\psi_1(\lambda)W(\delta;\la)\mathrm{e}^{A_*(\delta;\la)} $ and $A_*(Z;\la)\in\mathcal H_{(0,Z_1)}(B_R)  $, 
then the result follows by taking $ f(Z;\la):=\mathrm{e}^{A_*(\delta;\la)}f_*(Z;\la)$ for $Z\in[0, +\infty),\la\in B_R$.
\end{proof}

Now we are in a position to prove Lemma \ref{Lem.L_laf=g}.

\begin{proof}[Proof of Lemma \ref{Lem.L_laf=g}]
We fix $Z_0\in(0,Z_1)$, $\psi_2(\lambda)=W(Z_0;\la)$, $\la_0^*:=\g-k-2$. Let $ \psi_1(\lambda)$ be given by Lemma \ref{Lem.Fundamental_sol_2} and $\varphi(\lambda):=\psi_1(\lambda)\psi_2(\lambda)$ for all $\la\in B_R$.
By Lemma \ref{Lem.Wronski_1}, we have $\la_0^*\in B_R$ and $\psi_2\in \operatorname{Hol}(B_R)\setminus\{0\}$. By Lemma \ref{Lem.Fundamental_sol_2} we have $\psi_1\in \operatorname{Hol}(B_R)\setminus\{0\}$. Thus, $\varphi=\psi_1\psi_2\in \operatorname{Hol}(B_R)\setminus\{0\}$.

	Let $\zeta\in C^\infty(\R;[0,1])$ satisfy $\operatorname{supp}\zeta\subset (Z_1/2, +\infty)$ and $\zeta(Z)=1$ for $Z\in[3Z_1/4, +\infty)$. Let $g_1(Z)=g(Z)(1-\zeta(Z)), g_2(Z)=g(Z){\zeta}(Z)$ for all $Z\in[0,+\infty)$. Then
	\[\operatorname{supp}g_1\subset[0, 3Z_1/4],\quad \operatorname{supp}g_2\subset[Z_1/2,+\infty),\quad g_1, g_2\in C_\text{e}^\infty([0,+\infty)),\quad g=g_1+g_2.\]
	By Lemma \ref{Lem.g_near_0}, there exists $f_1=f_1(Z;\la)\in\mathcal H_{[0,+\infty)}^{\operatorname{e}}(B_R)$ such that $\LL_\la f_1=\varphi(\lambda)g_1$ on $(0,+\infty)$. By Lemma \ref{Lem.g_away_0}, there exists $f_2=f_2(Z;\la)\in\mathcal H_{[0,+\infty)}^{\operatorname{e}}(B_R)$ such that $\LL_\la f_2=\varphi(\lambda)g_2$ on $(0,+\infty)$. Let $f=f_1+f_2$, then $f\in\mathcal H_{[0,+\infty)}^{\operatorname{e}}(B_R)$ satisfies $\LL_\la f=\varphi(\lambda)g$ on $(0,+\infty)$.
\end{proof}

\appendix
\section{The derivation and properties of ODE \eqref{Eq.ODE_v(Z)}}\label{Appen.v(Z)}

\subsection{The derivation of ODE \eqref{Eq.ODE_v(Z)}}\label{Subsec.v(Z)_derivation}

\begin{lemma}\label{Lem.leading_order_eq}
	Let $\beta>1$ and $v=v(Z)\in C^\infty([0, +\infty);(-1, 1))$ be given by Assumption \ref{Assumption}. We define $\hat\phi_0(Z), \hat\rho_0(Z)$ according to \eqref{Eq.phi_0_rho_0} and we define $\phi_0(t,x),\rho_0(t,x)$ by \eqref{Eq.phi_0_selfsimilar}. Then $(\phi_0,\rho_0)$ solves the leading order equation \eqref{Eq.wave_eq_n=0}.
\end{lemma}

\begin{proof}
	Recall that $Z=r/(T-t)$ with $r=|x|$, we know that $\phi_0=\phi_0(t,r)$ and $\rho_0=\rho_0(t,r)$ are radially symmetric. Hence, \eqref{Eq.wave_eq_n=0} is equivalent to
	\begin{align}\label{Eq.wave_eq_n=0_radial}
		\rho_0^{p-1}-|\partial_t\phi_0|^2+|\partial_r\phi_0|^2=0,\qquad -\partial_t(\rho_0^2\partial_t\phi_0)+\partial_r(\rho_0^2\partial_r\phi_0)+\frac{k}{r}\rho_0^2\partial_r\phi_0=0,
	\end{align}
	where $k=d-1\in\Z_{\geq 1}$. It follows from \eqref{Eq.phi_0_selfsimilar} that
	\begin{align}\label{phi0}\partial_t\phi_0(t,x)=(T-t)^{-\beta}\big((\beta-1)\hat\phi_0(Z)+Z\hat\phi_0'(Z)\big),\quad \partial_r\phi_0(t,x)=(T-t)^{-\beta}\hat\phi_0'(Z),\end{align}
	where the prime $'$ stands for the derivative with respect to $Z$. By \eqref{Eq.phi_0_rho_0}, we have
	\begin{equation}\label{Eq.v(Z)}
		\hat\phi_0'(Z)=\frac{(\beta-1)\hat\phi_0(Z)v(Z)}{1-Zv(Z)}\Longleftrightarrow\big((\beta-1)\hat\phi_0(Z)+Z\hat\phi_0'(Z)\big)v(Z)=\hat\phi_0'(Z),
	\end{equation}
	thus
	\[\partial_t\phi_0(t,x)=\frac{(T-t)^{-\beta}(\beta-1)\hat\phi_0(Z)}{1-Zv(Z)},\qquad \partial_r\phi_0(t,x)=\frac{(T-t)^{-\beta}(\beta-1)\hat\phi_0(Z)v(Z)}{1-Zv(Z)},\]
	and then we have
	\[|\partial_t\phi_0|^2-|\partial_r\phi_0|^2=\frac{(T-t)^{-2\beta}(\beta-1)^2\hat\phi_0(Z)^2(1-v(Z)^2)}{(1-Zv(Z))^2}.\]
	Using \eqref{Eq.phi_0_selfsimilar} and \eqref{Eq.phi_0_rho_0} for $\rho_0$ and $\hat\rho_0$, we obtain the first equation in \eqref{Eq.wave_eq_n=0_radial}.
	
	Now we define
	\begin{equation}\label{Eq.ell=4/(p-1)+1}
		\ell=\frac4{p-1}+1>1,\quad \tilde\phi_0(Z):=\frac{\hat\phi_0(Z)^\ell(1-v(Z)^2)^{\frac{2}{p-1}}}{(1-Zv(Z))^\ell}=\frac{\hat\phi_0(Z)^\ell(1-v(Z)^2)^{\frac{\ell-1}{2}}}{(1-Zv(Z))^\ell}>0.
	\end{equation}
	Then we compute that
	\begin{align*}
		\rho_0^2\partial_t\phi_0(t,x)&=(T-t)^{-\beta\ell}(\beta-1)^\ell\tilde\phi_0(Z),\\
		\rho_0^2\partial_r\phi_0(t,x)&=(T-t)^{-\beta\ell}(\beta-1)^\ell\tilde\phi_0(Z)v(Z),\\
		\partial_t(\rho_0^2\partial_t\phi_0)(t,x)&=
		(T-t)^{-\beta\ell-1}(\beta-1)^\ell\big(\beta\ell\tilde\phi_0(Z)+Z\tilde\phi_0'(Z)\big),\\
		\partial_r(\rho_0^2\partial_r\phi_0)(t,x)&=(T-t)^{-\beta\ell-1}(\beta-1)^\ell(\tilde\phi_0 v)'(Z),\\
		\frac{k}{r}\rho_0^2\partial_r\phi_0(t,x)&=(T-t)^{-\beta\ell-1}(\beta-1)^\ell\frac{k}{Z}(\tilde\phi_0 v)(Z).
	\end{align*}
	Therefore, the second equation in \eqref{Eq.wave_eq_n=0_radial} is equivalent to
	\begin{align}\label{Eq.A.5}
	\beta\ell\tilde\phi_0+Z\tilde\phi_0'=(\tilde\phi_0 v)'+\frac{k}{Z}(\tilde\phi_0 v)\Longleftrightarrow (\beta\ell-v'-kv/Z)\tilde\phi_0=(v-Z)\tilde\phi_0'.
	\end{align}
	Recall from \eqref{Eq.v(Z)} and \eqref{Eq.ell=4/(p-1)+1} that
	\begin{align*}
	\frac{\tilde\phi_0'}{\tilde\phi_0} &=\ell\frac{\hat\phi_0'}{\hat\phi_0}-(\ell-1)\frac{vv'}{1-v^2}+\ell\frac{v+Zv'}{1-Zv}\\
	&=\ell\frac{(\beta-1)v}{1-Zv}+\ell\frac{v}{1-Zv}-(\ell-1)\frac{vv'}{1-v^2}+\ell\frac{Zv'}{1-Zv}\\
	&=\frac{\beta\ell v}{1-Zv}+\frac{vv'}{1-v^2}+\frac{\ell(Z-v)v'}{(1-v^2)(1-Zv)},
	\end{align*}
	hence \eqref{Eq.A.5} is equivalent to
	\begin{align*}
		\beta\ell-v'-\frac{kv}{Z}=(v-Z)\frac{\tilde\phi_0'}{\tilde\phi_0}
		=\frac{\beta\ell v(v-Z)}{1-Zv}+\frac{v(v-Z)v'}{1-v^2}-\frac{\ell(v-Z)^2v'}{(1-v^2)(1-Zv)},
	\end{align*}
	or equivalently,
	\begin{align*}
		&\beta\ell-\frac{\beta\ell v(v-Z)}{1-Zv}-\frac{kv}{Z}
		=v'+\frac{v(v-Z)v'}{1-v^2}-\frac{\ell(v-Z)^2v'}{(1-v^2)(1-Zv)},\\
		&\frac{\beta\ell(1-v^2)}{1-Zv}-\frac{kv}{Z}
		=\frac{(1-Zv)v'}{1-v^2}-\frac{\ell(v-Z)^2v'}{(1-v^2)(1-Zv)},
	\end{align*}
	which is a direct consequence of \eqref{Eq.ODE_v(Z)}. 
\end{proof}

\subsection{Properties of solutions to ODE \eqref{Eq.ODE_v(Z)}}\label{Subsec.v(Z)_properties}

In this part, we prove Remark \ref{Rmk.v(Z)_properties}.

\begin{lemma}\label{lem8}
	Under Assumption \ref{Assumption}, we have $v(Z_1)=v_1$, where
	\begin{equation}\label{Eq.Z_1_v_1}
		Z_1=\frac{k}{\sqrt{\ell}(k-\beta(\ell-1))},\qquad v_1=\frac{\beta\sqrt{\ell}}{k-\beta\ell}.
	\end{equation}
\end{lemma}
\begin{proof}
	We define a function $F_0\in C^\infty([0, +\infty))$ by
	\begin{equation}\label{Eq.F_0}
		F_0(Z):=1-Zv(Z)+\sqrt\ell(v(Z)-Z),\quad\forall\ Z\in[0, +\infty).
	\end{equation}
	Then $F_0(0)=1$ and $F_0(1)=(\sqrt\ell-1)(v(1)-1)<0$, where we have used that $v(1)\in(-1, 1)$, recalling Assumption \ref{Assumption}. By the intermediate value theorem, there exists $Z_0\in(0,1)$ such that $F_0(Z_0)=0$. Thus,
	$\Delta_Z(Z_0, v(Z_0))=Z_0F_0(Z_0)\big(1-Z_0v(Z_0)-\sqrt\ell(v(Z_0)-Z_0)\big)=0$.
	Then we have $\Delta_v(Z_0, v(Z_0))=\Delta_Z(Z_0, v(Z_0))v'(Z_0)=0$, i.e., $\Delta_v(Z_0, v(Z_0))=\Delta_Z(Z_0, v(Z_0))=0$. On the other hand, it is direct  to check that
	\begin{equation}\label{Eq.Z_v_zero}
		\left\{(Z,v)\in (0,+\infty)\times(-1,1): \Delta_v(Z,v)=\Delta_Z(Z,v)=0\right\}=\{(Z_1, v_1)\},
	\end{equation}
	where $Z_1, v_1$ are given by \eqref{Eq.Z_1_v_1}. Moreover, it follows from \eqref{Eq.parameter_range} that $0<v_1<Z_1<1$. Hence, we must have $(Z_0, v(Z_0))=(Z_1, v_1)$, which implies that $v(Z_1)=v_1$.
\end{proof}

\begin{lemma}
	Under Assumption \ref{Assumption}, let $\Delta_0(Z):=\Delta_Z(Z, v(Z))$ for $Z\in[0, +\infty)$, then we have $\Delta_0(Z)>0$ for $Z\in(0, Z_1)$,  $\Delta_0(Z)<0$ for $Z\in(Z_1, +\infty)$ and $\Delta_0'(Z_1)\neq 0$.
\end{lemma}
\begin{proof}
	By the definition of $\Delta_0(Z)$, we have 
$\Delta_0\in C^\infty([0, +\infty))$ and (see \eqref{Eq.ODE_v(Z)})
	\begin{equation}\label{Eq.Delta_0_F_0}
		\Delta_0(Z)=ZF_0(Z)\wt F_0(Z),\quad \forall\ Z\in[0, +\infty),
	\end{equation}
	where $F_0\in C^\infty([0, +\infty))$ is defined by \eqref{Eq.F_0} and $\wt F_0\in C^\infty([0, +\infty))$ is defined by
	\begin{equation}\label{tF0}
		\wt F_0(Z):=1-Zv(Z)-\sqrt\ell(v(Z)-Z),\quad\forall\ Z\in[0, +\infty).
	\end{equation}
	
	If $\Delta_0(Z_*)=0$ for some $Z_*\in(0, +\infty)$, then $\Delta_v(Z_*, v(Z_*))=\Delta_0(Z_*)v'(Z_*)=0$, and by \eqref{Eq.Z_v_zero} we obtain $Z_*=Z_1$, hence (using Lemma \ref{lem8}, \eqref{Eq.Z_v_zero} and \eqref{Eq.Delta_0_F_0})
	\begin{equation}\label{Eq.Delta_0_zero}
		\{Z\in[0, +\infty):\Delta_0(Z)=0\}=\{0, Z_1\}.
	\end{equation}
	
Now we prove that $\Delta_0'(Z_1)\neq 0$. By the proof of Lemma \ref{lem8}, we have $F_0(Z_1)=0$. Then by \eqref{Eq.F_0} and \eqref{tF0}, we have $\wt F_0(Z_1)=2(1-Z_1v(Z_1))>0$. Thus (using \eqref{Eq.Delta_0_F_0}), $\Delta_0'(Z_1)=Z_1F_0'(Z_1)\wt F_0(Z_1) $. Assume on the contrary that $\Delta_0'(Z_1)=0$, then $F_0'(Z_1)=0$. 

Let  $\Delta_1(Z):=\Delta_v(Z, v(Z))$. Then \eqref{Eq.ODE_v(Z)} becomes $\Delta_0(Z)v'(Z)=\Delta_1(Z)$. Taking derivative with respect to $Z$ at $Z=Z_1$ on both sides, we obtain(using $\Delta_0(Z_1)=0$) $\Delta_1'(Z_1)=\Delta_0'(Z_1)v'(Z_1)=0$ 
and $\Delta_1(Z_1)=0$. By \eqref{Eq.ODE_v(Z)}, we have $\Delta_1(Z)=F_1(Z)\wt F_1(Z)$ with $F_1(Z):=Z-g_1(v(Z)) $, $\wt F_1(Z):=(1-v(Z)^2)(\beta \ell+(k-\beta \ell)v(Z)^2) $, $g_1(v):=kv/(\beta \ell+(k-\beta \ell)v^2) $ (note that $0<\beta \ell<k$ using \eqref{Eq.parameter_range}). As $v(Z_1)\in(-1,1)$, we have $\wt F_1(Z_1)>0$, then by $\Delta_1'(Z_1)=0$ 
and $\Delta_1(Z_1)=0$ we have $F_1(Z_1)=0$, $\Delta_1'(Z_1)=F_1'(Z_1)\wt F_1(Z_1)=0$ and $F_1'(Z_1)=0$. Thus, $0=1-g_1'(v(Z_1))v'(Z_1)=1-g_1'(v_1)v'(Z_1) $ (using Lemma \ref{lem8}).

Similarly, by  \eqref{Eq.F_0}, we have $F_0(Z)=F_2(Z)\wt F_2(Z)$ with $F_2(Z):=g_2(v(Z))-Z $, $\wt F_2(Z):=v(Z)+\ell$, $g_2(v):=(1+\sqrt\ell v)/(v+\sqrt\ell) $, and $\wt F_2(Z_1)>0$. Thus, $F_2(Z_1)=0$, $F_2'(Z_1)=0$ (using $F_0(Z_1)=F_0'(Z_1)=0$) and $0=g_2'(v_1)v'(Z_1)-1$.

Now we have $1=g_1'(v_1)v'(Z_1)=g_2'(v_1)v'(Z_1)$ and $g_1'(v_1)=g_2'(v_1) $. On the other hand,
\begin{align*}
		g_1(v)-g_2(v)&=\frac{kv}{\beta \ell+(k-\beta \ell)v^2}-\frac{1+\sqrt\ell v}{v+\sqrt\ell}=
\frac{(1-v^2)(kv\sqrt\ell-\beta \ell(1+\sqrt\ell v)}{(\beta \ell+(k-\beta \ell)v^2)(v+\sqrt\ell)}\\
&=\frac{(1-v^2)(k-\beta \ell)\sqrt\ell(v-v_1)}{(\beta \ell+(k-\beta \ell)v^2)(v+\sqrt\ell)},
	\end{align*}
here we used \eqref{Eq.Z_1_v_1}, thus\begin{align*}
		g_1'(v_1)-g_2'(v_1)&=\frac{(1-v_1^2)(k-\beta \ell)\sqrt\ell}{(\beta \ell+(k-\beta \ell)v_1^2)(v_1+\sqrt\ell)}>0,
	\end{align*}
which is a contradiction. Therefore, $\Delta_0'(Z_1)\neq0$.

	By $F_0(0)=\wt F_0(0)=1>0$, \eqref{Eq.Delta_0_F_0}, {and \eqref{Eq.Delta_0_zero}}, we have $\Delta_0(Z)>0$ for $Z\in(0, Z_1)$. Finally, using $\Delta_0'(Z_1)\neq0$ 
and {\eqref{Eq.Delta_0_zero}}, we have $\Delta_0'(Z_1)<0$ and $\Delta_0(Z)<0$ for all $Z\in(Z_1, +\infty)$.
\end{proof}

To finish the proof of Remark \ref{Rmk.v(Z)_properties}, it remains to show that $v(Z)<Z$ and $Zv(Z)<1$ for all $Z\in(0, +\infty)$. We use the barrier function method. For any $V\in C^1((0, +\infty))$, we define
\begin{equation}\label{Eq.Av}
	(\mathscr AV)(Z):=-\Delta_Z(Z, V(Z))V'(Z)+\Delta_v(Z, V(Z)),\quad\forall\ Z\in(0, +\infty).
\end{equation}
Then $\mathscr Av=0$ if $v$ is given by Assumption \ref{Assumption}.

\begin{lemma}
	Under Assumption \ref{Assumption}, we have $v(Z)<Z$ for all $Z\in(0, +\infty)$.
\end{lemma}
\begin{proof}
	Since $v(Z)\in(-1, 1)$ for all $Z\in(0, +\infty)$ by Assumption \ref{Assumption}, it suffices to prove $v(Z)<Z$ for all $Z\in(0, 1)$. We first show that $v(Z)<Z$ for all $Z\in(Z_1, 1)$. By \eqref{Eq.Delta_0_F_0}, $F_0(0)=1, F_0(Z_1)=0$ and \eqref{Eq.Delta_0_zero}, we have 
	\[\{Z\in[0, +\infty): F_0(Z)=0\}=\{Z_1\}.\]
	As $F_0(1)<0$ and $Z_1\in(0,1)$, we have $F_0(Z)<0$ for all $Z\in(Z_1, 1)$, hence
	\[v(Z)<\frac{\sqrt\ell Z-1}{\sqrt{\ell}-Z}<Z,\quad\forall\ Z\in(Z_1, 1).\]
	
	Next we prove that $v(Z)<Z$ for all $Z\in(0, Z_1)$. Let $V_1(Z):=Z$ for all $Z\in[0, +\infty)$, then we have
	\begin{equation}\label{Eq.AV_1}
		(\mathscr AV_1)(Z)=(\beta\ell-k-1)Z(1-Z^2)^2<0,\quad \forall\ Z\in(0, 1),
	\end{equation}
	where we have used $\beta\ell-k-1<0$, which follows from \eqref{Eq.parameter_range}. On the other hand, letting $Z\to0+$ in $\mathrm dv/\mathrm dZ=\Delta_v(Z, v(Z))/\Delta_Z(Z,v(Z))$, by L'H\^{o}pital's rule, we have
	\[v'(0)=\frac{\pa_Z\Delta_v(0, 0)+\pa_v\Delta_v(0,0)v'(0)}{\pa_Z\Delta_Z(0, 0)+\pa_v\Delta_Z(0,0)v'(0)}=\beta\ell-kv'(0),\]
	hence $v'(0)=\beta\ell/(k+1)<1=V_1'(0)$. As $v(0)=0=V_1(0)$, there exists $\delta\in(0, Z_1)$ such that $v(Z)<V_1(Z)$ for all $Z\in(0, \delta)$. Assume for contradiction that $Z_*\in(0, Z_1)\subset(0, 1)$ satisfies $v(Z)<V_1(Z)$ for all $Z\in(0, Z_*)$ and $v(Z_*)=V_1(Z_*)=Z_*$, then $v'(Z_*)\geq V_1'(Z_*)$. Thus, by $\Delta_Z(Z_*, V_1(Z_*))=\Delta_Z(Z_*, v(Z_*))=Z_*(1-Z_*^2)^2>0$, we have
	\begin{align*}
		(\mathscr AV_1)(Z_*)&=-\Delta_Z(Z_*, V_1(Z_*))V_1'(Z_*)+\Delta_v(Z_*, V_1(Z_*))\\
		&\geq -\Delta_Z(Z_*, v(Z_*))v'(Z_*)+\Delta_v(Z_*, v(Z_*))=(\mathscr Av)(Z_*)=0,
	\end{align*}
	which contradicts with \eqref{Eq.AV_1}. Therefore, we have $v(Z)<V_1(Z)=Z$ for all $Z\in(0, Z_1)$.
	
	Finally, by $v(Z_1)=v_1<Z_1<1$, we obtain $v(Z)<Z$ for all $Z\in (0, 1)$.
\end{proof}

\begin{lemma}\label{Lem.Zv(Z)<1}
	Under Assumption \ref{Assumption}, we have $v(Z)<1/Z$ for all $Z\in(0, +\infty)$.
\end{lemma}
\begin{proof}
	Since $v(Z)\in(-1, 1)$ for all $Z\in(0, +\infty)$ by Assumption \ref{Assumption}, it suffices to prove $v(Z)<1/Z$ for all $Z\in(1, +\infty)$. Let $V_2(Z):=1/Z$ for $Z\in(0, +\infty)$, then we have
	\begin{equation}\label{Eq.AV_2}
		(\mathscr AV_2)(Z)=(\beta-1)\ell Z\left(1-1/{Z^2}\right)^2>0,\quad\forall\ Z\in(1,+\infty),
	\end{equation}
	where we have used $\beta>1$ by \eqref{Eq.parameter_range}. As $v(Z_1)=v_1<1=V_2(Z_1)$, there exists $\delta>0$ such that $v(Z)<V_2(Z)$ for $Z\in[Z_1, Z_1+\delta)$. Assume for contradiction that $Z^*\in (1, +\infty)$ satisfies $v(Z)<V_2(Z)$ for all $Z\in(Z_1, Z^*)$ and $v(Z^*)=V_2(Z^*)$, then $v'(Z^*)\geq V_2'(Z^*)$. Thus, by $\Delta_Z(Z^*, V_2(Z^*))=\Delta_Z(Z^*, v(Z^*))=-\ell(1-Z_*^2)^2/Z_*<0$, we have
	\begin{align*}
		(\mathscr AV_2)(Z^*)&=-\Delta_Z(Z^*, V_2(Z^*))V_2'(Z^*)+\Delta_v(Z^*, V_2(Z^*))\\
		&\leq-\Delta_Z(Z^*, v(Z^*))v'(Z^*)+\Delta_v(Z^*, v(Z^*))=(\mathscr Av)(Z^*)=0,
	\end{align*}
	which contradicts with \eqref{Eq.AV_2}. Therefore, $v(Z)<V_2(Z)=1/Z$ for all $Z\in(1,+\infty)$.
\end{proof}

The proof of Remark \ref{Rmk.v(Z)_properties} is completed now. To conclude this appendix, we prove that $\hat\phi_0, \hat\rho_0\in C_{\operatorname{e}}^\infty([0,+\infty))$, where $\hat\phi_0$ and $\hat\rho_0$ are defined by \eqref{Eq.phi_0_rho_0}.

\begin{lemma}\label{Lem.v_hat_rho}
	Let $v\in C_{\operatorname{o}}^\infty([0, +\infty))$ be given by Assumption \ref{Assumption}, and define $\hat\phi_0, \hat\rho_0$ by \eqref{Eq.phi_0_rho_0}. Then we have $\hat\phi_0, \hat\rho_0\in C_{\operatorname{e}}^\infty([0,+\infty))$.
\end{lemma}
\begin{proof}
	We first claim that
	\begin{equation}\label{Eq.o_integration_to_e}
		f\in C_{\text{o}}^\infty([0, +\infty))\Longrightarrow F(Z):=\int_0^Zf(s)\,\mathrm ds\in C_{\text{e}}^\infty([0, +\infty)).
	\end{equation}
	
	Now we prove that $\hat\phi_0\in C_{\text{e}}^\infty([0,+\infty))$. By $v\in C_{\text{o}}^\infty([0, +\infty))$, we have $1-Zv(Z)\in C_{\text{e}}^\infty([0,+\infty))$. Since $Zv(Z)<1$ for all $Z\in[0,+\infty)$ by Lemma \ref{Lem.Zv(Z)<1}, it follows from \eqref{Eq.reciprocal} that $0<\frac1{1-Zv(Z)}\in C_{\text{e}}^\infty([0,+\infty))$, hence by $v\in C_{\text{o}}^\infty([0, +\infty))$ and \eqref{Eq.eo_to_o} we have $\frac{v(Z)}{1-Zv(Z)}\in C_{\text{o}}^\infty([0, +\infty))$, then by \eqref{Eq.o_integration_to_e} we obtain
	\[(\beta-1)\int_0^Z \frac{v(s)}{1-sv(s)}\,\mathrm ds\in C_{\text{e}}^\infty([0, +\infty)).\]
	Thus, by \eqref{Eq.exponential} and \eqref{Eq.phi_0_rho_0} we have $\hat\phi_0\in C_{\text{e}}^\infty([0,+\infty))$. 
	
	As for $\hat\rho_0$, by $0<\frac1{1-Zv(Z)}\in C_{\text{e}}^\infty([0,+\infty))$, 
and \eqref{Eq.power}, we have
	$1/{(1-Zv(Z))^{\frac2{p-1}}}\in C_{\text{e}}^\infty([0,+\infty)).$
	Similarly, using $\hat\phi_0\in C_{\text{e}}^\infty([0,+\infty))$ and $\hat\phi_0(Z)>0$ for all $Z\in[0, +\infty)$, we get $\hat\phi_0(Z)^{2/(p-1)}\in C_{\text{e}}^\infty([0,+\infty))$.
	It follows from $1-v(Z)^2\in C_{\text{e}}^\infty([0,+\infty))$, $v(Z)\in (-1, 1)$ for all $Z\in[0, +\infty)$ and \eqref{Eq.power} that
	$(1-v(Z)^2)^{\frac1{p-1}}\in C_{\text{e}}^\infty([0,+\infty)).$ Therefore, by \eqref{Eq.phi_0_rho_0} and \eqref{Eq.ee_to_e}, we have $\hat\rho_0\in C_{\text{e}}^\infty([0,+\infty))$.
	
	Finally, it suffices to show the claim \eqref{Eq.o_integration_to_e}. By \eqref{Co}, there exists $\tilde f\in C^\infty([0, +\infty))$ such that $f(Z)=Z\tilde f(Z^2)$ for all $Z\in [0, +\infty)$. Let
	\[\tilde F(Z):=\frac12\int_0^Z\tilde f(s)\,\mathrm ds,\quad\forall\ Z\in[0, +\infty),\]
	then $\tilde F\in C^\infty([0, +\infty))$. Moreover, we have
	\[F(Z)=\int_0^Zf(s)\,\mathrm ds=\int_0^Z s\tilde f(s^2)\,\mathrm ds=\frac12\int_0^{Z^2}\tilde f(s)\,\mathrm ds=\tilde F(Z^2),\quad\forall\ Z\in[0, +\infty).\]
	Hence by \eqref{Ce}, we have $F\in C_{\text{e}}^\infty([0,+\infty))$.
\end{proof}

\section{Linear ODEs with singular points}\label{Appen.ODE}

In this appendix, we establish the well-posedness theory for a class of second order linear ODEs with singular points.

First of all, we introduce a preliminary lemma, which ensures that the functions we are considering are smooth in the sense of multi-variable functions. Let $I\subset\R$ be an interval and let $\Omega\subset \C$ be an open subset. We define
\begin{align}\label{Hol1}
	\operatorname{Hol}(\Omega):&=\{x=x(\la) \text{ is holomorphic (or equivalently, analytic) on }\Omega\},\\
	\mathcal H_I^{0}(\Omega):&=\big\{x=x(t;\la)\in C(I\times\Omega;\C): x(\cdot;\la)\in C^\infty(I) \text{ for all }\la\in\Omega,\\
	 &\qquad x(t;\cdot)\in\operatorname{Hol}(\Omega) \text{ for all }t\in I \text{ and $\pa_t^jx\in L^\infty(I\times\Omega)$ for all $j\in\Z_{\geq 0}$}\big\},\nonumber\\
	 \mathcal{H}_I(\Omega):&=\left\{x=x(t;\la)\in C^\infty(I\times\Omega;\C): x(t;\cdot)\in\operatorname{Hol}(\Omega)\text{ for all }t\in I\right\}.\label{Space.H_I(Omega)}
\end{align}
Then $ \operatorname{Hol}(\Omega)$, $ \mathcal H_I^{0}(\Omega)$, $ \mathcal H_I(\Omega)$ are rings and the definitions in \eqref{Hol1}, \eqref{Space.H_I(Omega)} are the same as in section 5.3. This appendix is only used in the proof of Lemma \ref{Lem.L_laf=g}, which does not require the definition of $Z$ in $Z=|x|/(T-t)$. 

So, with abuse of notations, we replace $Z$ by $t$ and use $x$ to denote a general function of $(t;\la)$. We stress that here $(t,x)$ has nothing to do with the coordinates in $\R^{1+d}$.
\begin{lemma}\label{Lem.smooth_multi_variable}
	Let $I\subset\R$ be an interval and $\Omega\subset \C$ be an open subset. Then $\mathcal H_I^0(\Omega)\subset \mathcal H_I(\Omega)$.
\end{lemma}
\begin{proof}
	Let $x=x(t;\la)\in \mathcal H_I^0(\Omega)$. Pick $\la_0\in \Omega$ and let $r\in(0,1)$ be such that $B_r(\la_0):=\{\la\in\C:|\la-\la_0|<r\}\subset\Omega$. By Cauchy's integration formula (Theorem 4.4 in Chapter 2 of \cite{Stein}), for any $t\in I, \la\in\Omega$ we have
	\begin{equation}\label{Eq.Appen_x_expansion}
		x(t;\la)=\sum_{k=0}^\infty x_k(t)(\la-\la_0)^k,
	\end{equation}
	where
	\begin{equation}\label{Eq.Appen_x_coefficient}
		x_k(t)=\frac1{2\pi\ii}\int_{|\la-\la_0|=r}\frac{x(t;\la)}{(\la-\la_0)^{k+1}}\,\mathrm d\la=\frac{r^{-k}}{2\pi}\int_0^{2\pi}x\left(t;\la_0+re^{\ii\theta}\right)e^{-\ii k\theta}\,\mathrm d\theta
	\end{equation}
	for all  $t\in I, k\in\Z_{\geq 0}$. Since $x(\cdot;\la)\in C^\infty(I)$ for all $\la\in\Omega$ and $\pa_t^jx\in L^\infty(I\times\Omega)$, by \eqref{Eq.Appen_x_coefficient} and the dominated convergence theorem, we have $x_k\in C^\infty(I)$ and
	\begin{equation}\label{Eq.Appen_x_k_bound}
		\left\|x_k^{(j)}\right\|_{L^\infty(I)}\leq \|\pa_t^jx\|_{L^\infty(I\times\Omega)}r^{-k},\quad \forall\ j\in\Z_{\geq 0},\ \forall\ k\in\Z_{\geq 0}.
	\end{equation}
	Using \eqref{Eq.Appen_x_k_bound}, we know that $\sum_{k=0}^\infty x_k^{(j)}(t)\pa_\la^\alpha((\la-\la_0)^k)$ is uniformly absolutely convergent on $I\times B_{r/2}(\la_0)$ for all $j\in\Z_{\geq 0}$ and $\alpha\in(\Z_{\geq 0})^2$, hence \eqref{Eq.Appen_x_expansion} implies that $x\in C^\infty(I\times B_{r/2}(\la_0))$. Since $\la_0\in\Omega$ is arbitrary, we have $x\in C^\infty(I\times\Omega)$. Hence $x\in\mathcal{H}_I(\Omega)$.
\end{proof}

\begin{remark}\label{Rmk.smooth_loc}
	As smoothness is a local property, we have $\mathcal{H}_{I,\text{loc}}^0(\Omega)\subset \mathcal H_I(\Omega)$, where
	\begin{equation*}
		\mathcal{H}_{I,\text{loc}}^0(\Omega):=\left\{x=x(t;\la):I\times\Omega\to\C\Big|x\in \mathcal H_{J}^0(\Omega) \text{ for any compact sub-interval }J\subset I\right\}.
	\end{equation*}
	Moreover, we have $\mathcal H_I(\Omega)\subset \mathcal H_{I,\text{loc}}^0(\Omega')$ for any open subset $\Omega'\subset\subset\Omega$ (i.e. there exists a compact set $K$ such that $\Omega'\subset K\subset \Omega$).
\end{remark}

\begin{lemma}\label{Lem.Appen_1}
	Let $I\subset \R$ be an interval and let $A(t)\in C^\infty(I; \C)$ be such that $A(t)=0$ has a unique solution $t=t_0$ in $I$\footnote{It means that $\{t\in I:A(t)=0\}=\{t_0\}$, and also for Proposition \ref{Prop.Appen_1}.} with $A'(t_0)\neq0$. Let $\Omega\subset \C$ be an open subset and let $B(t;\la), D(t;\la)\in C^\infty(I\times\Omega;\C)$ be such that $B, D\in \mathcal H_{I,\operatorname{loc}}^0(\Omega)$, which implies
	\begin{equation*}
		N_0^*:=\sup_{\la\in\Omega}\left(\max\left\{3, -\Re\left(\frac{B(t_0;\la)}{A'(t_0)}\right)+1\right\}\right)<+\infty.
	\end{equation*}
	Then there exists $N_0>N_0^*$
	such that for all $N\in \Z\cap(N_0, +\infty)$, if $f\in \mathcal{H}_{I,\operatorname{loc}}^0(\Omega)$ is such that
	\begin{equation}\label{Eq.Appen_f_high}
		\frac{|f(t;\la)|}{|t-t_0|^N}\in L^\infty(J\times\Omega) \quad\text{ for any compact sub-interval }\quad J\subset I,
	\end{equation}
	then the linear ODE (here the prime $'$ refers to the derivative with respect to $t$)
	\begin{equation}\label{Eq.Appen_ODE}
		A(t)x''(t;\la)+B(t;\la)x'(t;\la)+D(t;\la)x(t;\la)=f(t;\la)
	\end{equation}
	has a (complex-valued) smooth solution $x(t;\la)$ on $I\times \Omega$ such that $x\in\mathcal H_{I}(\Omega)$. 
\end{lemma}
\begin{proof}
	Without loss of generality, we assume that $t_0=0\in {I}$, and there exists $\delta_0\in(0, 1)$ such that $I_0=[-\delta_0, \delta_0]\subset I$ or $I_0=[0, \delta_0]= I\cap [-\delta_0, \delta_0]$.
	
{\textbf{Step 1.}} Existence of a $C^2$ local solution. We define the Banach space
	\begin{equation*}
		Y_N:=\left\{y\in C(I_0\times\Omega;\C): {y(t;\la)}/{|t|^N}\in L^\infty(I_0\times\Omega) \text{ and $y(t;\cdot)\in\operatorname{Hol}(\Omega)\ \forall\ t\in I_0$}\right\},
	\end{equation*}
	where $N\geq 3$ is an integer, with the norm
	$\|y\|_{Y_N}:=\left\|{y(t;\la)}/{|t|^N}\right\|_{L^\infty(I_0\times\Omega)}.$

	We define a linear operator $\mathcal T_N: Y_N\to Y_N$ by
	\[(\mathcal T_Ny)(t;\la):=\int_0^t\left(\frac{B(s;\la)}{A(s)}y(s;\la)+\frac{D(s;\la)}{A(s)}\int_0^sy(\tau;\la)\,\mathrm d\tau\right)\,\mathrm ds, \qquad \forall\ t\in I_0,\ \forall\ \la\in\Omega.\]
	By the hypotheses on the coefficients $A, B, D$, we have
	\begin{equation}\label{Eq.Appen_M}
		M:=\sup_{s\in I_0}\left|\frac s{A(s)}\right|+\sup_{s\in I_0,\la\in\Omega}\left|\frac{s B(s;\la)}{A(s)}\right|+\sup_{s\in I_0,\la\in\Omega}\left|\frac{s D(s;\la)}{A(s)}\right|\in (0, +\infty).
	\end{equation}
	Hence, for all $t\in I_0\subset[-1,1]$ and for all $\la\in \C$ we have
	\begin{align*}
		\left|(\mathcal T_Ny)(t;\la)\right|&=\left|\int_0^t\left(\frac{s B(s;\la)}{A(s)}\frac{y(s;\la)}{s^N}s^{N-1}+\frac{s D(s;\la)}{A(s)}\frac1s\int_0^s\frac{y(\tau;\la)}{\tau^N}\tau^N\,\mathrm d\tau\right)\,\mathrm ds\right|\\
		&\leq \frac{M}{N}|t|^N\|y\|_{Y_N},
	\end{align*}
	which gives
	\begin{equation}\label{Eq.Appen_T_N}
		\left\|\mathcal T_N\right\|_{Y_N\to Y_N}\leq \frac{M}{N},\qquad\forall\ N\in\Z\cap[3,+\infty).
	\end{equation}
	We also define an operator $\mathcal F_N:Y_N\to Y_N$ by
	\[(\mathcal F_Nf)(t;\la):=\int_0^t\frac{f(s;\la)}{A(s)}\,\mathrm ds, \qquad \forall\ t\in I_0,\ \forall\ \la\in\Omega.\]
	Using \eqref{Eq.Appen_M}, we know that $\mathcal F_N:Y_N\to Y_N$ is a bounded linear operator with $\left\|\mathcal F_Nf\right\|_{Y_N}\leq \frac{M}{N}\|f\|_{Y_N}$ for all $f\in Y_N$.	Now we take $N_0\in\Z$ such that $N_0>N_0^*+2M$.
	For any $N\in \Z\cap(N_0, +\infty)$, by \eqref{Eq.Appen_T_N} we know that $\|\mathcal T_N\|_{Y_N\to Y_N}\leq 1/2$, hence $\operatorname{id}+\mathcal{T}_N:Y_N\to Y_N$ is invertible, then $(\operatorname{id}+\mathcal T_N)^{-1}\mathcal F_N: Y_N\to Y_N$ is a bounded linear operator with
	\begin{equation*}
		\left\|(\operatorname{id}+\mathcal T_N)^{-1}\mathcal F_N\right\|_{Y_N\to Y_N}\leq \left\|(\operatorname{id}+\mathcal T_N)^{-1}\right\|_{Y_N\to Y_N}\left\|\mathcal F_N\right\|_{Y_N\to Y_N}\leq \frac{2M}{N}.
	\end{equation*}
	For any $N\in \Z\cap(N_0, +\infty)$, given $f\in \mathcal{H}_{I,\text{loc}}^0(\Omega)$ satisfying \eqref{Eq.Appen_f_high} (then $f\in Y_N$), we define $$y=(\operatorname{id}+\mathcal T_N)^{-1}\mathcal F_Nf\in Y_N,\quad x(t;\la)=\int_0^ty(s;\la)\,\mathrm ds,\ \forall\ t\in I_0,\ \forall\ \la\in\Omega,$$ then $x\in C(I_0\times\Omega)$, $x(\cdot;\la)\in C^1(I_0)$ for all $\la\in\Omega$, $x(t;\cdot)\in \operatorname{Hol}(\Omega)$ for all $t\in I_0$,
	\[x'(t;\la)=\int_0^t\left(-\frac{B(\tau;\la)}{A(\tau)}x'(\tau;\la)-\frac{D(\tau;\la)}{A(\tau)}x(\tau;\la)+
\frac{f(\tau;\la)}{A(\tau)}\right)\,\mathrm d\tau, \quad \forall\ t\in I_0,\ \forall\ \la\in\Omega,\]
and $x'(t;\la)=y(t;\la)$, $x(t;\la)=\int_0^tx'(s;\la)\,\mathrm ds$ for $t\in I_0$ and $\la\in\Omega$.	Moreover, we have (recalling $t_0=0$)
	\begin{equation}\label{Eq.Appen_ODE_Sol_1}
		\frac{x(t;\la)}{|t|^{N+1}}\in L^\infty(I_0\times\Omega),\qquad \frac{x'(t;\la)}{|t|^{N}}\in L^\infty(I_0\times\Omega).
	\end{equation}
	On the other hand, since $x'(t;\la)=\int_0^t X(s;\la)\,\mathrm ds$ for $t\in I_0, \la\in\Omega$, where
	\begin{align*}
		X(s;\la):&=-\frac{B(s;\la)}{A(s)}x'(s;\la)-\frac{D(s;\la)}{A(s)}x(s;\la)+\frac{f(s;\la)}{A(s)}\\
		&=-\frac{sB(s;\la)}{A(s)}\frac{x'(s;\la)}{s^N}s^{N-1}-\frac{sD(s;\la)}{A(s)}\frac{x(s;\la)}{s^{N+1}}s^N+
\frac{s}{A(s)}\frac{f(s;\la)}{s^N}s^{N-1}
	\end{align*}
	for $s\in I_0\setminus\{0\}$ and $\la\in\Omega$, thus there exists a constant $C>0$ such that we have $|X(s;\la)|\leq C|s|^{N-1}$ for $s\in I_0\setminus\{0\}, \la\in\Omega$. As $N>1$, we know that $x'(\cdot;\la)\in C^1(I_0)$ (thus $x(\cdot;\la)\in C^2(I_0)$) and $x''(0;\la)=0$ for $\la\in\Omega$. Hence, $x$ solves \eqref{Eq.Appen_ODE} on $(t,\la)\in I_0\times\C$. We also have
	\begin{equation}\label{Eq.B.10}
		\frac{x''(t;\la)}{|t|^{N-1}}\in L^\infty(I_0\times\Omega).
	\end{equation}

{\textbf{Step 2.} Smoothness of the $C^2$ local solution.} In this step, we show that $x(\cdot;\la)\in C^\infty(I_0)$ for any $\la\in\Omega$. By standard ODE theory, we have $x(\cdot;\la)\in C^\infty(I_0\setminus\{0\})$ for all $\la\in\Omega$. We claim that for any $k\in\Z\cap[0, N]$, there exists a constant $C_k>0$ such that
	\begin{equation}\label{Eq.Appen_induc_1}
		x^{(k)}(0;\la)=0\qquad\text{and}\qquad |x^{(k)}(t;\la)|\leq C_k|t|^{N+1-k},\qquad\forall \ t\in I_0,\ \forall\ \la\in\Omega.
	\end{equation}
	
	We use the induction. By \eqref{Eq.Appen_ODE_Sol_1} and \eqref{Eq.B.10}, we know that \eqref{Eq.Appen_induc_1} holds for $k\in\{0,1,2\}$. Assume that for some $K\in\Z\cap[1, N-1]$, \eqref{Eq.Appen_induc_1} holds for all $k\in\Z\cap[0, K]$. Now we prove that \eqref{Eq.Appen_induc_1} holds for $k=K+1$. By our induction hypotheses, $x(\cdot;\la)\in C^{(K)}(I_0)$ and $x^{(K+1)}(0;\la)=\lim_{t\to0}(x^{(K)}(t;\la)/t)=0$ for all $\la\in\Omega$. For $t\in I_0\setminus\{0\}$, taking derivative $K-1$ times on both sides of \eqref{Eq.Appen_ODE} with respect to $t$, we obtain
	\begin{equation}\label{Eq.Appen_k-1_derivative}
		A(t)x^{(K+1)}(t;\la)+\sum_{j=0}^{K}A_{j, K}(t;\la)x^{(j)}(t;\la)=f^{(K-1)}(t;\la),\quad \forall\ t\in I_0\setminus\{0\},\ \forall\ \la\in\Omega,
	\end{equation}
	where $A_{j,K}$'s are linear combinations of $A, B, D$ and their derivatives, hence $A_{j,K}(t;\la)\in C^\infty (I\times\Omega)\cap L^\infty(I_0\times\Omega)$ for all $j\in\Z\cap[0, K]$.\footnote{\label{14}\eqref{Eq.Appen_k-1_derivative} and the properties of $A_{j,K}$'s holds for all $K\in\Z_+$ (not merely for $K\in\Z\cap[1, N-1]$), and we also have $A_{K, K}(t;\la)=(K-1)A'(t)+B(t;\la)$ for all $t\in I_0, \la\in\Omega$.} As $f\in \mathcal H_{I,\text{loc}}^0(\Omega)$ satisfies \eqref{Eq.Appen_f_high}, we have\footnote{Indeed, \eqref{Eq.Appen_f_high} implies that $f^{(k)}(0;\la)=0$ for any $k\in\Z\cap[0, N-1]$ and any $\la\in\Omega$. As a consequence, we have $|f^{(N-1)}(t;\la)|=\left|\int_0^tf^{(N)}(s;\la)\,\mathrm ds\right|\leq \left(\sup_{s\in I_0,\la\in\Omega}|f^{(N)}(s;\la)|\right)|t|$ for all $t\in I_0,  \la\in\Omega$, where we have used $f\in \mathcal H_{I_0}^0(\Omega)$. Similarly one shows that $f^{(k)}(t;\la)/|t|^{N-k}\in L^\infty(I_0\times\Omega)$ for all $k\in \Z\cap[0, N]$.} $f^{(k)}(t;\la)/|t|^{N-k}\in L^\infty(I_0\times\Omega)$ for all $k\in \Z\cap[0, N]$. Therefore,
	\[\frac{\left|x^{(K+1)}(t;\la)\right|}{|t|^{N-K}}=\frac{|t|}{|A(t)|}\frac{\left|f^{(K-1)}(t;\la)-\sum_{j=0}^{K}A_{j, K}(t;\la)x^{(j)}(t;\la)\right|}{|t|^{N-(K-1)}}\in L^\infty(I_0\times\Omega).\]
	This proves \eqref{Eq.Appen_induc_1} for $k=K+1$. Hence, \eqref{Eq.Appen_induc_1} holds by the induction and thus $x(\cdot;\la)\in C^N(I_0)$ for all $\la\in\Omega$ and
	\begin{equation}\label{Eq.B.14}
		\sup_{t\in I_0,\la\in\Omega}|x^{(k)}(t;\la)|<+\infty,\quad\forall\ k\in\Z\cap[0, N].
	\end{equation}
	
	Next we claim that for $k\in\Z\cap[N,+\infty)$ we have
	\begin{equation}\label{Eq.Appen_induc_2}
		\sup_{t\in I_0\setminus\{0\},\la\in\Omega}|x^{(k)}(t;\la)|<+\infty.
	\end{equation}
	By \eqref{Eq.Appen_induc_1}, we know that \eqref{Eq.Appen_induc_2} holds for $k=N$. Assume that for some $k\in\Z_{\geq N}$ we have
	\begin{equation}\label{Eq.Appen_induc_hypo}
		\sup_{t\in I_0\setminus\{0\},\la\in\Omega}|x^{(N)}(t;\la)|<+\infty,\cdots, \sup_{t\in I_0\setminus\{0\},\la\in\Omega}|x^{(k)}(t;\la)|<+\infty.
	\end{equation}
	For $t\in I_0\setminus\{0\}$ and $\la\in\Omega$, by \eqref{Eq.Appen_k-1_derivative} for $K=k+1$ and footnote \ref{14} we have
	\begin{align}
		A(t)x^{(k+2)}&(t;\la)+(kA'(t)+B(t;\la))x^{(k+1)}(t;\la)=F_k(t;\la),\label{Eq.Appen_k_derivative}\\
		&F_k(t;\la):=f^{(k)}(t;\la)-\sum_{j=0}^{k}A_{j, k+1}(t;\la)x^{(j)}(t;\la).\label{Fk}
	\end{align}
	Then by $f\in \mathcal H_{I_0}^0(\Omega)$, \eqref{Eq.B.14} and \eqref{Eq.Appen_induc_hypo} we have $\sup_{t\in I_0\setminus\{0\},\la\in\Omega}|F_k(t;\la)|<+\infty$. Let $\widetilde A(t):=\int_0^1A'(ts)\,\mathrm ds$ and 
$\widetilde B(t;\la):=B(t;\la)-\frac{B(0;\la)}{A'(0)}\widetilde A(t)$ for $t\in I_0, \la\in\Omega$,  then $\widetilde B(0;\la)=0$, $\widetilde A\in C^{\infty}(I_0)$, $\widetilde B\in C(I_0\times\Omega)$ for $\la\in\Omega$, $\partial_t\widetilde B\in L^{\infty}(I_0\times\Omega) $ and $\widetilde A(t)={A(t)}/{t}$ for $ t\in I_0\setminus\{0\}$. Thus, $\widetilde B(t;\la)/t\in L^{\infty}(I_0\times\Omega)$ and $\widetilde B(t;\la)/A(t)\in L^{\infty}(I_0\times\Omega)$ (using \eqref{Eq.Appen_M}). 

Let $\eta(t;\la):=|t|^{\frac{B(0;\la)}{A'(0)}}\exp\left(\int_0^t\frac{\widetilde B(s;\la)}{A(s)}\,\mathrm ds\right) $ then {(here $\eta $ is different from the one in \eqref{Eq.rho_*_phi_*})} \begin{equation}\label{Eq.Appen_eta_asy}
		C_\eta^{-1}|t|^{\operatorname{Re}\frac{B(0;\la)}{A'(0)}}\leq|\eta(t;\la)|\leq C_\eta|t|^{\operatorname{Re}\frac{B(0;\la)}{A'(0)}},\qquad \forall\ t\in I_0\setminus\{0\},\ \forall\ \la\in\Omega
	\end{equation}
	for some constant $C_\eta>0$. We also have $\eta\in C^\infty((I_0\setminus\{0\})\times\Omega;\C\setminus\{0\})$ and 
\begin{align*}
	\frac{\eta'(t;\la)}{\eta(t;\la)}=\frac{\widetilde B(t;\la)}{A(t)}+\frac{B(0;\la)}{A'(0)t}=\frac{ B(t;\la)}{A(t)}-\frac{B(0;\la)}{A'(0)}\frac{\widetilde A(t)}{A(t)}+\frac{B(0;\la)}{A'(0)t}=\frac{ B(t;\la)}{A(t)},\quad \forall\ t\in I_0\setminus\{0\}.
	\end{align*}Here we used $\widetilde A(t)={A(t)}/{t}$.
\if0\[\eta(t;\la):=\begin{cases}\exp\left(-\int_t^{\delta_0}\frac{B(s;\la)}{A(s)}\,\mathrm ds\right), & t\in(0, \delta_0], \la\in\Omega,\\ \exp\left(\int_{-\delta_0}^t\frac{B(s;\la)}{A(s)}\,\mathrm ds\right), & t\in[-\delta_0, 0), \la\in\Omega,\quad \text{if}\ I_0=[-\delta_0,\delta_0]\end{cases}\]
	then , $\eta'(t;\la)/\eta(t;\la)=\widetilde B(t;\la)/A(t)+\frac{B(0;\la)}{A'(0)t}= B(t;\la)/A(t)+\frac{B(0;\la)}{A'(0)t}$ for $t\in I_0\setminus\{0\}, \la\in\Omega$ and\footnote{Proof of \eqref{Eq.Appen_eta_asy}: We only consider $t\in (0,\delta_0]$. Denote $\tilde B(t;\la):=tB(t;\la)/A(t)$ for $t\in(0,\delta_0], \la\in\Omega$, then \begin{align*}
			\tilde B'(t;\la)=\frac{t}{A(t)}B'(t;\la)+\frac{B(t;\la)(A(t)-tA'(t))}{A(t)^2}=\frac{t}{A(t)}B'(t;\la)+\frac{t^2B(t;\la)}{A(t)^2}\frac1{t^2}\int_0^t(A'(s)-A'(t))\,\mathrm ds,
	\end{align*} hence $\|\tilde B'\|_{L^\infty((0,\delta_0]\times\Omega)}\leq M\|B'\|_{L^\infty(I_0\times\Omega)}+M^2\|B'\|_{L^\infty(I_0\times\Omega)}\|A''\|_{L^\infty(I_0)}/2<+\infty$ (recalling the definition of $M$ in \eqref{Eq.Appen_M}). As a result, if we denote $g(t;\la):=\Re\tilde B(t;\la)$ and $g(0;\la):=\lim_{t\to0+}g(t;\la)=\Re(B(0;\la)/A'(0))$ for $t\in(0,\delta_0], \la\in\Omega$, then $\|g\|_{L^\infty([0,\delta_0]\times\Omega)}+\|g'\|_{L^\infty((0,\delta_0]\times\Omega)}<+\infty$ and \begin{align*}
	\ln|\eta(t;\la)|=-\int_t^{\delta_0}\frac{g(s;\la)}s\,\mathrm ds=-\int_t^{\delta_0}\frac{g(s;\la)-g(0;\la)}s\,\mathrm ds+g(0;\la)(\ln t-\ln\delta_0),
\end{align*} thus \begin{align*}
\left|\ln\frac{|\eta(t;\la)|}{t^{g(0;\la)}}\right|&=\left|\ln|\eta(t;\la)|-g(0;\la)\ln t\right|\leq \int_t^{\delta_0}\left|\frac{g(s;\la)-g(0;\la)}{s}\right|\,\mathrm ds+|g(0;\la)||\ln\delta_0|\\&\leq \delta_0\|g'\||_{L^\infty((0,\delta_0]\times\Omega)}+|\ln\delta_0|\|g\|_{L^\infty([0,\delta_0]\times\Omega)}
\end{align*} for all $t\in (0,\delta_0],\la\in\Omega$. This proves \eqref{Eq.Appen_eta_asy} for $t\in(0,\delta_0]$. The proof for $t\in[-\delta_0, 0)$ is along the same track.}
\fi	
	It follows from \eqref{Eq.Appen_k_derivative} that
	\begin{equation*}
		\left(A(t)^k\eta(t;\la)x^{(k+1)}(t;\la)\right)'=A(t)^{k-1}\eta(t;\la)F_k(t;\la),\qquad\forall\ t\in I_0\setminus\{0\},\ \forall\ \la\in\Omega.
	\end{equation*}
	By \eqref{Eq.Appen_k-1_derivative} for $K=k$ and \eqref{Eq.B.14}, \eqref{Eq.Appen_induc_hypo}, we have $A(t)x^{(k+1)}(t;\la)\in L^\infty((I_0\setminus\{0\})\times \Omega)$; using \eqref{Eq.Appen_eta_asy}, $|A(t)|\sim |t|$ as $t\to0$ and \begin{equation}\label{k>}k\geq N>N_0>-\inf_{\la\in\Omega}\Re\left(B(0;\la)/A'(0)\right)+1,\end{equation} we have
	\[\lim_{t\to0} A(t)^k\eta(t;\la)x^{(k+1)}(t;\la)=0,\quad \forall\ \la\in\Omega,\]
	hence
	\[A(t)^k\eta(t;\la)x^{(k+1)}(t;\la)=\int_0^tA(s)^{k-1}\eta(s;\la)F_k(s;\la)\,\mathrm ds,\qquad\forall\ t\in I_0\setminus\{0\},\ \forall\ \la\in\Omega.\]
	As a consequence, we have
	\begin{align*}
		\left|x^{(k+1)}(t;\la)\right|=\frac{\left|\int_0^tA(s)^{k-1}\eta(s;\la)F_k(s;\la)\,\mathrm ds\right|}{|A(t)|^k|\eta(t;\la)|}\leq C_{k+1}\frac{\int_0^{|t|}s^{k-1}s^{\Re(B(0;\la)/A'(0))}\,\mathrm ds}{|t|^k|t|^{\Re(B(0;\la)/A'(0))}}\leq \tilde C_{k+1}
	\end{align*}
	for all $t\in I_0\setminus\{0\}$ and $\la\in\Omega$, where $C_{k+1}>0$ and $\tilde C_{k+1}>0$ are constants. Here we have used $k+\inf_{\la\in\Omega}\Re(B(0;\la)/A'(0))>0$, which follows from \eqref{k>}. This proves \eqref{Eq.Appen_induc_2}. 
	
	Next we use once again the induction to prove that
	\begin{equation}\label{Eq.Appen_induc_3}
		x^{(k)}(0;\la) \text{ exists and } \lim_{t\to0}x^{(k)}(t;\la)=x^{(k)}(0;\la),\qquad\forall\ \la\in\Omega,\ \forall\  k\in\Z\cap[0,+\infty).
	\end{equation}
	We know from \eqref{Eq.Appen_induc_1}
that \eqref{Eq.Appen_induc_3} holds for $k\leq N$. Now we assume that for some $k\in\Z_{\geq N}$, \eqref{Eq.Appen_induc_3} holds for $0, 1,\cdots, k$. Then by \eqref{Fk}, we have $F_k(\cdot;\la)\in C(I_0)$, by \eqref{Eq.Appen_induc_2} with $k$ replaced by $k+2$ and $A(0)=0$ we have $\lim_{t\to0}A(t)x^{(k+2)}(t;\la)=0$, and by \eqref{Eq.Appen_k_derivative} we have
	\[\lim_{t\to0} x^{(k+1)}(t;\la)=\frac {F_k(0;\la)}{kA'(0)+B(0;\la)}\in\C,\qquad\forall\ \la\in\Omega,\]
	where we have used $kA'(0)+B(0;\la)\neq 0$, which follows from $k>-\inf_{\la\in\Omega}\Re(B(0;\la)/A'(0))$ (see \eqref{k>}).
	Finally, we get by L'Hôpital's rule that
	\[x^{(k+1)}(0;\la)=\lim_{t\to0}\frac{x^{(k)}(t;\la)-x^{(k)}(0)}{t}=\lim_{t\to0}x^{(k+1)}(t;\la),\qquad\forall\la\in\Omega.\]
	This proves \eqref{Eq.Appen_induc_3} for $k+1$. Then \eqref{Eq.Appen_induc_3} holds for all $k\in\Z_{\geq N}$ by the induction. Hence, $x(\cdot;\la)\in C^\infty(I_0)$ for all $\la\in\Omega$. Moreover, combining \eqref{Eq.B.14} and \eqref{Eq.Appen_induc_2} gives that
	\begin{equation}\label{Eq.Appen_x_smooth_bound}
		\pa_t^jx\in L^\infty(I_0\times\Omega),\qquad\forall\ j\in\Z_{\geq 0}.
	\end{equation}
	
{\textbf{Step 3.} $x\in C^\infty(I_0\times\Omega)$.} Recall that $x\in C(I_0\times\Omega)$ satisfies $x(t;\cdot)\in \operatorname{Hol}(\Omega)$ for all $t\in I_0$ (in {\bf Step 1}) and $x(\cdot;\la)\in C^\infty(I_0)$ for all $\la\in\Omega$ (in {\bf Step 2}). Using \eqref{Eq.Appen_x_smooth_bound} we have $x\in \mathcal H_{I_0}^0(\Omega)$. Then Lemma \ref{Lem.smooth_multi_variable} implies that $x\in C^\infty(I_0\times\Omega)$.\smallskip
	
{\textbf{Step 4.} Extension of the smooth local solution.} For any fixed $\la\in\Omega$, we have constructed a local solution $x_L(\cdot;\la)\in C^\infty(I_0)$ of \eqref{Eq.Appen_ODE} on $I_0\subset I$. Moreover, we have $x_L\in C^\infty(I_0\times\Omega)$ and $x_L(t;\cdot)\in \operatorname{Hol}(\Omega)$ for all $t\in I_0$. By standard ODE theory, the initial value problem
	\[\begin{cases}
		x''(t;\la)+\frac{B(t;\la)}{A(t)}x'(t;\la)+\frac{D(t;\la)}{A(t)}x(t;\la)=\frac{f(t;\la)}{A(t)},\\
		x(\delta_0/2;\la)=x_L(\delta_0/2;\la), x'(\delta_0/2;\la)=x_L'(\delta_0/2;\la)
	\end{cases}\]
	has a unique solution $x=x(t;\la)$ on $((0, +\infty)\cap I)\times\Omega$ and $x\in C^\infty(((0, +\infty)\cap I)\times\Omega)$. Moreover, by the analytic dependence on parameters (Lemma \ref{Lem.Appen_Analytic}), we have $x(t;\cdot)\in \operatorname{Hol}(\Omega)$ for all $t\in (0, +\infty)\cap I$.  Hence, $x_L$ can be extended to be a smooth solution of \eqref{Eq.Appen_ODE} on $((0, +\infty)\cap I)\times\Omega$; Similarly we can extend $x_L$ on the negative direction (for the case $I_0=[-\delta_0,\delta_0]$). And for the extended solution $x$, we have $x\in \mathcal H_I(\Omega)$.
\end{proof}

\begin{proposition}\label{Prop.Appen_1}
	Let $I\subset \R$ be an interval. Let $A(t)\in C^\infty(I;\C)$ be such that $A(t)=0$ has a unique solution $t=t_0$ in $I$ with $A'(t_0)\neq0$. Let $B(t;\la), D(t;\la)\in\mathcal H_I(\C)$. Assume that $B(t;\la)=\tilde B(t)+\la\hat B(t)$ for $t\in I$ and $\la\in\C$, where $\tilde B, \hat B\in C^\infty(I;\C)$. Suppose that
	\begin{equation}\label{Eq.Appen_B_nondegenerate}
		\text{either }\hat B(t_0)\neq0\quad \text{or}\quad \hat B(t_0)=0 \text{ and }-\tilde B(t_0)/A'(t_0)\notin\Z_{\geq 0}.
	\end{equation}
	We define 
	\begin{equation}
		\La_*:=\left\{\la\in\C: nA'(t_0)+B(t_0;\la)=0\text{ for some }n\in\Z_{\geq 0}\right\}.
	\end{equation}
	Then $\La_*\subset\C$ is a (probably empty) discrete set. Let $R\in(0,+\infty)$. There exists a nonzero polynomial $ \psi_1(\lambda)$ satisfying $\{\lambda\in B_R: \psi_1(\lambda)=0\}=\La_*\cap B_R$ such that for every $f(t;\la)\in\mathcal H_I(\C) $, the inhomogeneous ODE
	\begin{equation}\label{Eq.Appen_ODE1}
		\begin{cases}
			A(t)x''(t;\la)+B(t;\la)x'(t;\la)+D(t;\la)x(t;\la)=\psi_1(\lambda)f(t;\la),\quad t\in I, \la\in B_R,\\
			x(t_0;\la)=\psi_1(\lambda),\quad \la\in B_R,
		\end{cases}
	\end{equation}
	where the prime $'$ refers to the derivative with respect to $t\in I$, has a solution $x=x(t;\la)\in \mathcal H_I(B_R)$. Moreover, if $\hat B(t_0)=0$, then $ \psi_1(\lambda)=1$.
\end{proposition}
\begin{proof}
	We first show that $\La_*$ is a discrete set. If $\hat B(t_0)\neq0$, then $\La_*=\{-nA'(t_0)/\hat B(t_0)-\tilde B(t_0)/\hat B(t_0): n\in\Z_{\geq 0}\}$, hence $\La_*$ is a discrete set. If $\hat B(t_0)= 0$, then $\la_*\in\La_*$ if and only if $0=nA'(t_0)+B(t_0;\la_*)=nA'(t_0)+\tilde B(t_0)=0$ for some $n\in\Z_{\geq 0}$, which implies that $-\tilde B(t_0)/A'(t_0)\in\Z_{\geq 0}$, and this is a contradiction with our assumption \eqref{Eq.Appen_B_nondegenerate}. As a consequence, if $\hat B(t_0)=0$ (and $-\tilde B(t_0)/A'(t_0)\notin\Z_{\geq 0}$), then $\La_*=\emptyset$.
	
		Next, we construct $ \psi_1(\lambda)$. Let $N_0$ be given by Lemma \ref{Lem.Appen_1} (for $\Omega=B_R$) and fix an integer $N>\max\{N_0+1,-\inf_{\la\in B_R}\Re\left(B(t_0;\la)/A'(t_0)\right)+1\}$. Let $ \psi_1(\lambda):=1$ for the case $\hat B(t_0)=0$ and $ \psi_1(\lambda):=\prod_{j=0}^{N-1}\big(jA'(t_0)+B(t_0,\la)\big)$ 
for the case $\hat B(t_0)\neq0$.

{\bf Claim 1.} $ \psi_1(\lambda)$ is a nonzero polynomial. If $\hat B(t_0)=0$, then $ \psi_1(\la)\equiv1$ is a polynomial of degree $0$; if $\hat B(t_0)\neq0$, as $ B(t_0,\la)=\tilde B(t_0)+\la\hat B(t_0)$, then $ \psi_1(\lambda)$ is a polynomial of degree $N$. 

{\bf Claim 2}. $\{\lambda\in B_R: \psi_1(\lambda)=0\}=\La_*\cap B_R$. If $\hat B(t_0)=0$, then $\{\lambda\in B_R: \psi_1(\lambda)=0\}=\emptyset=\La_*=\La_*\cap B_R$. For the case $\hat B(t_0)\neq0$, if $\psi_1(\lambda)=0 $ then $jA'(t_0)+B(t_0,\la)=0$ for some $j\in\Z\cap[0,N-1]$ and $\la\in \La_*$, thus $\{\lambda\in B_R: \psi_1(\lambda)=0\}\subseteq\La_*\cap B_R$. On the other hand, if $\la_0\in \La_*\cap B_R$ (and $\hat B(t_0)\neq0$), then $nA'(t_0)+B(t_0;\la_0)=0$ for some $n\in\Z_{\geq0}$, and $n=-B(t_0;\la_0)/A'(t_0)\leq-\inf_{\la\in B_R}\Re\left(B(t_0;\la)/A'(t_0)\right)<N-1$, thus $n\in\Z\cap[0,N-1]$ and $\psi_1(\lambda_0)=0 $. So $\La_*\cap B_R\subseteq\{\lambda\in B_R: \psi_1(\lambda)=0\}$. 

It remains to construct $x(t;\la)$. For any $n\in \Z\cap[0, N]$ and $\la\in\C$, let $ \psi_{1,n}(\lambda):=1$ for the case $\hat B(t_0)=0$ and $ \psi_{1,n}(\lambda):=\prod_{j=n}^{N-1}\big(jA'(t_0)+B(t_0,\la)\big)$ (here $\psi_{1, N}(\la):=1$ )
for the case $\hat B(t_0)\neq0$. Then $ \psi_{1}(\lambda)=\psi_{1,0}(\lambda)$ for all $\la\in\C$.\smallskip

{\bf Claim 3.} If $n\in\Z\cap[0,N]$, $g(t;\la)\in\mathcal H_I(\C) $, $ \pa_t^ig(t_0;\la)=0$ for $i\in\Z$, $0\leq i<n$. Then 
\begin{equation}\label{eq1}
			A(t)y''(t;\la)+B(t;\la)y'(t;\la)+D(t;\la)y(t;\la)=\psi_{1,n}(\lambda)g(t;\la),\
			y(t_0;\la)=0,\ \la\in B_R,
	\end{equation}has a solution $y=y(t;\la)\in \mathcal H_I(B_R)$.

Let $g(t;\la)=f(t;\la)-D(t;\la)$, $n=0$, then by Claim 3, \eqref{eq1} has a solution $y=y(t;\la)\in \mathcal H_I(B_R)$ with $n=0$. $x(t;\la)=y(t;\la)+\psi_{1}(\lambda)\in \mathcal H_I(B_R)$ solves \eqref{Eq.Appen_ODE1} (using $ \psi_{1}(\lambda)=\psi_{1,0}(\lambda)$).\smallskip

It remains to prove Claim 3. We use the (backward) induction. We need to prove that:
\begin{enumerate}[(i)]
	\item Claim 3 holds for $n=N$; 
	\item if $j\in\Z\cap[0,N-1]$, Claim 3 holds for  $n=j+1$, then Claim 3 holds for  $n=j$.
\end{enumerate}

\underline{Proof of (i)}.  As $g\in\mathcal H_I(\C)\subset \mathcal H_{I,\text{loc}}^0(B_R)$, $n=N$, by Taylor's theorem with integral remainders, we have $g(t;\la)/|t-t_0|^{N}\in L^\infty_\text{loc}(I\times\C)$, and we also have $ \psi_{1,n}(\lambda)=\psi_{1,N}(\lambda)=1$. Then the result follows from Lemma \ref{Lem.Appen_1}.

\underline{Proof of (ii)}. We fix $j\in\Z\cap[0,N-1]$ and assume $g(t;\la)\in\mathcal H_I(\C) $, $ \pa_t^ig(t_0;\la)=0$ for $i\in\Z$, $0\leq i<j$. For $t\in I, \la\in \C$, let $x_j(t):=(t-t_0)^{j+1}$, $y_j(t;\la):=A(t)x_j''(t)+B(t;\la)x_j'(t)+D(t;\la)x_j(t)$, then $x_j\in C^{\infty}(I)$, $x_j(t_0)=0$, $y_j(t;\la)\in \mathcal H_I(\C)$, and \begin{align*}
		&y_j(t;\la)=A(t)j(j+1)(t-t_0)^{j-1}+B(t;\la)(j+1)(t-t_0)^{j}+D(t;\la)(t-t_0)^{j+1}.
	\end{align*}
	By Taylor's formula, we have $\pa_t^iy_j(t_0;\la)=0$ for $i\in\Z$, $0\leq i<j$ and 
	\begin{align*}
		\pa_t^jy_j(t_0;\la)&=j!\lim_{t\to t_0}\frac{y_j(t;\la)}{(t-t_0)^j}=j!\lim_{t\to t_0}j(j+1)\frac{A(t)}{t-t_0}+j!B(t_0;\la)(j+1)\\&=j!j(j+1)A'(t_0)+(j+1)!B(t_0;\la)=(j+1)!(jA'(t_0)+B(t_0;\la)).
	\end{align*}
	
	For the case of $\hat B(t_0)\neq0$, let $a_j=(j+1)!$, $b_j(\la)=jA'(t_0)+B(t_0;\la)$ then $a_j\neq0$, $b_j\in\operatorname{Hol}(\C)$. As $ \psi_{1,n}(\lambda)=\prod_{j=n}^{N-1}\big(jA'(t_0)+B(t_0;\la)\big)$ for $n\in\Z\cap[0,N]$ we have $ \psi_{1,j}(\lambda)=b_j(\la)\psi_{1,j+1}(\lambda)$ and 
$ \pa_t^jy_j(t_0;\la)=(j+1)!(jA'(t_0)+B(t_0;\la))=a_jb_j(\la)$ for all $\la\in\C$.

For the case of $\hat B(t_0)=0$, we have $jA'(t_0)+B(t_0,\la)=jA'(t_0)+\tilde B(t_0)\neq 0$ (using \eqref{Eq.Appen_B_nondegenerate}). Let $a_j=(j+1)!(jA'(t_0)+\tilde B(t_0))$, $b_j(\la)=1$ then $a_j\neq0$, $b_j\in\operatorname{Hol}(\C)$, $\pa_t^jy_j(t_0;\la)=a_j=a_jb_j $. As $ \psi_{1,n}(\lambda)=1$ for $n\in\Z\cap[0,N]$ we have  $\psi_{1,j}(\lambda)=b_j(\la)\psi_{1,j+1}(\lambda)$ for all $\la\in\C$.

Thus, we always have $a_j\neq0$, $b_j\in\operatorname{Hol}(\C)$,  $\psi_{1,j}(\lambda)=b_j(\la)\psi_{1,j+1}(\lambda)$, $ \pa_t^jy_j(t_0;\la)=a_jb_j(\lambda)$.

For $t\in I, \la\in\C$, let $\widetilde g(t;\la):=b_j(\la)g(t;\la)-\pa_t^jg(t_0;\la)\cdot y_j(t;\la)/a_j $ then $\widetilde g\in \mathcal H_I(\C)$,\footnote{\label{footnote.remainder}Here we use the fact that if $x\in \mathcal H_I(\Omega)$, then $x^{(n)}\in \mathcal H_I(\Omega)$ for any $n\in\Z_{\geq 0}$.} and $\pa_t^j\widetilde g(t_0;\la)=0$. As $ \pa_t^ig(t_0;\la)=0$, $\pa_t^iy_j(t_0;\la)=0$ for $i\in\Z$, $0\leq i<j$, we have $ \pa_t^i\widetilde g(t_0;\la)=0$, for $i\in\Z$, $0\leq i<j$. Thus, $ \pa_t^i\widetilde g(t_0;\la)=0$, for $i\in\Z$, $0\leq i\leq j$. 

By the induction assumption (for $n=j+1$), there exists $\widetilde y(t;\la)\in \mathcal H_I(B_R)$ such that\begin{align*}
		&A(t)\widetilde y''(t;\la)+B(t;\la)\widetilde y'(t;\la)+D(t;\la)\widetilde y(t;\la)=\psi_{1,j+1}(\lambda)\widetilde g(t;\la),\quad
			\widetilde y(t_0;\la)=0,\quad  \la\in B_R.
	\end{align*} For all $t\in I, \la\in B_R$, let $ y(t;\la):=\widetilde y(t;\la)+\psi_{1,j+1}(\lambda)\pa_t^jg(t_0;\la)\cdot x_j(t)/a_j$, then $y\in \mathcal H_I(B_R)$, $y(t_0;\la)=0 $ and \begin{align*}
		&A(t) y''(t;\la)+B(t;\la) y'(t;\la)+D(t;\la) y(t;\la)\\
&=\psi_{1,j+1}(\lambda)\widetilde g(t;\la)+\psi_{1,j+1}(\lambda)\pa_t^jg(t_0;\la)\cdot y_j(t;\la)/a_j=\psi_{1,j+1}(\lambda)b_j(\la)g(t;\la)=\psi_{1,j}(\lambda)g(t;\la),
	\end{align*}where we  have used $y_j(t;\la)=A(t)x_j''(t)+B(t;\la)x_j'(t)+D(t;\la)x_j(t)$, $\widetilde g(t;\la)+\pa_t^jg(t_0;\la)\cdot y_j(t;\la)/a_j =b_j(\la)g(t;\la)$ and $\psi_{1,j}(\lambda)=b_j(\la)\psi_{1,j+1}(\lambda)$. Thus, $y\in \mathcal H_I(B_R)$ solves \eqref{eq1} for $n=j$. This completes the proof.
\end{proof}

In the end of this appendix, we prove the analytic dependence on parameters of solutions to linear \emph{regular} ODEs. The following lemma has been used in {\bf Step 4} of the proof of Lemma \ref{Lem.Appen_1}, to show that the extended smooth solution is analytic with respect to the parameter $\la$.

\begin{lemma}\label{Lem.Appen_Analytic}
	Let $\Om\subset\C$ be an open set and $I\subset\R$. Let $p(t;\la), q(t;\la), f(t;\la)\in C^\infty(I\times\Omega;\C)$ be such that $p(t;\cdot), q(t;\cdot), f(t;\cdot)$ are analytic on $\Om$ for each $t\in I$. Let $x_0(\la), x_1(\la)$ be two analytic functions on $\Om$ and let $t_0\in I$. For each $\la\in\Om$, let $x(t;\la) (t\in I)$ be the unique smooth solution to the initial value problem
	\begin{equation*}
		x''(t;\la)+p(t;\la)x'(t;\la)+q(t;\la)x(t;\la)=f(t;\la),\quad x(t_0;\la)=x_0(\la), x'(t_0;\la)=x_1(\la),
	\end{equation*}
	where the prime $'$ refers to the derivative with respect to $t\in I$. Then for each $t\in I$, the function $\la\in\Om\mapsto x(t;\la)$ is analytic.
\end{lemma}
\begin{proof}
	By the standard ODE theory, we know that $x\in C^\infty(I\times \Om)$.
	For any complex function $\varphi=\varphi(\la):\C\to\C$ of class $C^1$ seen as a function on $\R^2$, we can define the Wirtinger derivatives
	\[\pa_{\bar\la}\varphi(\la)=\frac12\big(\pa_1\varphi(\la)+\ii\pa_2\varphi(\la)\big),\quad \pa_{\la}\varphi(\la)=\frac12\big(\pa_1\varphi(\la)-\ii\pa_2\varphi(\la)\big).\]
	Now it suffices to show that $\pa_{\bar\la}x(t;\la)=0$ for all $(t, \la)\in I\times\Om$. Since $x\in C^\infty(I\times\Om)$, the derivative with respect to $t$ and $\pa_{\bar\la}$ are commutable. By the analyticity of coefficients and the initial data, we know that $\pa_{\bar\la}x$ satisfies
	\[(\pa_{\bar\la}x)''(t;\la)+p(t;\la)(\pa_{\bar\la}x)'(t;\la)+q(t;\la)\pa_{\bar\la}x(t;\la)=0,\quad \pa_{\bar\la}x(t_0;\la)=(\pa_{\bar\la}x)'(t_0;\la)=0.\]
	By the uniqueness, we have $\pa_{\bar\la}x(t;\la)=0$ for all $(t, \la)\in I\times\Om$. 
\end{proof}

\section*{Acknowledgments}
D. Wei is partially supported by the National Key R\&D Program of China under the
grant 2021YFA1001500.	Z. Zhang is partially supported by  NSF of China  under Grant 12288101.


\begin{thebibliography}{99}
	
	
	
	\bibitem{Alinhac} S. Alinhac, \textit{Blowup for nonlinear hyperbolic equations}.	Progr. Nonlinear Differential Equations Appl. \textbf{17}, Birkhäuser Boston, Inc., Boston, MA, 1995. xiv+113 pp.
	
	\bibitem{Bourgain2000} J. Bourgain, Problems in Hamitonian PDE's. \textit{Gemo. Funct. Anal.}, Special Volume, Part I (2000), 32–56.
	
	\bibitem{Brenner-Wahl} P. Brenner and W. von Wahl, Global classical solutions of nonlinear wave equations. \textit{Math. Z.}, \textbf{176} (1981), 87–121.
	
	\bibitem{Bulut2012} A. Bulut, Global well-posedness and scattering for the defocusing energy-supercritical cubic nonlinear wave equation. \textit{J. Funct. Anal.}, \textbf{263} (2012), 1609–1660.
	
	\bibitem{Bulut2015} A. Bulut, The defocusing energy-supercritical cubic nonlinear wave equation in dimension five. \textit{Trans. Amer. Math. Soc.}, \textbf{367} (2015),  6017–6061.
	
	\bibitem{Christ-Colliander-Tao} M. Christ, J. Colliander and T.Tao, Ill-posedness for nonlinear Schrodinger and wave equations. \textit{arXiv:math/0311048}, 2003.
	
	\bibitem{Collot-Raphael-Szeftel2019} C. Collot, P. Raphaël and J. Szeftel, On the stability of type I blow up for the energy super critical heat equation. \textit{Mem. Amer. Math. Soc.}, \textbf{260} (2019), no. 1255, v+97 pp.
	
	\bibitem{Collot-Raphael-Szeftel2020} C. Collot, P. Raphaël and J. Szeftel, Strongly anisotropic type II blow up at an isolated point. \textit{J. Amer. Math. Soc.}, \textbf{33} (2020), 527–607.
	
	\bibitem{Cortazar-3-2020} C. Cortázar, M. del Pino and M. Musso, Green's function and infinite-time bubbling in the critical nonlinear heat equation. \textit{J. Eur. Math. Soc. (JEMS)}, \textbf{22} (2020), 283–344.
	
	\bibitem{Davila-del-Wei2020} J. Dávila, M. del Pino and J. Wei,  Singularity formation for the two-dimensional harmonic map flow into $S^2$. \textit{Invent. Math.}, \textbf{219} (2020), 345–466.
	
	\bibitem{del-Pino-3-2020} M. del Pino, M. Musso and J. Wei,  Infinite-time blow-up for the 3-dimensional energy-critical heat equation. \textit{Anal. PDE}, \textbf{13} (2020), 215–274.
	
	\bibitem{Donninger2017} R. Donninger, Strichartz estimates in similarity coordinates and stable blowup for the critical wave equation. \textit{Duke Math. J.}, \textbf{166} (2017), 1627–1683.
	
	\bibitem{Duyckaerts-Jia-Kenig-Merle2017} T. Duyckaerts, H. Jia, C. E. Kenig and F. Merle, Soliton resolution along a sequence of times for the focusing energy critical wave equation. \textit{Geom. Funct. Anal.}, \textbf{27} (2017), 798–862.
	
	\bibitem{Duyckaerts-Jia-Kenig-Merle2018} T. Duyckaerts, H. Jia, C. E. Kenig and F. Merle, Universality of blow up profile for small blow up solutions to the energy critical wave map equation. \textit{Int. Math. Res. Not. IMRN} (2018), 6961–7025.
	
	\bibitem{Duyckaerts-Kenig-Merle2012} T. Duyckaerts, C. E. Kenig and F. Merle, Profiles of bounded radial solutions of the focusing, energy-critical wave equation. \textit{Geom. Funct. Anal.}, \textbf{22} (2012),  639–698.
	
	\bibitem{Duyckaerts-Kenig-Merle2013} T. Duyckaerts, C. E. Kenig and F. Merle, Classification of radial solutions of the focusing, energy-critical wave equation. \textit{Camb. J. Math.}, \textbf{1} (2013), 75–144.
	
	\bibitem{Duyckaerts-Kenig-Merle2023} T. Duyckaerts, C. E. Kenig and F. Merle, Soliton resolution for the radial critical wave equation in all odd space dimensions. \textit{Acta Math.}, \textbf{230} (2023), 1–92.
	
	\bibitem{Duyckaerts-Yang2018} T. Duyckaerts and J. Yang, Blow-up of a critical Sobolev norm for energy-subcritical and energy-supercritical wave equations. \textit{Anal. PDE}, \textbf{11} (2018),  983–1028.
	
	\bibitem{Ginibre-Velo1} J. Ginibre and G. Velo, The global Cauchy problem for the nonlinear Klein-Gordon equation. \textit{Math. Z.}, \textbf{189} (1985), 487–505.
	
	\bibitem{Ginibre-Velo2} J. Ginibre and G. Velo, The global Cauchy problem for the nonlinear Klein-Gordon equation. II. \textit{Ann. Inst. H. Poincaré Anal. Non Linéaire}, \textbf{6} (1989), 15–35.
	
	\bibitem{Grillakis1990} M. Grillakis, Regularity and asymptotic behaviour of the wave equation with a critical nonlinearity. \textit{Ann. of Math.}, \textbf{132} (1990), 485-509. 
	
	\bibitem{Grillakis1992} M. Grillakis, Regularity for the wave equation with a critical nonlinearity. \textit{Comm. Pure Appl. Math.}, \textbf{45} (1992), 749-774.
	
	\bibitem{Harada2020} J. Harada, A type II blowup for the six dimensional energy critical heat equation. \textit{Ann. PDE}, \textbf{6} (2020), Paper No. 13, 63 pp.
	
	\bibitem{Hardy} M. Hardy, Combinatorics of partial derivatives. \textit{Electron. J. Combin.}, \textbf{13} (2006), Research Paper 1, 13 pp.
	
	\bibitem{Hormander} L. Hörmander, \textit{Lectures on nonlinear hyperbolic differential equations.} Math. Appl. (Berlin) \textbf{26}, Springer-Verlag, New York, 1997.
	
	\bibitem{Ibrahim-Majdoub-Masmoudi} S. Ibrahim, M. Majdoub and N. Masmoudi, Well- and ill-posedness issues for energy supercritical waves. \textit{Anal. PDE}, \textbf{4} (2011), 341-367.
	
	\bibitem{Jendrej2017} J. Jendrej, Construction of type II blow-up solutions for the energy-critical wave equation in dimension 5. \textit{J. Funct. Anal.}, \textbf{272} (2017),  866–917.
	
	\bibitem{Jendrej-Lawrie2023CVPDE} J. Jendrej and A. Lawrie, Bubble decomposition for the harmonic map heat flow in the equivariant case. \textit{Calc. Var. Partial Differential Equations}, \textbf{62} (2023), Paper No. 264, 36 pp.
	
	\bibitem{Jendrej-Lawrie2023} J. Jendrej and A. Lawrie, Soliton resolution for the energy-critical nonlinear wave equation in the radial case. \textit{Ann. PDE}, \textbf{9} (2023), Paper No. 18, 117 pp.
	
	
	\bibitem{John} F. John, Blow-up of solutions of nonlinear wave equations in three space dimensions. \textit{Manuscripta Math.}, \textbf{28} (1979), 235–268.
	
	\bibitem{Jorgens} K. J\"{o}rgens, Das Anfangswertproblem im Grossen für eine Klasse nichtlinearer Wellengleichungen. \textit{Math. Z.}, \textbf{77} (1961), 295–308.
	
	\bibitem{Kenig2015} C. E. Kenig, \textit{Lectures on the energy critical nonlinear wave equation.} CBMS Reg. Conf. Ser. Math., \textbf{122}. Published for the Conference Board of the Mathematical Sciences, Washington, DC; by the American Mathematical Society, Providence, RI, 2015. xiv+161 pp.
	
	\bibitem{Kenig-Merle2006} C. E. Kenig and F. Merle, Global well-posedness, scattering and blow-up for the energy-critical, focusing, non-linear Schrödinger equation in the radial case. \textit{Invent. Math.}, \textbf{166} (2006), 645–675.
	
	\bibitem{Kenig-Merle2008} C. E. Kenig and F. Merle, Global well-posedness, scattering and blow-up for the energy-critical focusing non-linear wave equation. \textit{Acta Math.}, \textbf{201} (2008), 147–212.
	
	\bibitem{Killp-Visan2011Trans} R. Killip and M. Visan, The defocusing energy-supercritical nonlinear wave equation in three space dimensions. \textit{Trans. Amer. Math. Soc.}, \textbf{363} (2011),  3893–3934.
	
	\bibitem{Killp-Visan2011Pro} R. Killip and M. Visan, The radial defocusing energy-supercritical nonlinear wave equation in all space dimensions. \textit{Proc. Amer. Math. Soc.}, \textbf{139} (2011), 1805–1817.
	
	\bibitem{Kim-Merle2024} K. Kim and F. Merle, On classification of global dynamics for energy-critical equivariant harmonic map heat flows and radial nonlinear heat equation. \textit{arXiv:2404.04247}, 2024.
	
	\bibitem{Krieger-Miao2020} J. Krieger and S. Miao, On the stability of blowup solutions for the critical corotational wave-map problem. \textit{Duke Math. J.}, \textbf{169} (2020), 435–532.
	
	\bibitem{Krieger-Miao-Schlag2020} J. Krieger, S. Miao and W. Schlag, A stability theory beyond the co-rotational setting for critical Wave Maps blow up. arXiv:2009.08843, 2020.
	
	\bibitem{Krieger-Schlag2014} J. Krieger and W. Schlag, Full range of blow up exponents for the quintic wave equation in three dimensions. \textit{J. Math. Pures Appl. (9)}, \textbf{101} (2014),  873–900.
	
	\bibitem{Krieger-Schlag-Tataru2008} J. Krieger, W. Schlag and D. Tataru, Renormalization and blow up for charge one equivariant critical wave maps. \textit{Invent. Math.}, \textbf{171} (2008), 543–615.
	
	\bibitem{Krieger-Schlag-Tataru2009} J. Krieger, W. Schlag and D. Tataru, Slow blow-up solutions for the $H^1(\R^3)$ critical focusing semilinear wave equation. \textit{Duke Math. J.}, \textbf{147} (2009), 1–53.
	
	\bibitem{Levine} H. A. Levine, Instability and nonexistence of global solutions to nonlinear wave equations of the form $Pu_{tt}=-Au+\mathcal F(u)$. \textit{Trans. Amer. Math. Soc.}, \textbf{192} (1974), 1–21.
	
	
	\bibitem{Lindblad-Sogge} H. Lindblad and C. D. Sogge, Long-time existence for small amplitude semilinear wave equations. 
	\textit{Amer. J. Math.}, \textbf{118} (1996), 1047-1135.
	
	\bibitem{Luk} J. Luk, \textit{Introduction to nonlinear wave equations}. \url{https://web.stanford.edu/~jluk/NWnotes.pdf}
	
	\bibitem{Martel-Merle2018} Y. Martel, Yvan and F. Merle,  Inelasticity of soliton collisions for the 5D energy critical wave equation. 
	\textit{Invent. Math.}, \textbf{214} (2018), 1267–1363.
	
	\bibitem{Matano-Merle2004} H. Matano and F. Merle, On nonexistence of type II blowup for a supercritical nonlinear heat equation. 
	\textit{Comm. Pure Appl. Math.}, \textbf{57} (2004), 1494–1541.
	
	\bibitem{Merle-Raphael2006} F. Merle and P. Rapha\"{e}l, On a sharp lower bound on the blow-up rate for the $L^2$ critical nonlinear Schrödinger equation. \textit{J. Amer. Math. Soc.}, \textbf{19} (2006), 37–90.
	
	\bibitem{Merle-Raphael-Rodnianski2013} F. Merle, P. Rapha\"{e}l and I. Rodnianski, Blowup dynamics for smooth data equivariant solutions to the critical Schrödinger map problem. \textit{Invent. Math.}, \textbf{193} (2013),  249–365.
	
	\bibitem{Merle-Raphael-Rodnianski2015} F. Merle, P. Rapha\"{e}l and I. Rodnianski, Type II blow up for the energy supercritical NLS. 
	\textit{Camb. J. Math.}, \textbf{3} (2015), 439–617.
	
	\bibitem{Merle-4-Invent} F. Merle, P. Rapha\"{e}l, I. Rodnianski and J. Szeftel, On blow up for the energy super critical defocusing nonlinear Schrödinger equations. \textit{Invent. Math.}, \textbf{227} (2022), 247-413.
	
	\bibitem{Merle-4-Ann1} F. Merle, P. Raphaël, I. Rodnianski and J. Szeftel, On the implosion of a compressible fluid I: Smooth self-similar inviscid profiles. \textit{Ann. of Math. (2)}, \textbf{196} (2022), 567-778.
	
	\bibitem{Merle-4-Ann2} F. Merle, P. Raphaël, I. Rodnianski and J. Szeftel, On the implosion of a compressible fluid II: Singularity formation. \textit{Ann. of Math. (2)}, \textbf{196} (2022), 779-889.
	
	\bibitem{Merle-Raphael-Szeftel-2010} F. Merle, P. Raphaël and J. Szeftel, Stable self-similar blow-up dynamics for slightly $L^2$ super-critical NLS equations. \textit{Geom. Funct. Anal.}, \textbf{20} (2010), 1028–1071.
	
	\bibitem{Merle-Raphael-Szeftel-2014} F. Merle, P. Raphaël and J. Szeftel, On collapsing ring blow-up solutions to the mass supercritical nonlinear Schrödinger equation. \textit{Duke Math. J.}, \textbf{163} (2014), 369–431.
	
	\bibitem{Merle-3-2020} F. Merle, P. Raphaël and J. Szeftel, On strongly anisotropic type I blowup. 
	\textit{Int. Math. Res. Not. IMRN} (2020),  541–606.
	
	\bibitem{Merle-Zaag1997} F. Merle and H. Zaag, Stability of the blow-up profile for equations of the type $u_t=\Delta u+|u|^{p-1}u$. 
	\textit{Duke Math. J.}, \textbf{86} (1997), 143–195.
	
	\bibitem{Merle-Zaag2003} F. Merle and H. Zaag, Determination of the blow-up rate for the semilinear wave equation. 
	\textit{Amer. J. Math.}, \textbf{125} (2003), 1147–1164.
	
	\bibitem{Perelman2001} G. Perelman, On the formation of singularities in solutions of the critical nonlinear Schrödinger equation. 
	\textit{Ann. Henri Poincaré}, \textbf{2} (2001), 605–673.
	
	\bibitem{Perelman2014} G. Perelman, Blow up dynamics for equivariant critical Schrödinger maps. \textit{Comm. Math. Phys.}, \textbf{330} (2014), 69–105.
	
	\bibitem{Perelman2023} G. Perelman, Formation of singularities in nonlinear dispersive PDEs. \textit{ICM—International Congress of Mathematicians. Vol. 5. Sections 9–11}, 3854–3879. EMS Press, Berlin, 2023.
	
	\bibitem{Raphael-Rodnianski2021} P. Raphaël and I. Rodnianski, Stable blow up dynamics for the critical co-rotational wave maps and equivariant Yang-Mills problems. \textit{Publ. Math. Inst. Hautes Études Sci.}, \textbf{115} (2012), 1–122.
	
	\bibitem{Raphael-Schweyer2013} P. Raphaël and R. Schweyer, Stable blowup dynamics for the 1-corotational energy critical harmonic heat flow. \textit{Comm. Pure Appl. Math.}, \textbf{66} (2013),  414–480.
	
	\bibitem{Rodnianski-Sterbenz2010} I. Rodnianski and J. Sterbenz, On the formation of singularities in the critical $O(3)$ $\sigma$-model. 
	\textit{Ann. of Math. (2)}, \textbf{72} (2010), 187–242.
	
	\bibitem{Shao-Wei-Zhang} F. Shao, D. Wei and Z. Zhang, Self-similar imploding solutions of the relativistic Euler equations. arXiv: 2403.11471, 2024.
	
	\bibitem{Shatah-Struwe1993} J. Shatah and M. Struwe, Regularity results for nonlinear wave equations. 
	\textit{Ann. of Math.}, \textbf{138} (1993), 503-518. 
	
	\bibitem{Shatah-Struwe1994} J. Shatah and M. Struwe, Well-posedness in the energy space for semilinear wave equations with critical growth. \textit{Int. Math. Res. Not.}, \textbf{1994} (1994), 303-309.
	
	
	\bibitem{Sogge} C. Sogge, \textit{Lectures on Nonlinear Wave Equations}. 
	Monographs in Analysis, II. International Press, Boston, MA, 1995.
	
	\bibitem{Stein} E. M. Stein and R. Shakarchi, \textit{Complex analysis}. Princeton Lectures in Analysis, vol. \textbf{2}. Princeton University Press, Princeton, NJ, 2003. xviii+379 pp.
	
	\bibitem{Strauss} W. A. Strauss, \textit{Nonlinear wave equations}. CBMS Regional Conference Series in Mathematics \textbf{73}, American Mathematical Society, Providence, RI, 1989. 
	
	\bibitem{Struwe1988} M. Struwe, Globally regular solutions to the $u^5$ Klein-Gordon equation. 
	\textit{Ann. Scuola Norm. Sup. Pisa Cl. Sci.}, \textbf{15} (1988), 495-513.
	
%
	\bibitem{Tao2007} T. Tao, Global regularity for a logarithmically supercritical defocusing nonlinear wave equation for spherically symmetric data. \textit{J. Hyperbolic Differ. Equ.}, \textbf{4} (2007), 259–265.
	
	\bibitem{Tao2016} T. Tao, Finite-time blowup for a supercritical defocusing nonlinear wave system. 
	\textit{Anal. PDE}, \textbf{9} (2016), 1999–2030.
	
	\bibitem{Tao2018} T. Tao, Finite time blowup for a supercritical defocusing nonlinear Schr\"{o}dinger system. 
	\textit{Anal. PDE}, \textbf{11} (2018),  383–438.
	
	
	\bibitem{Yang} S. Yang, Global behaviors of defocusing semilinear wave equations. 
	\textit{Ann. Sci. Éc. Norm. Supér. (4)}, \textbf{55} (2022), 405–428.
	
	
\end{thebibliography}
\end{document}